\documentclass[a4paper]{amsart}

\usepackage{amssymb,latexsym}
\usepackage{amsthm}
\usepackage{amsmath}
\usepackage{amsfonts}
\usepackage{mathrsfs}
\usepackage{mathdots}
\usepackage{graphicx}
\usepackage[all]{xy}
\usepackage{chngcntr}
\usepackage{fancyhdr}
\usepackage[top=1in,bottom=1in,left=1.25in,right=1.25in]{geometry}

\usepackage{todonotes}
\usepackage{tikz-cd}
\usepackage{stmaryrd}

\newtheorem{theorem}{Theorem}[section]

\newtheorem{remark}[theorem]{Remark}

\newtheorem{proposition}[theorem]{Proposition}
\newtheorem{corollary}[theorem]{Corollary}
\newtheorem{lemma}[theorem]{Lemma}
\newtheorem{example}[theorem]{Example}

\def\Q{\mathbb{Q}}
\def\F{\mathbb{F}}

\def\Z{\mathbb{Z}}
\def\A{\mathbb{A}}
\def\H{\mathcal{H}}
\def\C{\mathbb{C}}
\def\G{\mathbb{G}}

\def\D{\mathcal{D}}
\def\E{\mathcal{E}}

\def\Fc{\mathcal{F}}
\def\Fl{\mathscr{F}\ell}

\def\M{\mathcal{M}}
\def\N{\mathbb{N}}
\def\Nc{\mathcal{N}}
\def\Ol{\mathcal{O}}

\def\Bun{\mathrm{Bun}}

\def\Fil{\mathrm{Fil}}
\def\Hom{\mathrm{Hom}}
\def\Lie{\mathrm{Lie}}

\def\Mod{\mathrm{Mod}}
\def\Modif{\mathrm{Modif}}

\def\GL{\mathrm{GL}}
\def\Gal{\mathrm{Gal}}
\def\Perf{\mathrm{Perf}}

\def\Rep{\mathrm{Rep}}
\def\Spd{\mathrm{Spd}}
\def\Gr{\mathrm{Gr}}
\def\Sht{\mathrm{Sht}}
\def\Sh{\mathrm{Sh}}
\def\Eu{\underline{\mathcal{E}}}
\def\K{\mathbf{K}}

\def\pr{\mathrm{pr}}
\def\Spa{\mathrm{Spa}}

\def\Spec{\mathrm{Spec}}
\def\Vect{\mathrm{Vect}}
\def\deg{\mathrm{deg}}
\def\rank{\mathrm{rank}}
\def\HT{\mathrm{HT}}
\def\SL{\mathrm{SL}}

\def\phmod{\varphi\mathrm{-Mod}}

\def\lan{\langle}
\def\ran{\rangle}

\def\lra{\longrightarrow}
\def\ra{\rightarrow}
\def\ov{\overline}
\def\ul{\underline}
\def\wh{\widehat}
\def\wt{\widetilde}
\def\st{\stackrel}
\def\tr{\textrm}

\usepackage[pdftex]{hyperref}

\begin{document}

\title{Harder-Narasimhan strata and $p$-adic period domains}
\author{Xu Shen}
\date{}

\address{Morningside Center of Mathematics, Academy of Mathematics and Systems Science, Chinese Academy of Sciences\\
	No. 55, Zhongguancun East Road\\
	Beijing 100190, China}
\address{University of Chinese Academy of Sciences\\Beijing 100049, China}
\email{shen@math.ac.cn}

\renewcommand\thefootnote{}

\renewcommand{\thefootnote}{\arabic{footnote}}

\begin{abstract}
We revisit the Harder-Narasimhan stratification on a minuscule $p$-adic flag variety, by the theory of modifications of $G$-bundles on the Fargues-Fontaine curve. We compare the Harder-Narasimhan strata with the Newton strata introduced by Caraiani-Scholze. As a consequence, we get further equivalent conditions in terms of $p$-adic Hodge-Tate period domains for fully Hodge-Newton decomposable pairs. Moreover, we generalize these results to arbitrary cocharacters case by considering the associated $B_{dR}^+$-affine Schubert varieties. Applying Hodge-Tate period maps,  our constructions give applications to $p$-adic geometry of Shimura varieties and their local analogues. 
	
\end{abstract}

\maketitle
\setcounter{tocdepth}{1}
\tableofcontents

\section{Introduction}
This paper is a continuation and complement of our previous work \cite{CFS}. We look at ``$p$-adic period domains'' from a different perspective (we refer to \cite{CFS} and the references therein for more background on $p$-adic period domains). We also extend the main result of \cite{CFS} to general (not necessarily minuscule) cocharacters.
\\

More precisely, we revisit the Harder-Narasimhan stratifications on $p$-adic flag varieties, which were defined using the theory of filtered vector spaces with additional structures by Rapoport \cite{R}, and Dat-Orlik-Rapoport \cite{DOR} Parts 1 and 2. In fact, in \cite{R} only the maximal open strata were considered, while in  \cite{DOR}  Parts 1 and 2 these Harder-Narasimhan stratifications were mainly investigated for reductive groups over finite fields. In this paper, we are interested in the $p$-adic setting, motivated by the work of Fargues \cite{F2} in the context of Harder-Narasimhan polygons for $p$-divisible groups.  The pure linear algebra context here suggests that it should be easier to access than the usual context of filtered isocrystals with additional structures as \cite{Ra, RZ} and \cite{DOR} Part 3.
Under base change, filtered vector spaces can be viewed as filtered isocrystals with trivial underlying isocrystals.  Thus we can study these $p$-adic Harder-Narasimhan strata by plugging them into the setting of Rapoport-Zink \cite{RZ} chapter 1 and Dat-Orlik-Rapoport \cite{DOR} Part 3, where the theory of filtered isocrystals with additional structures serves as the basic tool. In a different direction, the open Harder-Narasimhan strata were also defined and studied in certain cases by van der Put and Voskuil in \cite{van}.

Thanks to the recent developments in $p$-adic Hodge theory \cite{FF, SW},  
now we can apply the theory of modifications of $G$-bundles on the Fargues-Fontaine curve to study the Harder-Narasimhan strata (in \emph{minuscule} $p$-adic flag varieties). This new method has the advantage that it is easier and natural to compare the Harder-Narasimhan stratification with some other important stratifications on (minuscule) $p$-adic flag varieties, such as the Newton stratification introduced by Caraiani-Scholze in \cite{CS} section 3, where the Fargues-Fontaine curve also plays the key role. 
From the point of view of period morphisms of local Shimura varieties, we consider these Harder-Narasimhan and Newton stratifications as constructions on the \emph{Hodge-Tate side}. The purpose of this paper is to understand the relation between these two stratifications. 
In our previous work \cite{CFS},  we studied the Harder-Narasimhan strata and Newton strata on the \emph{de Rham side} (although we mostly restricted to the open strata: the weakly admissible locus and the admissible locus). At the end, we will see the theories on both sides are very closely related, in the sense they are actually \emph{dual} to each other. Therefore, we make the main result of \cite{CFS} into a more symmetric form here. Moreover, we can extend these results to general (arbitrary cocharacter) case, replacing $p$-adic flag varieties by the $B_{dR}^+$-affine Schubert cells constructed in \cite{SW}, which are key examples of \emph{diamonds}. 
\\

To be more precise, let us fix some notations. Let $G$ be a reductive group over\footnote{Throughout this paper, the base field for our reductive groups is $\Q_p$. However one can replace it by any finite extension of $\Q_p$ and all the results are still true.} $\Q_p$ and $\{\mu\}$ a conjugacy class of cocharacters $\mu: \G_{m,\ov{\Q}_p}\ra G_{\ov{\Q}_p}$. Attached to $(G,\{\mu\})$, we have flag varieties $\Fl(G,\mu)$ and $\Fl(G,\mu^{-1})$, parametrizing ``$G$-filtrations'' of type $\mu$ and $\mu^{-1}$ respectively, defined over a finite extension $E$ of $\Q_p$. We view them as \emph{adic spaces} over $\breve{E}$, the completion of the maximal unramified extension of $E$.  We assume that $\mu$ is \emph{minuscule} at this moment for simplicity. 
\\

Consider the $p$-adic flag variety $\Fl(G,\mu^{-1})$ first. By studying modifications of the trivial $G$-bundle over the Fargues-Fontaine curve, we can introduce two stratifications as follows. The first one is the Newton stratification introduced by Caraiani-Scholze in \cite{CS} section 3. Let $C|\breve{E}$ be any algebraically closed perfectoid field and $X=X_{C^\flat}$ be the Fargues-Fontaine curve over $\Q_p$ attached to the tilt $C^\flat$ equipped with a closed point $\infty$ with residue field $C$. A very basic construction (originally due to Beauville-Laszlo) is that, to each point $x\in \Fl(G,\mu^{-1})(C)$, we can attach a modified $G$-bundle at $\infty$ \[\E_{1,x}\] of the trivial $G$-bundle $\E_1$  over $X$. Recall that by Fargues's main theorem in \cite{F3}, we have a bijection $B(G)\st{\sim}{\lra} H_{\tr{et}}^1(X,G),\quad [b']\mapsto [\E_{b'}]$, where $B(G)$ is the set of $\sigma$-conjugacy classes in $G(\breve{\Q}_p)$, cf. \cite{Kot1, Kot2}.
The isomorphism class of $\E_{1,x}$ thus defines a point in $B(G)$, which in fact lies in the Kottwitz set $B(G,\mu)$ (\cite{Kot2} section 6). This gives the Newton stratification \[ \Fl(G,\mu^{-1})=\coprod_{[b']\in B(G,\mu)}\Fl(G,\mu^{-1})^{Newt=[b']}.\] Each stratum $\Fl(G,\mu^{-1})^{Newt=[b']}$ is a locally closed subspace of $\Fl(G,\mu^{-1})$, therefore we can either view it as a pseudo-adic space in the sense of Huber (\cite{Hub}) or a diamond in the sense of Scholze (\cite{S}).

On the other hand, we can define the Harder-Narasimhan vector (see subsections \ref{subsection adm modif} and \ref{subsection G-str}) \[\nu(\E_1,\E_{1,x},f)\] attached to the modification triple $(\E_1,\E_{1,x},f)$, which is an element in the set $\Nc(G)$ of \cite{RR} 1.7 attached to $G$. In the case of $\GL_n$ this has been studied by Fargues \cite{F2} and Cornut-Irissarry \cite{CI}.  It turns out that under the above \emph{minuscule} condition, the vector $\nu(\E_1,\E_{1,x},f)$ is identical to the Harder-Narasimhan vector $\nu(\Fc_x)$ defined in \cite{DOR} chapter VI.3 for the ``$G$-filtration'' $\Fc_x$ attached to $x$. 
We can show that in fact $\nu(\E_1,\E_{1,x},f)$ (in fact some normalization of it) lies in $\Nc(G,\mu)$, the image of $B(G,\mu)$ under the Newton map $\nu: B(G)\ra \Nc(G)$ (cf. \cite{Kot1} section 4). In this way we get the Harder-Narasimhan stratification \[ \Fl(G,\mu^{-1})=\coprod_{[b']\in B(G,\mu)}\Fl(G,\mu^{-1})^{HN=[b']}.\]
Similarly as above, each Harder-Narasimhan stratum is a locally closed subspace of $\Fl(G,\mu^{-1})$.
For both stratifications, the maximal open strata are indexed by the \emph{basic} element $[b]\in B(G,\mu)$ and we have an inclusion
\[\Fl(G,\mu^{-1})^{Newt=[b]}\subset \Fl(G,\mu^{-1})^{HN=[b]},\]
which comes from the inequality between the Harder-Narasimhan vector and the Newton vector in $\Nc(G)$ \[\nu(\E_1,\E_{1,x},f)\leq \nu(\E_{1,x}),\]
cf. Proposition \ref{P: HN and Newt}. See also \cite{F2} Proposition 14 and \cite{CI} Proposition 3.5 (both in the case of $\GL_n$).
\\

Now consider the side of $\Fl(G,\mu)$. Let $b\in G(\breve{\Q}_p)$ be such that the associated class $[b]\in B(G,\mu)$ and it is the \emph{basic} element. The triple $(G,\{\mu\},[b])$ then forms a basic local Shimura datum (\cite{RV}). 
Then we can also define the Newton stratification and Harder-Narasimhan stratification on $\Fl(G,\mu)$ by considering the modifications \[\E_{b,x}, \quad x\in \Fl(G,\mu)(C)\] of the $G$-bundle $\E_b$ associated to $b$ on $X$ (by Fargues's main theorem in \cite{F3}) similarly as above. The Newton stratification\footnote{In \cite{CFS} this was called the Harder-Narasimhan stratification, which should not be confused with the Harder-Narasimhan stratification here.} in this setting was introduced in \cite{CFS} 5.3, while the Harder-Narasimhan stratification was introduced in \cite{DOR} chapter IX.6, where the more classical theory of filtered isocrystals with additional structures was used. 
The open Newton stratum is the admissible locus $\Fl(G,\mu,b)^a$ (\cite{SW1, SW, R, CFS}), while the open Harder-Narasimhan stratum is the weakly admissible locus $\Fl(G,\mu,b)^{wa}$ (\cite{RZ, DOR}). By the theorem of Colmez-Fontaine (see \cite{FF} chapter 10),
we have also the inclusion
\[ \Fl(G,\mu,b)^a\subset \Fl(G,\mu,b)^{wa}.\]
The Newton and Harder-Narasimhan stratifications on the side of $\Fl(G,\mu)$ also have the same index set, $B(G,0, \nu_b\mu^{-1})$, a generalized Kottwitz set which was introduced in \cite{R} and \cite{CFS} section 4.
\\

To summarize, we have
the open strata $\Fl(G,\mu^{-1})^{Newt=[b]}$ and $\Fl(G,\mu^{-1})^{HN=[b]}$ inside $\Fl(G,\mu^{-1})$ constructed starting from the triple $(G,\{\mu^{-1}\},[1])$ (the Hodge-Tate side), and the open strata $\Fl(G,\mu,b)^a$ and $\Fl(G,\mu,b)^{wa}$ inside $\Fl(G,\mu)$ constructed from the local Shimura datum $(G,\{\mu\},[b])$ (the de Rham side). These open strata are related by the following diagram
\[\xymatrix{ & \M(G,\mu,b)_\infty\ar@{->>}[ld]_{\pi_{dR}}\ar@{->>}[rd]^{\pi_{HT}}&\\
	\Fl(G,\mu,b)^{a,\Diamond}\ar@{^{(}->}[d] & &\Fl(G,\mu^{-1})^{Newt=[b],\Diamond}\ar@{^{(}->}[d]\\
	\Fl(G,\mu,b)^{wa,\Diamond}	& & \Fl(G,\mu^{-1})^{HN=[b],\Diamond},
}\]
where $\M(G,\mu,b)_\infty$ is  the local Shimura variety with infinite level attached to the datum $(G,\{\mu\},[b])$ (cf. \cite{CFS} Theorem 3.3 and \cite{SW} sections 23 and 24), $\pi_{dR}$ and $\pi_{HT}$ are the $p$-adic de Rham and Hodge-Tate period morphisms respectively. Thus it is more reasonable to call $\Fl(G,\mu,b)^{a}$ and $\Fl(G,\mu^{-1})^{Newt=[b]}$ \emph{$p$-adic period domains}, although historically in \cite{R, DOR} it was the open Harder-Narasimhan strata $\Fl(G,\mu,b)^{wa}$ and $\Fl(G,\mu^{-1})^{HN=[b]}$ that were called period domains.
By construction, $\M(G,\mu,b)_\infty$  is a diamond  over $\breve{E}$. This is why we pass to the diamonds associated to the above spaces. 
\\

Recall that G\"ortz-He-Nie have introduced the notion of \emph{fully Hodge-Newton decomposability} for the Kottwitz set $B(G,\mu)$ (or the pair $(G,\{\mu\})$, cf. \cite{GHN} Definition 2.1, where $\mu$ is a not necessarily minuscule cocharater). Roughly, this condition means that for any \emph{non basic} element $[b']\in B(G,\mu)$, the pair $([b'], \{\mu\})$ satisfies the \emph{Hodge-Newton condition}. By \cite{GHN} Theorem 2.5 we have a complete classification for fully Hodge-Newton decomposable pairs $(G,\{\mu\})$. Now we have the following theorem.
\begin{theorem}[Theorem \ref{T: fully HN}]\label{T1}
Assume that $\mu$ is minuscule and $[b]\in B(G,\mu)$ is basic.
The following statements are equivalent:
\begin{enumerate}
	\item $B(G,\mu)$ is fully Hodge-Newton decomposable,
	\item $\Fl(G,\mu,b)^a=\Fl(G,\mu,b)^{wa}$,
	\item  $\Fl(G,\mu^{-1})^{Newt=[b]}=\Fl(G,\mu^{-1})^{HN=[b]}$.
\end{enumerate}
\end{theorem}
The equivalence $(1)\Leftrightarrow (2)$ was proved in \cite{CFS} Theorem 6.1. Here the novelty is to add the additional information (3). 
In fact, the equivalence $(1)\Leftrightarrow (3)$ was conjectured by Fargues in \cite{F2} 9.7.
In Theorem \ref{T: fully HN} we will give several further equivalent conditions.
\\

The idea to prove the above theorem is to introduce the \emph{dual} local Shimura datum $(J_b,\{\mu^{-1}\},[b^{-1}])$ (see subsection \ref{subsection twin towers} or \cite{RV} Conjecture 5.8 and \cite{SW} Corollary 23.2.3) and consider the following similar statements:
\begin{enumerate}
\item[(a)] $B(J_b,\mu^{-1})$ is fully Hodge-Newton decomposable,
\item[(b)] $\Fl(J_b,\mu^{-1},b^{-1})^a=\Fl(J_b,\mu^{-1},b^{-1})^{wa}$,
\item[(c)] $\Fl(J_b,\mu)^{Newt=[b^{-1}]}=\Fl(J_b,\mu)^{HN=[b^{-1}]}$.
\end{enumerate}
By \cite{CFS} Corollary 4.15, we have shown $(1)\Leftrightarrow (a)$ by purely group theoretical methods. Then 
by \cite{CFS} Theorem 6.1, we get $(a)\Leftrightarrow (b)$. The point here is to show $(3)\Leftrightarrow (b)$ and $(2)\Leftrightarrow (c)$, which can be viewed as certain dualities for the Newton and Harder-Narasimhan stratifications on the $p$-adic flag varieties $\Fl(G,\mu)$ and $\Fl(G,\mu^{-1})$. See Theorem \ref{T: duality HN Newt} and Corollary \ref{C: duality HN Newt}. In fact, the duality for Newton stratifications already appeared implicitly in \cite{CFS} 5.3, and the duality for Harder-Narasimhan stratifications appeared implicitly in \cite{DOR} IX.6. The novelties here are: 
\begin{itemize}
\item 
studying both dualities more systematically in the setting of twin towers principle (see \cite{CFS} 5.1 and the following section \ref{section twin towers}), 
\item showing that how the combination of both dualities produces new information and sheds new lights on the other side of the whole story,
\item extending both dualities to general not necessarily minuscule cocharacters $\mu$ by looking at the corresponding $B_{dR}^+$-affine Schubert cells, see below.
\end{itemize}

On the other hand, we can show \emph{directly} the equivalence $(1)\Leftrightarrow (3)$ by similar arguments as in the proof of \cite{CFS} Theorem 6.1, see Remark \ref{R: fully HN}. Then under the equivalences $(1)\Leftrightarrow (a)$, $(2)\Leftrightarrow (c)$ and $(3)\Leftrightarrow (b)$, we get $(a)\Leftrightarrow (c) \Leftrightarrow (b)$, and thus $(1)\Leftrightarrow(2)$. In this way we give another proof for \cite{CFS} Theorem 6.1, although essentially the two proofs are the same.
As one has seen, the equivalence $(1)\Leftrightarrow (a)$ is in fact the only key ingredient which we take from \cite{CFS}.
\\

For a general \emph{not necessarily minuscule} cocharacter $\mu$, to have a similar picture as above, the correct objects to study are the $B_{dR}^+$-affine Schubert cells (cf. \cite{SW} sections 19 and 20) \[\Gr_\mu \quad  \tr{and} \quad \Gr_{\mu^{-1}}\] instead of the corresponding flag varieties. One of the main results of \cite{SW} says that $\Gr_\mu$ and $\Gr_{\mu^{-1}}$ are \emph{locally spatial diamonds} over $E$. They are related to flag varieties by the Bialynicki-Birula maps (cf. \cite{CS} Proposition 3.4.3 and \cite{SW} Proposition 19.4.2)
\[ \pi_\mu: \Gr_\mu\ra \Fl(G,\mu)^\Diamond, \quad \tr{and} \quad \pi_{\mu^{-1}}: \Gr_{\mu^{-1}}\ra \Fl(G,\mu^{-1})^\Diamond.\]
For any perfectoid algebraically closed field $C|\breve{E}$ and any point $x\in \Gr_{\mu^{-1}}(C)$, by the Beauville-Laszlo construction, we still have a modification $\E_{1,x}$ of the trivial $G$-bundle $\E_1$ on the Fargues-Fontaine curve $X=X_{C^\flat}$. By considering the Newton vector (resp. Harder-Narasimhan vector) attached to $\E_{1,x}$ (resp. the triple $(\E_1,\E_{1,x}, f)$), we can construct the Newton (resp. Harder-Narasimhan) stratification  on $\Gr_{\mu^{-1}}$ similarly as before.  Here we study the Harder-Narasimhan vector \[\nu(\E_1,\E_{1,x}, f)\] along the lines of \cite{DOR} chapter V, replacing $G$-filtrations by admissible modifications of $G$-bundles.
In the case $G=\GL_n$, the vector $\nu(\E_1,\E_{1,x}, f)$ was defined by the Harder-Narasimhan formalism in \cite{CI}. To show each HN stratum $\Gr_{\mu^{-1}}^{HN=v}$ defines a locally spatial diamond, we actually construct a refinement of the HN vector stratification by associating each point a ``HN type'', which is an analogue of the HN type stratification introduced in \cite{DOR} chapter VI.3. Along the way, we also show that non semi-stable HN strata are parabolically induced, see Theorem \ref{T: HN semi cont}.
We remark that if $\mu$ is non minuscule, then in general $\nu(\E_1,\E_{1,x}, f)\neq \nu(\Fc_{\pi_{\mu^{-1}}(x)})$ (see Example \ref{P: HN vectors}), where $\nu(\Fc_{\pi_{\mu^{-1}}(x)})$ is the HN vector associated to the $G$-filtration $\Fc_{\pi_{\mu^{-1}}(x)}$ attached to $\pi_{\mu^{-1}}(x)\in \Fl(G,\mu^{-1})(C)$ studied in \cite{DOR}. 

Let $b\in G(\breve{\Q}_p)$ be such that $[b]\in B(G,\mu)$. For any point $x\in \Gr_\mu(C)$, as above we have a modification $\E_{b,x}$ of the $G$-bundle $\E_b$ over $X=X_{C^\flat}$ attached to $[b]$. By considering the Newton vector of $\E_{b,x}$,
we can construct the Newton stratification on $\Gr_\mu$. Moreover, the dualities for Newton stratifications on $\Gr_{\mu^{-1}}$ and $\Gr_\mu$ also hold in this general setting (see Theorem \ref{T: duality general}). If $[b]$ is \emph{basic}, by duality we can also define the Harder-Narasimhan stratification on $\Gr_\mu$ by working with $(J_b, \{\mu\}, [1])$.
When $\mu$ is minuscule,  the Bialynicki-Birula maps $\pi_{\mu^{-1}}: \Gr_{\mu^{-1}}\ra \Fl(G,\mu^{-1})^\Diamond$ and $\pi_\mu: \Gr_\mu\ra \Fl(G,\mu)^\Diamond$  are isomorphisms (cf. \cite{CS} Theorem 3.4.5 and \cite{SW} Proposition 19.4.2), and we recover the Newton and Harder-Narasimhan stratifications on the flag varieties $\Fl(G,\mu^{-1})$ and $\Fl(G,\mu)$.
\\

Let $b\in G(\breve{\Q}_p)$ be such that $[b]\in B(G,\mu)$ \emph{basic}. Starting from the datum $(G, \{\mu\}, [b])$, we get the admissible locus $\Gr_\mu^a$ (the open Newton stratum) and the weakly admissible locus $\Gr_\mu^{wa}$ (the open Harder-Narasimhan stratum) inside $\Gr_\mu$. Both of them are \emph{open} sub diamonds of $\Gr_\mu$.
We have the inclusion of locally spatial diamonds over $\breve{E}$:
\[\Gr_\mu^a\subset \Gr_\mu^{wa}. \]
 On the Hodge-Tate side $\Gr_{\mu^{-1}}$, 
 by the inequality $\nu(\E_1,\E_{1,x},f)\leq \nu(\E_{1,x})$ as above, we have the inclusion for \emph{open} Newton and Harder-Narasimhan strata: \[\Gr_{\mu^{-1}}^{Newt=[b]}\subset\Gr_{\mu^{-1}}^{HN=[b]}.\]
Here is the generalization of Theorem \ref{T1}, where we remove the minuscule condition (see \cite{F2} 9.7, Conjecture 1 (1)):
\begin{theorem}[Theorem \ref{T: fully HN general}, Corollary \ref{C: fully HN general}]
Let $[b]\in B(G,\mu)$ be basic. The following statements are equivalent:
\begin{enumerate}
	\item $B(G,\mu)$ is fully Hodge-Newton decomposable,
	\item $\Gr_\mu^a=\Gr_\mu^{wa}$,
	\item  $\Gr_{\mu^{-1}}^{Newt=[b]}=\Gr_{\mu^{-1}}^{HN=[b]}$.
\end{enumerate}
\end{theorem}
As for Theorem \ref{T1}, once we prove the equivalence $(1) \Leftrightarrow (2)$, which is the generalized version of \cite{CFS} Theorem 6.1, the remaining equivalence $(1)\Leftrightarrow (3)$ follows by the dualities for Newton and Harder-Narasimhan stratifications.
The key new idea is to study the geometry of $\Gr_\mu$ in terms of $B_{dR}^+$-affine Grassmannians of parabolic and Levi subgroups of $G$, which is in some sense a theory of generalized semi-infinite orbits in the current setting. More precisely, we have the following new\footnote{We note that some of the results here can be deduced from the recent work of Fargues-Scholze \cite{FS} chapter VI. However, our method is more concrete and direct, as we work with $B_{dR}^+$-affine Grassmannians over $\Spa(\Q_p)^\Diamond$.} information:
\begin{itemize}
	\item We prove the dimension formula and closure relation for the $B_{dR}^+$-affine Schubert cells (same as the classical setting, see Proposition \ref{P:affine Schubert}).
	\item Let $M$ be a Levi subgroup inside a parabolic $P$ of $G$ over $\Q_p$. Then we have a stratification $\Gr_\mu=\coprod_{\lambda\in S_M(\mu)}\Gr_{G,\mu,\lambda}$, induced on $\Gr_\mu$ by the natural diagram of the $B_{dR}^+$-affine Grassmannians of $M, P$ and $G$ respectively (the strata $\Gr_{G,\mu,\lambda}$ are intersections of the generalized semi-infinite orbits $S_\lambda$ with $\Gr_\mu$, see subsection \ref{subsection semi-infinite} for more details). Moreover, we know the closure relation for this stratification and we give some description for the index set $S_M(\mu)$ (which is in fact related to the geometric Satake equivalence for $B_{dR}^+$-affine Grassmannians, cf. \cite{FS}).
	\item The above stratification naturally arises when we study reductions of modifications of $G$-bundles to $P$-bundles (resp. $M$-bundles) on the Fargues-Fontaine curve, cf. Lemma \ref{L:modif type}. Using this, we give an interpretation of the weakly admissible locus $\Gr_\mu^{wa}$ in terms of the Fargues-Fontaine curve, cf. Proposition \ref{P:weakly adm general}, which is a generalization of \cite{CFS} Proposition 2.7.
\end{itemize}
With these new ingredients at hand, 
the arguments in the proof of \cite{CFS} Theorem 6.1 apply here to establish the above equivalence $(1) \Leftrightarrow (2)$,
see Theorem \ref{T: fully HN general} for more details. 
\\

The pullbacks under the Hodge-Tate period morphisms define Harder-Narasimhan stratifications on moduli of local $G$-Shtukas (cf. \cite{SW} section 23) and on Shimura varieties, see sections \ref{section loc Shtukas} and \ref{section Shimura}. In particular, for applications to Shimura varieties, Theorem \ref{T1} will be enough, and the new perspective $(1)\Leftrightarrow (3)$ is crucial. We hope these constructions will be found useful for further arithmetic applications (cf. \cite{F2} 9.7.2).
\\

In a recent paper \cite{Vieh}, Viehmann has made some similar study on modifications of $G$-bundles on the Fargues-Fontaine curve. The work \cite{Vieh} is more group theoretical in nature. Here we are mainly motivated by the Harder-Narasimhan formalism as in the works \cite{DOR, CI, F2}. When we submitted our revised version, we learned that Nguyen and Viehmann have introduced and studied  Harder-Narasimhan strata on $B_{dR}^+$-affine Grassmannians in a similar way (\cite{NV}).
 \\

We briefly describe the structure of this article. In section 2, we review some basics about modifications of $G$-bundles on the Fargues-Fontaine curve, which will be our tool in the following. In section 3, we define and study the Newton and Harder-Narasimhan strata on the flag variety $\Fl(G,\mu^{-1})$ under the assumption that $\mu$ is minuscule. In section 4, we explain how to transfer the point of view by using modifications of $J_b$-bundles. More precisely, we explain how to identify the Newton and Harder-Narasimhan strata on the Hodge-Tate (resp. de Rham) side for $G$ by the corresponding strata on the de Rham (resp. Hodge-Tate) side for $J_b$. These are the dualities of the Newton and Harder-Narasimhan strata established in Theorem \ref{T: duality HN Newt} and Corollary \ref{C: duality HN Newt}.
In section 5, we summarize various equivalent conditions for a fully Hodge-Newton decomposable pair $(G,\{\mu\})$ with $\mu$ minuscule, applying results of sections 3 and 4.  In section 6, we generalize the previous constructions and results to not necessarily minuscule cocharacters $\mu$ by studying the $B_{dR}^+$-affine Schubert cells $\Gr_{\mu^{-1}}$ and $\Gr_\mu$. In particular, we choose to work on the de Rham side $\Gr_\mu$ and prove the generalized Fargues-Rapoport conjecture (Theorem \ref{T: fully HN general}). Then we transfer back to the Hodge-Tate side $\Gr_{\mu^{-1}}$ (Corollary \ref{C: fully HN general}) by dualities (Theorem \ref{T: duality general}). We could start with the contents of section 6 directly after section 2. However, we decide to discuss firstly the more classical objects of flag varieties to illustrate the duality principle.
In sections 7 and 8, we discuss some applications to moduli of local $G$-Shtukas and Shimura varieties respectively.\\
 \\
\textbf{Acknowledgments.} I would like to thank Sian Nie for some discussions on group theory. I wish to thank Michael Rapoport for helpful remarks on the first version of this paper. I also thank Miaofen Chen and Laurent Fargues for useful conversations. I want to thank Eva Viehmann and Christophe Cornut sincerely for valuable correspondences which lead to the correction for a mistake in the previous version. I would like to thank the referee for helpful comments and suggestions which lead to some further improvements.
The author was partially supported by the National Key R$\&$D Program of China 2020YFA0712600, the CAS Project for Young Scientists in Basic Research, Grant No. YSBR-033, and the NSFC grants No. 11631009 and No. 11688101.

\section{Modifications of $G$-bundles on the Fargues-Fontaine curve}
The purpose of this section is to study the Harder-Narasimhan theory for \emph{admissible} modifications of $G$-bundles on the Fargues-Fontaine curve. We refer to \cite{And, Cor1, Pott} for some generalities on the Harder-Narasimhan formalism. We will also establish the basic setting of some closely related objects.

Let $C|\Q_p$ be a fixed algebraically closed perfectoid field, with $C^\flat$ its tilt. We have the Fargues-Fontaine curve $X=X_{C^\flat}$ over $\Q_p$, together with the canonical point $\infty\in X$ with completed local ring  $\wh{\Ol}_{X,\infty}=B^+_{dR}(C)$. We refer the reader to \cite{FF} for a detailed study of this curve, and to \cite{CFS} section 1 for a brief summary.
We will simply write $B_{dR}=B_{dR}(C)$ and $B^+_{dR}=B^+_{dR}(C)$ in the following. Let $\xi\in B^+_{dR}$ be a fixed uniformizer.
Let $\phmod_{\breve{\Q}_p}$ be the category of $F$-isocrystals over $\ov{\F}_p$, and $\Bun_X$ be the category of vector bundles on $X$. A basic result of \cite{FF} says that we have a natural functor \[\E(-): \phmod_{\breve{\Q}_p}\lra \Bun_{X}\] which is essentially surjective.
For $\E\in \Bun_X$, we have the Harder-Narasimhan filtration of $\E$ with the associated Harder-Narasimhan vector $\nu(\E)\in \Q_+^n$ where $n=\rank \,\E$.
To avoid confusion, we will call it the Newton filtration, since later we will introduce several further Harder-Narasimhan filtrations. We remind the reader that the main tools used in \cite{CFS} are the Harder-Narasimhan theories for $G$-bundles over $X$ and filtered isocrystals with $G$-structures.

\subsection{Modifications of vector bundles}\label{subsection modif vect}
We are interested in the category of modifications\footnote{In this paper we only consider modifications at the canonical point $\infty \in X$. } of vector bundles on $X$, which we denote by
\[\Modif_X.\]
Recall that a modification of vector bundles is a triple $\Eu=(\E_1,\E_2,f)$, where \begin{itemize}
	\item $\E_1, \E_2$ are vector bundles on $X$,
	\item $f: \E_1|_{X\setminus\{\infty\}} \st{\sim}{\lra}\E_2|_{X\setminus\{\infty\}} $ is an isomorphism.
\end{itemize} 
A morphism $F: \Eu\ra \Eu'$ is a pair of morphisms $F_i: \E_i\ra \E_i'$ with $F_2\circ f=f'\circ F_1$. This category $\Modif_X$ is a quasi-abelian $\Q_p$-linear rigid $\otimes$-category with a Tate twist, cf. \cite{CI} 3.1.4. For a modification $(\E_1,\E_2,f)$, let \[\E_{i,dR}^+=\E_{i,\infty}^\wedge\] be the completed local stalk at $\infty$ of $\E_i$, and \[f_{dR}: \E_{1,dR}^+[\xi^{-1}]\ra  \E_{2,dR}^+[\xi^{-1}]\] be the induced isomorphism of $B_{dR}$-vector spaces.
We have the Newton filtrations $\Fc_{N,i}(\Eu):=\Fc(\E_i)$ for $i=1,2$. Moreover, we have the Hodge filtrations $\Fc_{H,i}(\Eu)$, which are the $\Z$-filtrations on the residues $\E_i(\infty):=\E_{i,dR}^+/\xi\E_{i,dR}^+$ of $\E_i$ induced by $\E_{3-i,dR}^+$: for any $j\in \Z$,
\[\Fc_{H,1}^j=\frac{f_{dR}^{-1}(\xi^j\E_{2,dR}^+)\cap\E_{1,dR}^++\xi\E_{1,dR}^+}{\xi\E_{1,dR}^+},\quad \Fc_{H,2}^j=\frac{f_{dR}(\xi^j\E_{1,dR}^+)\cap\E_{2,dR}^++\xi\E_{2,dR}^+}{\xi\E_{2,dR}^+}. \]
Let $n=\rank(\E_1)=\rank(\E_2)$. Then $\Fc_{H,1}$ and $\Fc_{H,2}$ define opposed types $\nu_{H,i}(\Eu)\in \Z^n_+$.

We have  the following natural functors
\[\xymatrix{ & \Modif_X\ar[ld]_{\overleftarrow{h}}\ar[rd]^{\overrightarrow{h}}&\\
	\Bun_X& & \Bun_X,
} \]
with \[\overleftarrow{h}(\E_1,\E_2, f)=\E_2, \quad \overrightarrow{h}(\E_1,\E_2, f)=\E_1. \]These functors $\overleftarrow{h}$ and $\overrightarrow{h}$ will be related to the de Rham periods and the Hodge-Tate periods respectively.

\subsection{Filtered $F$-isocrystals}\label{subsection Fil iso}
For any extension $K|\breve{\Q}_p$ (not necessary finite),
let \[\varphi\mathrm{-}\Fil\Mod_{K/\breve{\Q}_p}\] be the category of filtered $F$-isocrystals with respect to $K|\breve{\Q}_p$. A  filtered $F$-isocrystal $\D=(D,\varphi,\Fc)\in \varphi\mathrm{-}\Fil\Mod_{K/\breve{\Q}_p}$ consists of a underlying $F$-isocrystal $(D,\varphi)\in \phmod_{\breve{\Q}_p}$ together with a $\Q$-filtration $\Fc$ on $D\otimes_{\breve{\Q}_p} K$.
We have the rank and degree functions
\[\rank: \varphi\mathrm{-}\Fil\Mod_{K/\breve{\Q}_p}\ra \N,\quad \deg: \varphi\mathrm{-}\Fil\Mod_{K/\breve{\Q}_p}\ra \Z \]
defined by \[\rank\,\D=\dim\,D, \quad \deg\,\D=t_H(\D)-t_N(\D), \]
where $t_H(\D)=\sum_ii\dim gr_\Fc^iD_K$ and $t_N(\D)=v_p(\det\varphi)$. These functions induce a Harder-Narasimhan filtration on $\varphi\mathrm{-}\Fil\Mod_{K/\breve{\Q}_p}$.

Consider the case $K=C$.
By Fargues's de Rham classification for modifications of vector bundles,
there exists a functor
$\pi: \Modif_X\lra \varphi\mathrm{-}\Fil\Mod_{C/\breve{\Q}_p}$, for which
we refer to \cite{F5} 4.2.2 for more details.

\subsection{Admissible modifications}\label{subsection adm modif}
Consider the full subcategory of admissible modifications \[\Modif^{ad}_X\] inside $\Modif_X$. Recall that a modification $\Eu=(\E_1,\E_2,f)$ is called admissible if
$\E_1$ is a semi-stable vector bundle of slope 0 (i.e. $\E_1$ is the trivial vector bundle). This is again a quasi-abelian $\Q_p$-linear rigid $\otimes$-category with a Tate twist. For an admissible modification $\Eu=(\E_1,\E_2,f)$, we set
\[\Fc_N(\Eu):=\Fc(\E_2),\quad \Fc_H(\Eu):=\Fc_{H,1}(\Eu), \quad \nu_N(\Eu):=\nu(\E_2),\quad \tr{and}\quad \nu_H(\Eu):=\nu_{H,1}(\Eu).\]
We have an exact $\Q_p$-linear faithful $\otimes$-functor
\[\omega: \Modif^{ad}_X\ra \Vect_{\Q_p},\quad \omega(\Eu)=\Gamma(X,\E_1), \]
which induces a bijection between the poset of strict subobjects of $\Eu$ in $\Modif^{ad}_X$ and the poset of $\Q_p$-subspaces of $\omega(\Eu)$. We have the rank and degree functions
\[\rank: \Modif^{ad}_X \ra \N, \quad \deg: \Modif^{ad}_X \ra \Z \]
defined by \[\rank(\Eu)=\rank(\E_1)=\rank(\E_2)=\dim\omega(\Eu) \]
and \[\deg(\Eu)=\deg\,\E_2.\]
They induce a Harder-Narasimhan\footnote{In \cite{CI} this filtration is called the Fargues filtration.} filtration on $\Modif^{ad}_X$ with slopes $\mu=\deg/\rank$ in $\Q$. We denote it by $\Fc(\Eu)$ with the associated Harder-Narasimhan vector $\nu(\Eu)$.

Later we will need the following variant.
Let $\Modif^{ad'}_X$ be the subcategory of modifications $\Eu=(\E_1,\E_2,f)$ with $\E_2$ trivial. On this category we have the rank and degree functions defined by $\rank(\Eu)=\rank(\E_1)=\rank(\E_2)$ and $\deg(\Eu)=\deg\,\E_1$. One checks similarly as above that they induce a Harder-Narasimhan filtration on $\Modif^{ad'}_X$. Moreover, we have the equivalence
\[\Modif^{ad'}_X\st{\sim}{\lra}\Modif^{ad}_X,\quad \Eu=(\E_1,\E_2,f)\mapsto \Eu'=(\E_2,\E_1,f^{-1})\]
and $\nu(\Eu)=\nu(\Eu')$.
\\

Let $\HT^{B_{dR}}$ be the category of pairs $(V, \Xi)$, where
\begin{itemize}
	\item $V$ is a finite dimensional $\Q_p$-vector space,
	\item $\Xi$ is a $B_{dR}^+$-lattice in $V_{dR}=V\otimes B_{dR}$.
\end{itemize}
A morphism $F: (V, \Xi)\ra (V',\Xi')$ is a $\Q_p$-linear morphism $f: V\ra V'$ such that the induced morphism $f_{dR}: V_{dR}\ra V_{dR}'$ satisfies $f(\Xi)\subset \Xi'$. 
This defines a quasi-abelian rigid $\Q_p$-linear $\otimes$-category.
We have the natural functor $\omega: \HT^{B_{dR}}\ra \Vect_{\Q_p}, (V, \Xi)\mapsto V$.
For $(V,\Xi)\in \HT^{B_{dR}}$, the lattice $\Xi$ induces a Hodge filtration $\Fc_H(V,\Xi)$, which is the $\Z$-filtration on the residue $V_C=V\otimes C$ of the lattice $V_{dR}^+=V\otimes B_{dR}^+\subset V_{dR}$, giving by the formula:
\[\Fc_H(V,\Xi)^i=\frac{\xi^i\Lambda\cap V_{dR}^++\xi V_{dR}^+}{\xi V_{dR}^+}.\]
Moreover, we have the rank and degree functions
\[\rank: \HT^{B_{dR}}\ra \N,\quad \deg: \HT^{B_{dR}}\ra \Z \]defined by
\[\rank(V,\Xi)=\dim V=\rank_{B_{dR}^+}(\Xi), \]and
\[\deg(V,\Xi)=\sum_ii\dim gr^i_{\Fc_H}V_C.\]
They induce a Harder-Narasimhan filtration on $\HT^{B_{dR}}$. We denote it by $\Fc(V,\Xi)$ with Harder-Narasimhan vector $\nu(V,\Xi)$.

By Fargues's Hodge-Tate classification in \cite{F5} 4.2.3, we have an exact $\otimes$-equivalence of $\otimes$-categories
\[\HT: \Modif^{ad}_X\ra  \HT^{B_{dR}}, \quad \Eu\mapsto (\Gamma(X,\E_1), f_{dR}^{-1}(\E_{2,dR}^+)). \]
The inverse functor is given by $(V,\Xi)\mapsto (\E_1,\E_2,f)$, where
\begin{itemize}
	\item $\E_1=V\otimes \Ol_X$
	\item $\E_2$ and $f$ are given by the modification of $\E_1$ at $\infty$ corresponding to the $B^+_{dR}$-lattice $\Xi$ of $\E_{1,\infty}^\wedge[\xi^{-1}]=V\otimes B_{dR}$ under the Beauville-Laszlo correspondence (cf. \cite{FF} 5.3.1).
\end{itemize}
The functor $\HT$ preserves the rank and deg functions on the two categories, and it induces a bijection between the posets of strict subobjects of $\Eu$ and $\HT(\Eu)$ with compatible ranks and degrees. Thus it preserves the Harder-Narasimhan filtrations and types on both sides.

\subsection{Filtered vector spaces}\label{subsection fil vec space}
Consider the category $\Fil_{\Q_p}^C$ of pairs $(V, \Fc)$, where
\begin{itemize}
	\item $V$ is a finite $\Q_p$-vector space.
	\item $\Fc$ is a descending $\Q$-filtration on $V_C=V\otimes C$.
\end{itemize}
We have the natural functor $\omega: \Fil_{\Q_p}^C\ra \Vect_{\Q_p}, (V,\Fc)\mapsto V$.
For $(V,\Fc)\in \Fil_{\Q_p}^C$, the rank and deg functions are defined by \[\rank(V, \Fc)=\dim V, \quad \deg(V, \Fc)=\sum_i i\dim gr^i_\Fc V_C\] induce a Harder-Narasimhan filtration on $\Fil_{\Q_p}^C$.
We have a natural functor
\[\pi: \HT^{B_{dR}}\lra \Fil_{\Q_p}^C,\quad (V,\Xi)\mapsto (V, \Fc_H(V,\Xi)) \] preserving the rank and deg functions. Composed with the equivalence functor $\HT: \Modif^{ad}_X\ra  \HT^{B_{dR}}$ we get \[\pi: \Modif^{ad}_X\lra \Fil_{\Q_p}^C.\] We remark that we can also construct the functor $\pi: \Modif^{ad}_X\lra \Fil_{\Q_p}^C$ by using the functor in \ref{subsection Fil iso} and \cite{F5} Proposition 4.17.  In summary, we get the following commutative diagram
\[\xymatrix{
\Modif_{X}^{ad}\ar[r]^{\HT}_{\sim}\ar[rd]_\omega&\HT^{B_{dR}}\ar[r]^\pi\ar[d]^\omega&\Fil_{\Q_p}^C\ar[ld]^\omega\\
&\Vect_{\Q_p}&.
}\]

It is curious to compare the Harder-Narasimhan vectors $\nu(V,\Xi)$ and $\nu(\pi(V,\Xi))$. If the filtration $\Fc_H(V,\Xi)$ is minuscule, then by \cite{F2} subsection 7.2 we have $\nu(V,\Xi)=\nu(\pi(V,\Xi))$. In general, this is not true. The reason comes from the following subtle point: Let $(V, \Xi)\in \HT^{B_{dR}}$ and $W\subset V$ be a subspace, then we get the sub objects $(W, \Xi_W)$ of $(V, \Xi)$ and $(W, \Fc_W)$ of $(V, \Fc_H(V,\Xi))$, where $\Xi_W=\Xi\cap (W\otimes B_{dR})$ and $\Fc_W$ is the induced filtration on $W\otimes C$ by $\Fc_H(V,\Xi)$. In general, we have \[\Fc_H(W, \Xi_W)\neq \Fc_W.\] Indeed, the functor $\pi$ does not preserve the degrees of sub objects in general. This leads the inequality \[\nu(V,\Xi)\neq\nu(\pi(V,\Xi))\] in general.
The following example is due to Viehmann:
\begin{example}\label{P: HN vectors}
Let $V=\Q_p^2$ with standard basis $e_1$ and $e_2$.  For any $a\in C$, consider the $B_{dR}^+$-lattice $\Xi_a \subset B_{dR}^2$ generated by $\xi^2 e_1$ and $e_2+a\xi e_1$. Then $\Fc_H(V, \Xi_a)=(\Fil^i)_{i\in \Z}$ with
\[\Fil^i=\begin{cases}
C^2,\quad &i\leq -2;\\
C e_2, \quad &i=-1, 0;\\
0, \quad &i\geq 1.
\end{cases}\]
Consider the subspace $W=\Q_p  e_2\subset V$. Then an easy computation gives $\deg(\Fc_W)=0$ and $\deg(\Fc_H(W, \Xi_{a,W}))=-1$ if $a\neq 0$.
\end{example}

\subsection{$G$-structures}\label{subsection G-str}
Let $G$ be a connected reductive group over $\Q_p$. We would like to add ``$G$-structures'' to our previous discussions.

Let us first fix some notations.
We fix a minimal parabolic subgroup $P_0$ of $G$ defined over $\Q_p$ and a Levi subgroup $M_0$ of $P_0$. Then a standard parabolic subgroup is a parabolic $P$ with $P\supset P_0$. There is a unique Levi subgroup $M$ of $P$ containing $M_0$, which we call a standard Levi subgroup. We write $U_P$ for the unipotent radical of $P$. 
Let $A\subset M_0$ be the maximal split torus over $\Q_p$, and $T\subset M_0$ be a maximal torus of $M_0$ defined over $\Q_p$ which contains $A$.  Then $T=M_0$ if and only if $G$ is quasi-split over $\Q_p$.

For a parabolic subgroup $P\subset G$ with Levi subgroup $M\subset P$ over $\Q_p$, let $W_P:=W_M$ be the absolute Weyl group of $M$. Assume that $P\supset P_0$ is standard with associated standard Levi $M$. Let $A_M$ be the maximal split torus contained in the center $Z_M$ of $M$, and $A_M'$ be the maximal split quotient torus of $M$. Then we have a natural isogeny $A_M\ra A_M'$.  We also write $A_P=A_M$ and $A_P'=A_M'$.  In particular $A=A_{P_0}=A_{M_0}$. If $Q\supset P$, then we have an inclusion $A_Q\subset A_P$ and a quotient $A_P'\ra A_Q'$. Let $B\subset G_{\ov{\Q}_p}$ be a Borel subgroup such that $B\subset P_{0,\ov{\Q}_p}$. Let $T\subset B$ be a maximal torus such that $A\subset T\subset M_0$. Then we get the 
absolute based root datum
\[(X^\ast(T), \Phi, X_\ast(T), \Phi^\vee, \Delta)\]
and the relative based root datum
\[(X^\ast(A), \Phi_0, X_\ast(A), \Phi^\vee_0, \Delta_0).\]
Let $\Delta_P$ (resp. $\Delta_{0,P}$) be the set of non-trivial restrictions of elements of  $\Delta$ (resp. $\Delta_{0}$) to $Z_M$ (resp. $A_P$) (recall $Z_M\subset T$ resp. $A_P\subset A$). If we replace $G$ by $M$ and let\footnote{Note that the notation here is compatible with the notation of \cite{CFS}.} $\Delta_M$ (resp. $\Delta_{0,M}$) be the set of simple roots (resp. relative roots) of $M$, then $\Delta_P$ (resp. $\Delta_{0,P}$) is in bijection with $\Delta\setminus\Delta_M$ (resp. $\Delta_{0}\setminus \Delta_{0,M}$). Let $\Delta^\vee$ be the set of simple coroots of $G$, then we have $\Delta_P^\vee$ corresponding to $P$. Similarly we have the relative version $\Delta_0^\vee$ and $\Delta_{0,P}^\vee$.
 \\

Let $W$ and $W_0$ be the absolute and relative Weyl groups of $G$ respectively.  We identify \[X_\ast(A)_\Q/W_0= X_\ast(A)_\Q^+:=\{x\in X_\ast(A)_\Q|\,\lan x,\alpha\ran \geq 0,\;\forall\,\alpha\in \Delta_{0}\}.\]
On the other hand, consider $$\mathcal{N} (G):= \big [ \Hom (\mathbb{D}_{\overline{\Q}_p},G_{\overline{\Q}_p}) \, /\, G(\overline{\Q}_p)\text{-conjugacy}\big ]^\Gamma,$$ with $\mathbb{D}$ the pro-torus over $\Q_p$ whose character group is $\Q$ and $\Gamma=\Gal(\ov{\Q}_p/\Q_p)$. As in \cite{DOR}, let $X_\ast(G)$ denote the set of cocharacters of $G$ defined over $\ov{\Q}_p$. Then $X_\ast(G)_\Q=\Hom (\mathbb{D}_{\overline{\Q}_p},G_{\overline{\Q}_p})$, on which $G(\ov{\Q}_p)$ acts by conjugation. We will write $\Nc(G)=(X_\ast(G)_\Q/G)^\Gamma$. We have identifications $(X_\ast(T)_\Q/W)^\Gamma=(X_\ast(G)_\Q/G)^\Gamma$ and $X_\ast(A)_\Q/W_0=X_\ast(G)_\Q^\Gamma/G(\Q_p$. Then the
 natural inclusion $X_\ast(G)_\Q^\Gamma/G(\Q_p)\subset (X_\ast(G)_\Q/G)^\Gamma$ can be rewritten as \[X_\ast(A)_\Q^+\subset \Nc(G).\] We have
\[G\; \tr{is quasi-split over}\;\Q_p \quad \Longleftrightarrow\quad
X_\ast(A)_\Q^+=\Nc(G).\]
We identify \[X_\ast(T)_\Q/W=X_\ast(T)^+_\Q=\{x\in X_\ast(T)_\Q|\,\lan x,\alpha\ran \geq 0,\;\forall\,\alpha\in \Delta\}.\]
Moreover, the choice of $B$ defines a partial order $\leq$ on $X_\ast(T)$ by $\mu_1\leq \mu_2$ if $\mu_2-\mu_1$ is a sum of positive coroots with non negative integral coefficients. We get an induced partial order $\leq$ on $X_\ast(T)_\Q$ and thus on $\Nc(G)\subset X_\ast(T)^+_\Q\subset X_\ast(T)_\Q$.  By \cite{Springer} 15.5.8, we get an involution $x\mapsto x^\ast:=w_0(-x)$ on $\Nc(G)$, where $w_0$ is the element of longest length in $W$ acting on $X_\ast(T)_\Q$. 
\\

Recall that an $F$-isocrystal with $G$-structure is an exact tensor functor \[N: \Rep\, G\lra \phmod_{\breve{\Q}_p}.\]
An element $b\in G(\breve{\Q}_p)$ defines an isocrystal with $G$-structure \[\begin{split} N_b: \Rep\, G&\lra \phmod_{\breve{\Q}_p} \\ V&\longmapsto (V_{\breve{\Q}_p}, b\sigma).\end{split}\] 
Its isomorphism class only depends on the $\sigma$-conjugacy class $[b]\in B(G)$ of $b$, where $B(G)$ is the set of $\sigma$-conjugacy classes in $G(\breve{\Q}_p)$, cf. \cite{Kot1, Kot2, RR}. By Steinberg's theorem, any isocrystal with $G$-structure arises in this way. Thus $B(G)$ is the set of isomorphism classes of isocrystals with $G$-structure, cf. \cite{RR} Remarks 3.4 (i). 
We have the Newton map (\cite{Kot1} section 4) and Kottwitz map (\cite{Kot} section 6 and \cite{Kot2} 4.9, 7.5) \[\nu: B(G)\ra \Nc(G), \quad \kappa: B(G)\ra \pi_1(G)_\Gamma, \]
where \[\pi_1(G)=X_\ast(T)/\lan \Phi^\vee\ran\] (by our previous group theoretic notations, and it does not depend on the choice of $T$) and $\pi_1(G)_\Gamma$ is its Galois coinvariant.
 In fact, $\nu$ is induced by a map $\nu: G(\breve{\Q}_p)\ra \Hom (\mathbb{D}_{\breve{\Q}_p},G_{\breve{\Q}_p})$, while $\kappa$ is induced by a map $\kappa: G(\breve{\Q}_p)\ra \pi_1(G)_\Gamma$. For this reason we also write $\nu([b])=[\nu_b]$ and $\kappa([b])=\kappa(b)$ for $b\in G(\breve{\Q}_p)$ with the induced class $[b]\in B(G)$. The partial order on $\Nc(G)$ induces a partial order $\leq$ on $B(G)$ (cf. \cite{RR} section 2). Consider the subset $B(G)_{basic}\subset B(G)$ consisting of elements whose image under the Newton map are central (such elements are called \emph{basic}). Then the restriction of $\kappa$ induces a bijection $\kappa: B(G)_{basic}\st{\sim}{\ra}\pi_1(G)_\Gamma$, cf. \cite{Kot1} section 5.

Recall that we have the following commutative diagram for the Kottwitz and Newton maps (see \cite{RR} p. 162):
\[\xymatrix{ B(G)\ar[r]^{\nu}\ar[d]^{\kappa}& \Nc(G)\ar[d]\\
	\pi_1(G)_\Gamma\ar[r]& \pi_1(G)_{\Gamma,\Q},
}\]
where we identify \[\pi_1(G)_{\Gamma,\Q}=\pi_1(G)^{\Gamma}_\Q=X_\ast(Z_G)_\Q^\Gamma=X_\ast(A_G)_\Q,\]where $A_G$ is the maximal split torus inside the center $Z_G$ of $G$.
For an element $v\in \Nc(G)$, in the following
we will denote its image in $\pi_1(G)_{\Gamma, \Q,}$ by the same notation $v$ for simplicity. For a Levi subgroup $M\subset G$, we have the corresponding commutative diagram as above for $G$, which maps to that for $G$, since all the maps in the diagram is functorial. 

We explain $B(G)$ in terms of $G$-bundles on the Fargues-Fontaine curve $X$ as follows. Recall that
we have the following two equivalent definitions of a $G$-bundle on $X$:
\begin{itemize}
	\item an exact tensor functor $\Rep \, G\ra \Bun_X$, where $\Rep \, G$ is the category of rational algebraic representations of $G$,
	\item a $G$-torsor on $X$ locally trivial for the \'etale topology.
\end{itemize}
Attached to a $G$-bundle $\E$ on $X$, we have the Newton vector $\nu(\E)\in \Nc(G)$ and the $G$-equivariant first Chern class $c_1^G(\E)\in \pi_1(G)_\Gamma$.
For $b\in G(\breve{\Q}_p)$, let $\E_b$ be the composition of the above functor $N_b$ and \[\E(-):  \phmod_{\breve{\Q}_p}\lra  \Bun_X.\] 
In this way, the set $B(G)$ also classifies $G$-bundles on $X$. In fact, we have
\begin{theorem}[\cite{F3}]\label{T: G-bundles}
	There is a bijection of pointed sets
	\[\begin{split} B(G)&\st{\sim}{\lra} H^1_{\textrm{\'et}}(X, G) \\
	[b]&\longmapsto [\E_b]. \end{split}	\]
	Under this bijection, we have\[
		\nu(\E_b)=w_0(-\nu([b])), 
		\quad c_1^G ( \E_b)=- \kappa ([b]).\]
\end{theorem}

Let $\omega^G: \Rep\,G\ra \Vect_{\Q_p}$ be the standard fiber functor for the category $\Rep\,G$ of algebraic representations of $G$. For any field extension $K|\Q_p$, let $\Fil_{K}(\omega^G)$ be the set of $\Q$-filtrations of $\omega^G$ over $K$. An element $\Fc\in \Fil_{K}(\omega^G)$ is given by a tensor functor \[F: \Rep\,G\lra \Fil_{\Q_p}^K\] such that $\omega^G=\omega\circ F$ and the induced tensor functor \[\tr{gr}\circ F: \Rep\,G\lra \tr{Grad}_K\] is exact. Here $\omega: \Fil_{\Q_p}^K\ra \Vect_{\Q_p}$ is the natural functor and $\tr{gr}: \Fil_{\Q_p}^K\ra \tr{Grad}_K$ is the graded functor from $\Fil_{\Q_p}^K$ to the category of graded $K$-vector spaces. We refer the reader to \cite{DOR} chapter IV.2 for more discussions on these objects. We have a natural map
\[\Fil_{\Q_p}(\omega^G)\ra X_\ast(G)^\Gamma_\Q/G(\Q_p)=X_\ast(A)_\Q^+\subset (X_\ast(G)_\Q/G)^\Gamma=\Nc(G). \]

A modification of $G$-bundles is given by 
\begin{itemize}
	\item either an exact tensor functor $\Rep \, G\ra \Modif_X$,
	\item or a triple $(\E_1,\E_2,f)$, where $\E_1,\E_2$ are $G$-bundles on $X$ and $f: \E_1|_{X\setminus\{\infty\}} \st{\sim}{\ra}\E_2|_{X\setminus\{\infty\}} $ is an isomorphism.
\end{itemize}
Applying the functor $\pi: \Modif_X\ra \varphi\mathrm{-}\Fil\Mod_{C/\breve{\Q}_p}$ in subsection \ref{subsection Fil iso},
a modification of $G$-bundles $\ul{\E}=(\E_1,\E_2,f)$ gives rise to a filtered $F$-isocrystal with $G$-structure
\[\pi(\ul{\E}): \Rep \, G \lra  \varphi\mathrm{-}\Fil\Mod_{C/\breve{\Q}_p},\]
which is in turn equivalent to a pair $(N,\Fc)$, where (cf. \cite{DOR} p. 239)
\begin{itemize}
	\item $N: \Rep \, G \ra  \varphi\mathrm{-}\Mod_{\breve{\Q}_p}$ is the underlying 
$F$-isocrystal with $G$-structure induced by the natural functor (forgetting filtrations) $\varphi\mathrm{-}\Fil\Mod_{C/\breve{\Q}_p}\ra \varphi\mathrm{-}\Mod_{\breve{\Q}_p}$, 
\item $\Fc\in \Fil_{C}(\omega^G)$.
\end{itemize}
By \cite{DOR} Theorem 9.2.18,
	for the pair $(N,\Fc)$,
	there exists a unique $\Q$-filtration ${}^\bullet N_\Fc$ of $N$, such that for any $(V,\rho)\in \Rep\,G$, the induced filtration ${}^\bullet N_\Fc(V)$ on $N(V)$ is the Harder-Narasimhan filtration of the filtered isocrystal $(N(V), \Fc^\bullet N(V))$.
In particular, the $\Q$-filtration ${}^\bullet N_\Fc\in \Fil_{\breve{\Q}_p}(\omega^G)$ defines a Harder-Narasimhan vector $\nu(N,\Fc) \in (X_\ast(G)_\Q/G)^{\Gamma_0}$ where $\Gamma_0=\Gal(\ov{\Q}/\breve{\Q}_p)$. By \cite{DOR} IX.4, we have in fact \[\nu(N,\Fc)\in \Nc(G)=(X_\ast(G)_\Q/G)^\Gamma. \]

An admissible modification of $G$-bundles is given by 
\begin{itemize}
	\item either an exact tensor functor $\Rep \, G\ra \Modif^{ad}_X$, such that when composing with $\omega: \Modif^{ad}_X\ra \Vect_{\Q_p}$ we get $\omega^G: \Rep \, G\ra \Vect_{\Q_p}$,
	\item or a triple $(\E_1,\E_2,f)$, where $\E_1,\E_2$ are $G$-bundles on $X$ such that $\E_1$ is the trivial $G$-bundle and $f: \E_1|_{X\setminus\{\infty\}} \st{\sim}{\ra}\E_2|_{X\setminus\{\infty\}} $ is an isomorphism.
\end{itemize}
By the equivalence of categories $\HT: \Modif^{ad}_X \st{\sim}{\ra} \HT^{B_{dR}}$,
given an admissible modification $\ul{\E}=(\E_1,\E_2,f)$ is equivalent to given an exact functor
\[\HT(\ul{\E}): \Rep \, G\lra \HT^{B_{dR}}.\]
Composing with the functor $\pi: \HT^{B_{dR}}\ra \Fil_{\Q_p}^C$ in \ref{subsection fil vec space},
we then
get a functor \[\pi(\ul{\E}): \Rep \, G \lra \Fil_{\Q_p}^C,\] which defines
an element of $\Fil_C(\omega^G)$. Recall that on all the categories $\Modif^{ad}_X, \HT^{B_{dR}}$ and $\Fil_{\Q_p}^C$, there exist Harder-Narasimhan filtrations. The Harder-Narasimhan filtrations are compatible under the equivalence $\HT: \Modif^{ad}_X\lra \HT^{B_{dR}}$.
\begin{theorem}\label{T: HN fil}
	Let $\mathcal{C}$ be one of the categories $\Modif^{ad}_X, \HT^{B_{dR}}, \Fil_{\Q_p}^C$, and $N: \Rep\,G\ra \mathcal{C}$  an exact tensor functor such that $\omega\circ N=\omega^G$, where $\omega: \mathcal{C}\ra \Vect_{\Q_p}$ is the natural functor. There exists a unique $\Q$-filtration $\Fc_N$ of $\omega^G$, such that for any $(V,\rho)\in \Rep\,G$, the induced filtration $\rho_\ast(\Fc_N)$ on $V$ is induced by the Harder-Narasimhan filtration of $N(V)$. 
\end{theorem}
\begin{proof}
	For $\mathcal{C}=\Fil_{\Q_p}^C$, this follows from \cite{DOR} Theorem 5.3.1. For $\mathcal{C}=\Modif^{ad}_X$ or $\mathcal{C}=\HT^{B_{dR}}$, by \cite{CI} Proposition 3.8, since the Harder-Narasimhan filtrations are compatible with tensor products, duals, symmetric and exterior powers, one sees that
	the arguments in the proof\footnote{If the base field is of characteristic $p$, which is not our case here, then one needs the correction as in  \cite{Ans} p. 1239.} of \cite{DOR} Theorem 5.3.1 work here.
	See also \cite{Cor1} Theorem 5.8, Proposition 5.9 and Proposition 4.2.
\end{proof}

Let us explain a little more on the meaning of $\Fc_N$.
Recall that $\omega: \mathcal{C}\ra \Vect_{\Q_p}, Y\mapsto \omega(Y)$ induces a bijection between the set of strictly sub objects of $Y$ and the set of sub objects of $\omega(Y)$, thus we can view the $\Q$-filtration $\Fc_N\in \Fil_{\Q_p}(\omega^G)$ as a filtration of $N$, which we call the Harder-Narasimhan filtration of $N$. 

Consider the case $\mathcal{C}=\Modif^{ad}_X$ in Theorem \ref{T: HN fil}. The functor $N: \Rep\,G\ra \Modif^{ad}_X$ is equivalent to a triple $(\E_1,\E_2,f)$, where $\E_1,\E_2$ are $G$-bundles on $X$ such that $\E_1$ is the trivial $G$-bundle and $f: \E_1|_{X\setminus\{\infty\}} \st{\sim}{\ra}\E_2|_{X\setminus\{\infty\}} $ is an isomorphism.
Let $\ul{\E}=(\E_1,\E_2,f)$ be an admissible modification of $G$-bundles on $X$, with associated $\HT(\ul{\E})$ and $\pi(\ul{\E})\in \Fil_C(\omega^G)$.  
We get Harder-Narasimhan  vectors \[\nu(\pi(\ul{\E}))\in X_\ast(A)_\Q^+\] and
\[\nu(\ul{\E})=\nu(\HT(\ul{\E}))\in  X_\ast(A)_\Q^+.\]
As before, in general $\nu(\ul{\E})\neq \nu(\pi(\ul{\E}))$. In the following, we discuss $\nu(\ul{\E})$ more. By construction, $\nu(\ul{\E})$ comes from the Harder-Narasimhan filtration \[\Fc_{\ul{\E}}\in \Fil_{\Q_p}(\omega^G).\]
We get the associated parabolic $P=P_{\Fc_{\ul{\E}}}$ of $G$ such that the associated Levi $M$ is the centralizer of $\nu(\ul{\E})$. By construction, the vector $\nu(\ul{\E})$ is obtained as the $G(\Q_p)$-conjugacy class of a splitting $\lambda: \mathbb{D}\ra A_M\subset G$ of $\Fc_{\ul{\E}}$. As in \cite{DOR} Theorem 4.2.13, the pair $(P, \lambda)$ uniquely determines $\Fc_{\ul{\E}}$.
For any $(V,\rho)\in \Rep\,G$, we get an induced admissible modification of vector bundles $\ul{\E}_V=(\E_{1,V},\E_{2,V},f_V)$, and  we have \[\rho(\nu(\ul{\E}))=\nu(\ul{\E}_V),\] where $\rho: X_\ast(G)_\Q^\Gamma/G(\Q_p)\ra X_\ast(\GL(V))_\Q^\Gamma/\GL(V)(\Q_p)$ is the induced map.

For any parabolic $P\subset G$, recall that a reduction of $(\E_1,\E_2,f)$ to $P$ is a triple $(\E_{1,P},\E_{2,P},f_P)$ together with an isomorphism \[\iota_P: (\E_{1,P},\E_{2,P},f_P)\times^PG:=(\E_{1,P}\times^PG,\E_{2,P}\times^PG,f_P\times^PG)\st{\sim}{\lra} (\E_1,\E_2,f),\] where $\E_{1,P},\E_{2,P}$ are $P$-bundles on $X$  and $f_P: \E_{1,P}|_{X\setminus\{\infty\}} \st{\sim}{\ra}\E_{2,P}|_{X\setminus\{\infty\}} $ is an isomorphism. Note that by \cite{CFS} Lemma 2.4, given $(\E_1,\E_2,f)$ and $\E_{1,P}$, the data $(\E_{2,P}, f_P)$ is uniquely determined.
We sometimes omit $\iota_P$ and simply say that $(\E_{1,P},\E_{2,P},f_P)$ is a reduction of $(\E_1,\E_2,f)$.  
We call an admissible modification of $G$-bundles $\ul{\E}=(\E_1,\E_2,f)$ \emph{semi-stable} if $\nu(\ul{\E})\in X_\ast(A)^+_\Q$ is \emph{central}.
\begin{lemma}\label{L:semi-stable G bundles}
\begin{enumerate}
	\item The admissible modification of $G$-bundles $\ul{\E}=(\E_1,\E_2,f)$ is  semi-stable if and only if the induced modification of adjoint bundles $\mathrm{Ad}(\ul{\E})=(\mathrm{Ad}(\E_1),\mathrm{Ad}(\E_2),\mathrm{Ad}(f))$ is semi-stable in the sense of subsection \ref{subsection adm modif}.
	\item Let $\rho: G_1\ra G_2$ be a closed embedding of reductive groups, and $\ul{\E}=(\E_1,\E_2,f)$ an admissible modification of $G_1$-bundles. Then the push forward $\rho_\ast(\ul{\E})=(\rho_\ast(\E_1), \rho_\ast(\E_2), \rho_\ast(f))$ semi-stable implies $\ul{\E}$ semi-stable.
\end{enumerate}
\end{lemma}
\begin{proof}
(1) If $\ul{\E}=(\E_1,\E_2,f)$ is  semi-stable, the HN filtration has a unique splitting which factors through $A_G$. Therefore the induced modification $\mathrm{Ad}(\ul{\E})$ has trivial HN filtration, i.e. it is semi-stable. Conversely, if $\mathrm{Ad}(\ul{\E})$ is semi-stable in the sense of subsection \ref{subsection adm modif}, let $\lambda: \mathbb{D}\ra G$ be a splitting of the HN filtration of $\ul{\E}$, then the composition of $\lambda$ with the natural projection $G\ra G_{ad}$ is the splitting of the HN filtration of $\mathrm{Ad}(\ul{\E})$, which is trivial, therefore $\lambda$ factors through $A_G$.

(2) The closed embedding $\rho$ induces an embedding $\Lie\,G \hookrightarrow \Lie\,G'$, and $\mathrm{Ad}(\ul{\E})$ can be viewed as a strict sub object of $\mathrm{Ad}(\rho_\ast(\ul{\E}))$. If $\mathrm{Ad}(\rho_\ast(\ul{\E}))$ is semi-stable, since $\mathrm{Ad}(\ul{\E})$ and $\mathrm{Ad}(\rho_\ast(\ul{\E}))$ are of slope 0, we deduce that $\mathrm{Ad}(\ul{\E})$ is semi-stable. Then we conclude by (1).
\end{proof}

Now we translate the Tannakian description of the Harder-Narasimhan filtration on $\Modif^{ad}_X$ into an internal form. First of all, recall that for any linear algebraic group $G$ over $\Q_p$ and any $G$-bundle $\E$ on $X$, the map
\[X^\ast(G)\ra \Z,\quad \chi\mapsto \deg\,\chi_\ast\E\]defines a
vector \[\deg\,\E\in X_\ast(A_G')_\Q.\] If $G$ is reductive, let \[\mu(\E)\in \pi_1(G)_{\Gamma,\Q}=X_\ast(A_G)_\Q\] be its inverse image under the natural isomorphism $X_\ast(A_G)_\Q\ra X_\ast(A_G')_\Q$. The vectors $\deg\,\E$ and $\mu(\E)$ are called the degree and slope respectively of the $G$-bundle $\E$. If $G=\GL_n$ and $\E$ is viewed as a vector bundle of rank $n$, then the definition $\deg\,\E$ here is compatible with the usual definition of degree for the vector bundle $\E$.
Note that the slope $\mu(\E)$ is given by either the image of $\nu(\E)$ under $\Nc(G)\ra \pi_1(G)_{\Gamma,\Q}$, or the image of $c_1^G(\E)$ under $\pi_1(G)_\Gamma\ra \pi_1(G)_{\Gamma,\Q}$.
If $\ul{\E}=(\E_1,\E_2, f)$ is an admissible modification of $G$-bundles, we define $\deg(\ul{\E})=\deg\,\E_2$, and let $\mu(\ul{\E})$ be its inverse image in $X_\ast(A_G)_\Q$.
\begin{proposition}\label{P: HN reduction}
Let $\ul{\E}=(\E_1,\E_2,f)$ be an admissible modification of $G$-bundles. Then there exits a unique pair \[(P, v),\] where $P$ is a standard parabolic of $G$ with  associated Levi $M$, and $v\in X_\ast(A_M)_\Q$ with $\lan v, \alpha\ran >0$ for all $\alpha\in \Delta_{0,P}$, such that the following holds: 

Let $(\E_{1,P},\E_{2,P},f_P)$ be the reduction of $(\E_1,\E_2,f)$ to $P$ such that $\E_{1,P}$ is trivial and set \[(\E_{1,M},\E_{2,M},f_M)=(\E_{1,P},\E_{2,P},f_P)\times^PM,\] the induced admissible modification of $M$-bundles. Then  $(\E_{1,M},\E_{2,M},f_M)$ is semi-stable of slope $v$.

\end{proposition}
\begin{proof}
Let $P=P_{\Fc_{\ul{\E}}}$ be the standard parabolic associated to the Harder-Narasimhan filtration $\Fc_{\ul{\E}}$
of $\ul{\E}$ and $v=\nu(\ul{\E})\in X_\ast(A)^+_\Q$ the HN vector, then $(P,v)$ is uniquely determined by Theorem \ref{T: HN fil}. From $P$, we consider the unique admissible reduction
$(\E_{1,P},\E_{2,P},f_P)$  of  $(\E_1,\E_2,f)$ to $P$, then as $v$ comes from a splitting $\lambda: \mathbb{D}\ra A_M$, $P=P_\lambda$ and $v=\nu(\ul{\E})$ is central in $M$, we have $v\in X_\ast(A_M)_\Q$ and $\lan v, \alpha\ran >0$ for all $\alpha\in \Delta_{0,P}$, i.e. $v\in X_\ast(A_M)_\Q^+$ . 

To show
that $\ul{\E}_M=(\E_{1,M},\E_{2,M},f_M)$ is semi-stable of slope $v$, we first note that this is true for $G=\GL_n$, since in this case the proposition is just a reformulation of the Harder-Narasimhan filtration for admissible modifications of vector bundles in subsection \ref{subsection adm modif}. For the general case, take a faithful representation $\rho: G\ra G'=\GL(V)$ and consider the induced filtration $\Fc'=\rho_\ast(\Fc_{\ul{\E}})\in \Fil_{\Q_p}(G')$ with associated parabolic $P'$ and Levi $M'$.  Then the induced morphism $\rho_M: M\ra M'$ is injective, cf. \cite{DOR} p. 135. Moreover, if $\ul{\E}'$ denotes the push forward of $\ul{\E}$ to $G'$, with associated $\ul{\E}'_{P'}$ and $\ul{\E}'_{M'}$ we have $\rho_\ast(\ul{\E}_M)=\ul{\E}'_{M'}$. Since $\ul{\E}'_{M'}$ is semi-stable of slope $\rho_\ast(v)$, by Lemma \ref{L:semi-stable G bundles} (2), $\ul{\E}_M$ is semi-stable. As $\rho_M\circ\mu(\ul{\E}_M)=\rho_\ast(v)$ and $\rho_M$ is injective, we get $\mu(\ul{\E}_M)=v$.

\end{proof}

From this proposition, we get the following characterization of semi-stable admissible modifications.
\begin{proposition}\label{P: semi-stable adm modif}
Let $\ul{\E}=(\E_1,\E_2,f)$ be an admissible modification of $G$-bundles. It is semi-stable if and only if for any standard parabolic subgroup $P\subset G$ and any $\chi\in X^\ast(P/Z_G)^+$, we have \[\deg\,\chi_\ast\E_{2,P}\leq 0,\] where $\E_{2,P}$ is the reduction to $P$ of $\E_2$ determined by $\ul{\E}$ and the trivial $P$-bundle $\E_{1,P}$.
\end{proposition}
\begin{proof}
Note that $\ul{\E}$ is semi-stable if and only if the HN parabolic $P=G$. Then this proposition follows immediately from Proposition \ref{P: HN reduction}. 
\end{proof}
Let us explain a little more on the notations.
For any standard parabolic $P\supset P_0$ with associated standard Levi $M\supset M_0$, we view \[X^\ast(P/Z_G)\subset X^\ast(P)=X^\ast(M)=X^\ast(M_{ab})\subset X^\ast(Z_M),\] where $M_{ab}$ is the maximal abelian quotient of $M$ and $Z_M\ra M_{ab}$ is the natural isogeny. From the set $\Delta_P^\vee$, we get the following dominant set  \[X^\ast(P)^+=X^\ast(M)^+=\{\chi\in X^\ast(Z_M)|\,\lan \chi, \alpha^\vee\ran \geq 0,\, \forall\, \alpha^\vee\in \Delta_P^\vee\}\] and \[X^\ast(P/Z_G)^+:=X^\ast(P/Z_G)\cap X^\ast(P)^+.\] Similarly, we have \[X^\ast(P/Z_G)^\Gamma\subset X^\ast(P)^\Gamma=X^\ast(M)^\Gamma=X^\ast(A_M')\subset X^\ast(A_M)\] and $A_M\ra A_M'$ is the natural isogeny. Using the set $\Delta_{0,P}^\vee$, we define similarly $X^\ast(P)^{\Gamma,+}$ and $X^\ast(P/Z_G)^{\Gamma,+}$. We remark that in Proposition \ref{P: semi-stable adm modif}, it suffices to consider for any $\chi\in X^\ast(P/Z_G)^{\Gamma,+}$.

Here is another usual form of the above proposition:
\begin{corollary}\label{C: semi-stable adm modif}
	Let $\ul{\E}=(\E_1,\E_2,f)$ be an admissible modification of $G$-bundles. It is semi-stable if and only if for any maximal standard parabolic subgroup $P\subset G$, we have
	\[\lan\mu(\E_{2,M}), \alpha\ran\leq 0,\]
	where $\E_{2,M}=\E_{2,P}\times^PM$ and $\alpha$ is the unique element of $\Delta_{0,P}$.
\end{corollary}
\subsection{Moduli of local $G$-Shtukas}\label{subsection local shtukas}
As before, $G$ is a connected reductive group over $\Q_p$.
Let $\{\mu\}$ be the conjugacy class of cocharacters $\mu: \G_{m,\ov{\Q}_p}\ra G_{\ov{\Q}_p}$. Fixing a Borel subgroup $B\subset G_{\ov{\Q}_p}$ containing a maximal torus $T$. The class $\{\mu\}$ defines an element $\mu\in X_\ast(T)^+$ for the choice of $B$. We view it as an element in $X_\ast(G)_\Q/G$.
Then we have the associated flag variety $\Fl(G,\mu)$ over a finite extension $E=E(G,\{\mu\})$ of $\Q_p$. Recall that we have a natural map $\Fil_{\ov{\Q}_p}(\omega^G)\ra X_\ast(G)_\Q/G$, sending a filtration to its type.
By construction, \[\Fl(G,\mu)(\ov{\Q}_p)=G(\ov{\Q}_p)/P_\mu(\ov{\Q}_p)=\{\Fc\in \Fil_{\ov{\Q}_p}(\omega^G)\;\tr{of type}\; \mu\},\]where $P_\mu$ is the parabolic subgroup of $G_{\ov{\Q}_p}$ associated to $\mu$ by the formula \[P_\mu=\{g\in G_{\ov{\Q}_p}|\lim_{t\to 0}\mu(t)g\mu(t)^{-1} \;\tr{exists}\}.\]
In particular $P_\mu\supset B$.
In the following sections 3-5, we will assume that $\mu$ is minuscule and work with the associated $p$-adic flag varieties $\Fl(G,\mu)$ and $\Fl(G,\mu^{-1})$.

For an arbitrary $\mu$, we will need the $B_{dR}^+$-affine Schubert cell $\Gr_\mu$, which is a \emph{diamond} over $E$, see \cite{SW} 19.2, 20.2 and the following subsection \ref{subsection affine schubert}. There is a morphism of diamonds \[\pi_\mu: \Gr_\mu\ra \Fl(G,\mu)^\Diamond,\] which is an \emph{isomorphism} if $\mu$ is \emph{minuscule}, see \cite{CS} Proposition 3.4.3, Theorem 3.4.5, \cite{SW} Proposition 19.4.2 and the following \ref{subsection affine schubert}. We have also the $B_{dR}^+$-affine Schubert variety $\Gr_{\leq\mu}=\coprod_{\mu'\leq \mu}\Gr_{\mu'}$, which is a proper diamond over $E$.
\\

A local Shtuka datum\footnote{In this paper we only consider local Shtuka data with one conjugacy class $\{\mu\}$.} (cf. \cite{SW} 23.1) is a triple $(G,\{\mu\}, [b])$, where
\begin{itemize}
	\item $G$ is a connected reductive group over $\Q_p$,
	\item $\{\mu\}$ is a conjugacy class of cocharacter $\mu: \G_{m,\ov{\Q}_p}\ra G_{\ov{\Q}_p}$,
	\item $[b]\in B(G)$ is a $\sigma$-conjugacy class of $b\in G(\breve{\Q}_p)$ such that $[b]\in B(G, \mu)$.
\end{itemize}
If moreover $\mu$ is minuscule, then $(G,\{\mu\}, [b])$ is called a local Shimura datum (cf. \cite{RV} Definition 5.1). 

Let $(G,\{\mu\}, [b])$ be a local Shtuka datum and fix a representative $b\in G(\breve{\Q}_p)$.
Attached to the triple $(G,\{\mu\}, b)$, we have the moduli space of local $G$-Shtukas with one leg (cf. \cite{SW} sections 12-14 and the appendix to section 19) with infinite level (cf. \cite{SW} section 23) \[\Sht(G,\mu,b)_\infty,\] which is a diamond over $\breve{E}$, and up to isomorphism, all of which depend only on $(G,\{\mu\}, [b])$. By construction, there exist two natural morphisms of diamonds
\[\pi_{dR}: \Sht(G,\mu,b)_\infty\ra \Gr_\mu,\quad \tr{and}\quad \pi_{HT}: \Sht(G,\mu,b)_\infty\ra \Gr_{\mu^{-1}},\]
which factor through certain subspaces $\Gr_\mu^a\subset \Gr_\mu$ (see subsection \ref{subsection Newt HN de Rham}) and $\Gr_{\mu^{-1}}^{Newt=[b]}\subset \Gr_{\mu^{-1}}$ (see subsection \ref{subsection Newt HN Hodge-Tate}) respectively. We call $\pi_{dR}$ (resp. $\pi_{HT}$) the de Rham (resp. Hodge-Tate) period morphism.
By \cite{SW} subsection 23.3,
$\Sht(G,\mu,b)_\infty$ classifies 
\begin{itemize}
	\item either
modifications of $G$-bundles of type $\mu$ between $\E_b$ and $\E_1$ over $\Gr_\mu^a$,
\item or modifications of $G$-bundles of type $\mu^{-1}$ between $\E_1$ and $\E_b$ over $\Gr_{\mu^{-1}}^{Newt=[b]}$. 
\end{itemize} 
We get the following diagram of de Rham and Hodge-Tate period morphisms:
\[\xymatrix{ &\Sht(G,\mu, b)_\infty\ar@{->>}[ld]_{\pi_{dR}}\ar@{->>}[rd]^{\pi_{HT}}&\\
	\Gr_\mu^a& &\Gr_{\mu^{-1}}^{Newt=[b]}.
}\]
The morphism $\pi_{dR}$ is a $\ul{G(\Q_p)}$-torsor, while $\pi_{HT}$ is a $\mathrm{Aut}(\E_b)$-torsor.

One can replace $\Gr_\mu$ by $\Gr_{\leq\mu}$ in the above construction to get the diamond $\Sht(G,\leq\mu, b)_\infty$, which is exactly the version of moduli space of local $G$-Shtukas with one leg bounded by $\mu$ studied in \cite{SW}.

If $\mu$ is minuscule, we have $\Sht(G,\leq\mu, b)_\infty=\Sht(G,\mu, b)_\infty$, and we will
also use the notation $\M(G,\mu,b)_\infty$ for $\Sht(G,\mu,b)_\infty$. In this case $\Gr_\mu^a\simeq\Fl(G,\mu,b)^{a,\Diamond}$ and $\Fl(G,\mu,b)^a\subset \Fl(G,\mu)$ is the admissible locus introduced in \cite{CFS} Definition 3.1. 

\section{Newton strata and Harder-Narasimhan strata on $p$-adic flag varieties}\label{section Newt HN flag}

We keep our notations and let $(G,\{\mu\})$ be as before. 
Recall the Kottwitz set (cf. \cite{Kot2} section 6, here we use the notation of \cite{CFS} 2.1) \[B(G,\mu)=\{[b]\in B(G)\;| \; \nu([b])\leq \mu^\diamond,\quad \kappa([b])=\mu^\sharp\}.\] We have also the set (cf. \cite{RV} 2.2) \[A(G,\mu)=\{[b]\in B(G)\;| \; \nu([b])\leq \mu^\diamond\}.\]Both $B(G,\mu)$ and $A(G,\mu)$ are  finite subsets of $B(G)$, equipped with the induced partial order $\leq$.

In the rest of this section, we will mainly consider the induced conjugacy class $\{\mu^{-1}\}$ instead.
Let $\Fl(G,\mu^{-1})$ be the associated flag variety defined over $E=E(G,\{\mu^{-1}\})$, which we consider as an \emph{adic space}.
 We are interested in the geometry of the $p$-adic flag variety $\Fl(G,\mu^{-1})$ from the point of view of $p$-adic Hodge theory. After reviewing the Newton stratification introduced in \cite{CS},  we define and study the Harder-Narasimhan strata of the $p$-adic flag variety $\Fl(G,\mu^{-1})$, following the lines in \cite{DOR} chapter VI, but via modifications of $G$-bundles on the Fargues-Fontaine curve.  These strata generalize the Harder-Narasimhan strata in the case of $\GL_n$ studied by Fargues in \cite{F2}. We assume that $\mu$ is \emph{minuscule} in this section.

\subsection{Newton strata}\label{subsection newt}
We first consider Newton strata. 
Let $C|E$ be an algebraically closed perfectoid field,  and $\E$ be a $G$-bundle on the Fargues-Fontaine curve $X=X_{C^\flat}$.  Since $\mu$ is minuscule, by \cite{CS} 3.4.5, \cite{F5} 4.2 and \cite{F4} 3.20, for any $x\in \Fl(G,\mu^{-1})(C,\Ol_C)$ we can associate to it a modification \[\E_x\] of $\E$ at $\infty$ (see also subsection \ref{subsection Hecke stack}). Consider the case $\E=\E_1$, the trivial $G$-bundle. By Theorem \ref{T: G-bundles}, the isomorphism class of $\E_{1,x}$ defines a point $b(\E_{1,x})\in B(G)$. Letting $C$ vary, we get a map \[Newt: |\Fl(G,\mu^{-1})|\lra B(G).\]We can determine the image of $Newt$ as follows.
\begin{proposition}\label{T: Newton strata}
\begin{enumerate}
\item We have the following decomposition of $\Fl(G,\mu^{-1})$ into locally closed subsets over $E$:
\[ \Fl(G,\mu^{-1})=\coprod_{[b]\in B(G,\mu)}\Fl(G,\mu^{-1})^{Newt=[b]},\] such that for $x\in \Fl(G,\mu^{-1})(C,\Ol_C)$, we have \[x\in \Fl(G,\mu^{-1})^{Newt=[b]}(C,\Ol_C)\quad\Leftrightarrow\quad \E_{1,x}\simeq \E_b.\] The open stratum is associated to the unique basic element $[b_0]\in B(G,\mu)$.
Each stratum $\Fl(G,\mu^{-1})^{Newt=[b]}$ is stable under the $G(\Q_p)$-action on $ \Fl(G,\mu^{-1})$.
\item  We have the following dimension formula: for $[b]\in B(G,\mu)$,
\[ \dim \,\Fl(G,\mu^{-1})^{Newt=[b]}=\langle \mu-\nu([b]), 2\rho\rangle,\]where $\rho$ is the half of the sum of positive roots of $G$.
\end{enumerate}
\end{proposition}
\begin{proof}
(1) follows from \cite{CS} Proposition 3.5.7, Corollary 3.5.9 and \cite{R} Proposition A.9. The fact that each stratum is locally closed comes from the upper semi-continuity of the Newton map (cf. \cite{KL} and \cite{SW} subsection 22.5).

(2) follows from the theory of local Shimura varieties and \cite{BF} Lemma 3.2.5 (see also \cite{CS} Proposition 4.2.23 for the PEL case). More precisely, consider the local Shimura datum $(G,\{\mu\}, [b])$. Fix a representative $b\in G(\breve{\Q}_p)$ of $[b]$. We have the associated local Shimura variety at infinite level $\M(G,\mu,b)_\infty$, which fits into the following diagram
\[\xymatrix{ &\M(G,\mu, b)_\infty\ar@{->>}[ld]_{\pi_{dR}}\ar@{->>}[rd]^{\pi_{HT}}&\\
	\Fl(G,\mu,b)^{a,\Diamond}& &\Fl(G,\mu^{-1})^{Newt=[b],\Diamond},
}\]
where $\pi_{dR}$ is the Hodge-de Rham period map, which is a $\ul{G(\Q_p)}$-torsor, and $\pi_{HT}$ is the Hodge-Tate period map, which is a $\wt{J_{b}}$-torsor. Here $\wt{J_{b}}=Aut(\E_b)$ and we have $\dim\, \wt{J_{b}}=\lan \nu([b]), 2\rho\ran$ (cf. \cite{F4} for example). As $\mu$ is minuscule, $\dim\,\Fl(G,\mu)=\lan \mu, 2\rho\ran$. As $\pi_{dR}$ is pro-\'etale and $\Fl(G,\mu,b)^a\subset \Fl(G,\mu)$ is open, $\dim\,\M(G,\mu, b)_\infty=\dim\,\Fl(G,\mu,b)^a=\lan \mu, 2\rho\ran$. 
Then by \cite{BF} Lemma 3.2.5 $\dim \,\Fl(G,\mu^{-1})^{Newt=[b]}=\langle \mu-\nu([b]), 2\rho\rangle$ (see also \cite{CFS} Proposition 5.3). 
\end{proof}
\begin{remark}
We call the decomposition in the above theorem the Newton stratification. By the recent work of Viehmann \cite{Vieh},
the closure relation holds. 
\end{remark}
\subsection{Harder-Narasimhan strata}
Recall the finite subsets $B(G,\mu)\subset A(G,\mu)\subset B(G)$. Let $\pi_1(G)_{\Gamma,tors}$ be the torsion subgroup of $\pi_1(G)_{\Gamma}$. There is a map (see \cite{RV} (2.10)) $c: A(G,\mu)\ra \pi_1(G)_{\Gamma,tors},\quad [b]\mapsto \kappa([b])-\mu^\sharp$. By definition, $B(G,\mu)=c^{-1}(0)$. Now consider the images $\nu(B(G,\mu))\subset\nu(A(G,\mu))$ under the Newton map $\nu: B(G)\ra \Nc(G)$. For later use, we denote $\Nc(G,\mu)=\nu(B(G,\mu))\subset \Nc(G)$, the set of Newton vectors of elements of  $B(G,\mu)$. In \cite{CFS} Corollary 4.7, we gave an internal description of the set $\Nc(G,\mu)$ (using roots and weights). One may wonder whether there is an external (Tannakian) description. Since $\pi_1(\GL_n)_{\Gamma,tors}=0$, we have $B(G,\mu)= A(G,\mu)$ for $G=\GL_n$.  It turns out that the Tannakian description in general only holds for $\nu(A(G,\mu))$.
\begin{lemma}\label{L: index set HN}
	Let $v\in \Nc(G)$. We have
\begin{enumerate}
	\item  $v\in \tr{Im}\,\nu$ if and only if for any representation $(V,\rho)\in\Rep\,G$ we have $\rho(v)\in \tr{Im}\,\nu_{\GL(V)}$.
	\item  $v\in \nu(A(G,\mu))$ if and only if for any representation $(V,\rho)\in\Rep\,G$ we have $\rho(v)\in \Nc(\GL(V),\rho\circ\mu)$.
\end{enumerate}
\end{lemma}
\begin{proof}
(1) The only if part follows from the functoriality of the slope map $\nu: B(\cdot)\ra \Nc(\cdot)$. The if part follows from the Tannakian definition of $\nu$, cf. \cite{Kot1} 4.2.

(2) The only if part follows from the functorialities of the slope map $\nu: B(\cdot)\ra \Nc(\cdot)$ and the Kottwitz map $\kappa: B(\cdot)\ra \pi_1(\cdot)_\Gamma$ and the properties of the partial order on $B(G)$ and $\Nc(G)$. To show the if part, by (1) we have found $[b]\in B(G)$ such that $\nu([b])=v$ and $v\leq \mu^\diamond$ by the properties of the partial order. Then by definition we have $[b]\in A(G,\mu)$. Thus $v=\nu([b])\in \nu(A(G,\mu))$.
\end{proof}

Now we consider Harder-Narasimhan stratifications. Let $C|E$ be an algebraically closed perfectoid field. Applying Theorem \ref{T: HN fil} to the admissible modification $(\E,\E',f)$ with
 $\E=\E_1$ and $\E'=\E_{1,x}$ for a point $x\in \Fl(G, \mu^{-1})(C,\Ol_C)$, we get a well defined map \[ \Fl(G, \mu^{-1})(C,\Ol_C)\lra \Nc(G), \quad x\longmapsto \nu(\E_{1}, \E_{1,x}, f).\]
 We denote $\nu(\E_{1}, \E_{1,x}, f)^\ast=w_0(-\nu(\E_{1}, \E_{1,x}, f))$.
\begin{proposition}\label{P: HN and Newt}
	 For any $x\in \Fl(G, \mu^{-1})(C,\Ol_C)$,  we have:
	 \begin{enumerate}
	 	\item The inequality of elements in $\Nc(G)$:
	 \[\nu(\E_{1}, \E_{1,x}, f)\leq \nu(\E_{1,x}).\]
\item	The Harder-Narasimhan vector $\nu(\E_1, \E_{1,x}, f)^\ast$  lies in $\mathcal{N}(G,\mu)$. 
\end{enumerate}
\end{proposition}
\begin{proof}
(1) By Theorem \ref{T: HN fil} and \cite{RR} Lemma 2.2 (see also \cite{DOR} Proposition 6.3.9), it suffices to show that for any $(V,\rho)\in \Rep\,G$, \[\nu(\E_{1,V}, \E_{1,x,V}, f_V)\leq \nu(\E_{1,x,V})\in \Nc(\GL(V)).\] This is exactly \cite{CI} Proposition 3.5. See also \cite{F2} Proposition 14.

(2) By (1), we have $\nu(\E_1, \E_{1,x}, f)^\ast\leq \nu(\E_{1,x})^\ast:=w_0(-\nu(\E_{1,x}))$, which implies that for any $(V,\rho)\in \Rep\,G$, $\rho(\nu(\E_1, \E_{1,x}, f)^\ast)\in \Nc(\GL(V),\rho\circ\mu)$. By Lemma \ref{L: index set HN}, $\nu(\E_1, \E_{1,x}, f)^\ast\in \nu(A(G,\mu))$.

We claim that in fact $\nu(\E_1, \E_{1,x}, f)^\ast\in \Nc(G,\mu)\subset \nu(A(G,\mu))$. This follows from the intrinsic description of the Harder-Narasimhan vector $v=\nu(\E_1, \E_{1,x}, f)$ as in \ref{subsection G-str} and the proof of Theorem \ref{T: property HN strata} below. Indeed, let $P$ be the standard parabolic associated to the HN vector $v$ with corresponding standard Levi $M$. Let $\lambda \in X_\ast (T)^+_M$ be the $M$-dominant cocharacter such that $\lambda =w(\mu^{-1})$ for some $w\in W$ and $x\in P(C)wP_{\mu^{-1}}(C)/P_{\mu^{-1}}(C)$. Then $v$ is the image of $\lambda$ under the composition of natural maps $X_\ast(T)_\Q\ra X_\ast(A_P)_\Q\hookrightarrow X_\ast(A)_\Q$. On the other hand, let $\lambda^\sharp\in \pi_1(M)_\Gamma$ be the image of $\lambda$, which then corresponds a basic element $[b']\in B(M)_{basic}$ under the bijection $\kappa_M: B(M)_{basic}\st{\sim}{\ra} \pi_1(M)_\Gamma$. Let $[b]\in B(G)$ be the image of $[b']$ under the natural map $B(M)\ra B(G)$. Then by construction $\nu([b])=v\leq -w_0(\mu^\diamond)$ and $\kappa([b])=-\mu^\sharp$, i.e. $-w_0(v)\in \Nc(G,\mu)$.

\end{proof}

For any $x\in \Fl(G,\mu^{-1})(C,\Ol_C)$ we write \[HN(x)=\nu(\E_1, \E_{1,x}, f)^\ast.\]
Letting $C$ vary,
we get the following map on topological spaces
$HN: |\Fl(G,\mu^{-1})|\lra  \mathcal{N}(G,\mu)$.
\begin{theorem}\label{T: HN strata Tate}
	The map $HN$ is upper semi-continuous, that is, for any $v\in \mathcal{N}(G,\mu)$, the subset \[\Fl(G,\mu^{-1})^{HN\geq v}:=\{x\in |\Fl(G,\mu^{-1})|\,| HN(x)\geq v\}\] is closed. In particular, the subset \[\Fl(G,\mu^{-1})^{HN= v}:=\{x\in |\Fl(G,\mu^{-1})|\,| HN(x)= v\}\] is locally closed.
\end{theorem}
\begin{proof}
For any $x\in \Fl(G,\mu^{-1})(C,\Ol_C)$, since $\mu$ is minuscule, we have
 \[\nu(\E_1, \E_{1,x}, f)=\nu(\Fc_x)\] with $\Fc_x\in \Fil_C(\omega^G)$ attached to $x$. Indeed, this follows from the descriptions of $\nu(\E_1, \E_{1,x}, f)$ and $\nu(\Fc_x)$ by HN types as in our later Theorem \ref{T: HN semi cont} and \cite{DOR} chapter VI.3 respectively, and the fact that if $\mu$ is minuscule, then the Bialynicki-Birula map is an isomorphism, and moreover the generalized semi-infinite orbits stratification on the $B_{dR}^+$-affine Schubert cell $\Gr_{\mu^{-1}}$ agrees with the Bruhat decomposition on the flag variety (cf. subsection \ref{subsection semi-infinite}). The reader can accept this fact on first reading. Then the arguments in the proof of
\cite{DOR} Theorem 6.3.5 (see also the proof of the following Theorem \ref{T: property HN strata}) and Proposition 6.3.12 apply to the $p$-adic setting.
\end{proof}

In the following, we will identify $\mathcal{N}(G,\mu)$ with $B(G,\mu)$ by the Newton map. For $x\in\Fl(G,\mu^{-1})(C,\Ol_C)$ we will also  write $HN(x)=b(\E_1, \E_{1,x}, f)\in B(G,\mu)$.
We have the following stratification over $E$:
\[\Fl(G,\mu^{-1})=\coprod_{[b]\in B(G,\mu)}\Fl(G,\mu^{-1})^{HN=[b]}.\]
For any $[b]\in B(G,\mu)$, the stratum $\Fl(G,\mu^{-1})^{HN=[b]}$ is a locally closed subspace of $\Fl(G,\mu^{-1})$, and it is stable under the action of $G(\Q_p)$ on $\Fl(G,\mu^{-1})$.

Let $[b_0]\in B(G,\mu)$ be the basic element. Then the stratum \[\Fl(G,\mu^{-1})^{HN=[b_0]}\] is open, which is also called the semi-stable locus of $\Fl(G,\mu^{-1})$.
We have the following description for $\Fl(G,\mu^{-1})^{HN=[b_0]}$, which is similar to \cite{CFS} Proposition 2.7 (but here we don't need the assumption that $G$ is quasi-split). 
\begin{proposition}\label{P: Hodge-Tate wa}
	Let $x\in \Fl(G,\mu^{-1})(C,\Ol_C)$. Then
	$x\in \Fl(G,\mu^{-1})^{HN=[b_0]}(C,\Ol_C)$ if and only if for any standard parabolic $P$ and any $\chi\in X^\ast(P/Z_G)^+$, we have
	\[\deg\, \chi_\ast (\E_{1,x})_P\leq 0, \]where $(\E_{1,x})_P$ is the reduction of $\E_{1,x}$ to $P$ induced by the reduction $\E_{1_P}$ of $\E_1$ to $P$.
\end{proposition}
\begin{proof}
	This is essentially a reformulation of Proposition \ref{P: semi-stable adm modif} (see also \cite{DOR} Corollary 5.2.10.).
	Indeed, consider the Schubert cell decomposition\footnote{Since the Schubert cell decomposition exists on the algebraic varieties level, we omit $\Ol_C$ here to simplify the notations.}
	\[\Fl(G,\mu^{-1})(C)=\coprod_{w\in W_P\setminus W/W_{P_{\mu^{-1}}}}\Fl(G,\mu^{-1})(C)^w, \]
	where \[\Fl(G,\mu^{-1})(C)^w=P(C)wP_{\mu^{-1}}(C)/P_{\mu^{-1}}(C)=P(C)/(P(C)\cap P_{\mu^{-1,w}}(C))=:\Fl(P,\mu^{-1,w})(C).\] 
	Projection to the Levi quotient $M$ of $P$ induces an affine fibration:
	\[\pr_w: \Fl(P,\mu^{-1,w})(C)\ra \Fl(M,\mu^{-1,w})(C).\]
	Now \[(\E_{1,x})_P\times^P M\simeq \E_{1_M,\pr_w(x)},\]
	and one can argue as in the proof of \cite{CFS} Proposition 2.7 (see also the proof of Proposition \ref{P: semi-stable adm modif}).
\end{proof}
\begin{remark}
We note that in the above proposition, for each $P$ it suffices to consider the subset $\Delta_{0,P}\subset X^\ast(P/Z_G)^{\Gamma,+}\subset X^\ast(P/Z_G)^+$. In fact, it suffices to consider all maximal parabolic subgroups $P$, in which case each $\Delta_{0,P}$ consists of only one element.
\end{remark}

\begin{remark}[\cite{DOR} Theorem 6.2.8] \label{T: GIT}
Fix an invariant inner product on $G$ and let $\mathcal{L}$ be the corresponding ample homogeneous $\Q$-line bundle on $\Fl(G,\mu^{-1})$ (cf. \cite{DOR} p. 146). Let $K$ be a field extension of $E$ and $x\in \Fl(G,\mu^{-1})(K,\Ol_K)$. Then we have the following GIT description for $\Fl(G,\mu^{-1})^{HN=[b_0]}$:
\[x\in \Fl(G,\mu^{-1})^{HN=[b_0]}(K,\Ol_K) \quad\Longleftrightarrow\quad \,\forall\, \lambda: \G_m\ra G_{der},\quad \mu^\mathcal{L}(x,\lambda)\geq 0.\]
\end{remark}

The following theorem gives some basic properties of the Harder-Narasimhan stratification.
\begin{theorem}[\cite{F2} Conjecture 2 (1)]\label{T: property HN strata}
		For any non basic $[b]\neq [b_0]$, the stratum $\Fl(G,\mu^{-1})^{HN=[b]}$ is a parabolic induction.
\end{theorem}
\begin{proof}
We may assume that $\Fl(G,\mu^{-1})^{HN=[b]}\neq \emptyset$.
We adapt the arguments in \cite{DOR}. Fix a minimal parabolic subgroup $P_0$ with Levi subgroup $M_0$ as above Proposition \ref{P: Hodge-Tate wa}. Let $T$ be a fixed maximal torus in $M_0$ defined over $\Q_p$.
We introduce a finite set $\Theta(G,\mu)$ which is the set of pairs $(P,\nu_P)$ with $P$ a standard parabolic subgroup of $G$ and $\nu_P\in X_\ast(T)_\Q/W_{P}$, satisfying the following two conditions:
\begin{enumerate}
	\item $\nu_P\equiv \mu^{-1}\, \mod\, W$,
	\item Let $\mu(\nu_P)\in X_\ast(A_P)_\Q$ be the image of $\nu_P$ under $X_\ast(T)_\Q\ra X_\ast(A_P')_\Q\st{\sim}{\ra} X_\ast(A_P)_\Q$. Then $\lan \mu(\nu_P),\alpha\ran >0,\quad \forall\, \alpha\in \Delta_{0,P}$.
\end{enumerate}
A such pair $(P,\nu_P)$  is called a HN type. Let $\mathcal{H}(G,\mu)$ be the set of HN vectors which contribute in the HN stratification. Then we have an inclusion $\mathcal{H}(G,\mu)\hookrightarrow \Nc(G,\mu) $ by Proposition \ref{P: HN and Newt}.
We have also a natural surjective map \[H: \Theta(G,\mu)\twoheadrightarrow \mathcal{H}(G,\mu), \quad (P,\nu_P)\mapsto \mu(\nu_P)\]sending a HN type to its HN vector. In the following we fix a finite extension $\wt{E}$ of $E$ which splits $G$ and base change everything to $\wt{E}$. We will denote by the same notations over $\wt{E}$. Similar to \cite{DOR} p. 152 (and p. 280-281), we have  a refinement of the Harder-Narasimhan stratification \[\Fl(G,\mu^{-1})=\coprod_{\theta\in \Theta(G,\mu)}\Fl(G,\mu^{-1})^\theta,\]
which is $G(\Q_p)$-equivariant
and such that \[\Fl(G,\mu^{-1})^{HN=v}=\coprod_{\theta\in \Theta(G,\mu), H(\theta)^\ast=v} \Fl(G,\mu^{-1})^\theta.\] Fix a  HN type $\theta=(P,\nu_P)\in \Theta(G,\mu)$. Consider the $P$-orbits in the flag variety $\Fl(G,\mu^{-1})=G/P_{\mu^{-1}}$. Then $\nu_P$ determines a unique Schubert cell \[\Fl(P,\nu_P)=PwP_{\mu^{-1}}/P_{\mu^{-1}}\] where $w\in W_P\setminus W/W_{P_{\mu^{-1}}}$ such that $\nu_P=\mu^{-1,w}$. By abuse of notation, we still denote $w$ the minimal length representative in the corresponding coset $W_PwW_{P_{\mu^{-1}}}$.
Let $M$ be the standard Levi of $P$ with induced $\nu_M$. Then the natural projection
\[\Fl(P,\nu_P)\ra  \Fl(M,\nu_M)\]is an affine bundle of rank $\ell(w)$. Set \[\Fl(P,\nu_P)^\theta=\Fl(G,\mu^{-1})^\theta\cap \Fl(P,\nu_P).\] The $G(\Q_p)$-action restricts to an action of $P(\Q_p)$ on $\Fl(P,\nu_P)^\theta$. Let $\Fl(M,\nu_M)^{ss}$ be the open HN stratum for the flag variety $\Fl(M,\nu_M)$.
Then the above projection $\Fl(P,\nu_P)\ra  \Fl(M,\nu_M)$ restricts to an affine fibration of rank $\ell(w)$
\[\Fl(P,\nu_P)^\theta\ra  \Fl(M,\nu_M)^{ss}.\]
We have a homeomorphism \[\Fl(P,\nu_P)^\theta\times^{P(\Q_p)}G(\Q_p)\st{\sim}{\lra} \Fl(G,\mu^{-1})^\theta. \]
Thus the stratum $\Fl(G,\mu^{-1})^\theta$ is an affine bundle of rank $\ell(w)$ over $\Fl(M,\nu_M)^{ss}\times^{P(\Q_p)}G(\Q_p)$. We deduce that for any $v\in \Nc(G,\mu)$, the stratum $\Fl(G,\mu^{-1})^{HN=v}$ is a parabolic induction. 

\end{proof}

\begin{remark}
	We know the dimension formula for the basic stratum, since it is open.
	For any non basic $[b]\neq [b_0]$, if the stratum $\Fl(G,\mu^{-1})^{HN=[b]}\neq \emptyset$, then by the above proof we have \[\dim\,\Fl(G,\mu^{-1})^{HN=[b]}=\max_{w}\;\lan \mu^{-1,w}, 2\rho_M\ran +\ell(w),\]
	where $M=M_{v}\subset P=P_{v}$ with $v=w_0(-\nu([b]))$ and $w$ runs through the set $w\in {}^PW^{P_{\mu^{-1}}}$ such that $\lan \mu^{-1,w},\alpha\ran >0$ for any $\alpha\in \Delta_{0,P}$. Here we view $\mu^{-1,w}\in X_\ast(A_P)_\Q$ under the above map $X_\ast (T)_\Q\ra X_\ast(A_P)_\Q$. 
	In fact,  Conjecture 2 (2) of \cite{F2} predicts that for any $[b]\in B(G,\mu)$ such that the stratum $\Fl(G,\mu^{-1})^{HN=[b]}\neq \emptyset$, we have \[\dim\,\Fl(G,\mu^{-1})^{HN=[b]}=\lan \mu-\nu([b]), 2\rho\ran.\] This is verified in the case $G=\GL_n$ by Fargues in \cite{F2} Proposition 23. The above theorem was also proved by Fargues in \cite{F2} Propositions 21 and 22 in the case $G=\GL_n$ by a different method.
\end{remark}

\begin{remark}
By Theorems \ref{T: property HN strata}  and \ref{T: GIT}, we can calculate the $\ell  (\neq p)$-adic cohomology of $\Fl(G,\mu^{-1})^{HN=[b]}$. Indeed, by \ref{T: property HN strata}  it suffices to consider the open stratum $\Fl(G,\mu^{-1})^{HN=[b_0]}$. By the GIT description in \ref{T: GIT}, we can follow \cite{DOR} chapter VII.2 to calculate the Euler-Poincar\'e characteristic, and \cite{O1} to calculate the individual cohomology groups. 
\end{remark}

\begin{remark}
For any $[b]\in B(G,\mu)$ with the associated HN stratum $\Fl(G,\mu^{-1})^{HN=[b]}$, we neither know its non-emptiness\footnote{If $[b]$ is basic, then the associated stratum is open and non-empty. Thus the non-emptiness is a problem on non-basic strata.  For the case of $\GL_n$, see \cite{O4} for a complete solution. For the general case, see \cite{DOR} Remark 9.6.3 for some hints.}, nor the closure relation. If $\mu$ is non minuscule, then in \cite{DOR} the authors there gave a counter example for the closure relation, see loc. cit. Example 2.3.7. 
\end{remark}

\subsection{Newton strata vs Harder-Narasimhan strata}\label{subsection Newton vs HN}

By \cite{HN} Theorem 0.1, there exists a unique maximal element $[b_1]$ in $B(G,\mu)$ for the partial order $\leq$. If $G$ is quasi-split, then $[\nu_{b_1}]=\mu^\diamond$. In the general case, this is \emph{not true}, see \cite{HN1} Example 3.1.
We call the stratum \[\Fl(G,\mu^{-1})^{HN=[b_1]} \quad (\tr{resp.} \quad \Fl(G,\mu^{-1})^{Newt=[b_1]})\]
 the $\mu$-ordinary Harder-Narasimhan (resp. Newton) stratum. Both of the $\mu$-ordinary strata $\Fl(G,\mu^{-1})^{HN=[b_1]}$  and $\Fl(G,\mu^{-1})^{Newt=[b_1]}$ are \emph{closed} in $\Fl(G,\mu^{-1})$, by the semi-continuity of the maps $Newt$ and $HN$.
 Proposition \ref{P: HN and Newt} implies that we have the inclusion \[\Fl(G,\mu^{-1})^{HN=[b_1]}\subset \Fl(G,\mu^{-1})^{Newt=[b_1]}.\]
\begin{proposition}[\cite{F2} Conjecture 1 (2)]\label{P: Newt HN ordinary}
	Assume that $G$ is quasi-split.
	Then we have always \[\Fl(G,\mu^{-1})^{HN=[b_1]}=\Fl(G,\mu^{-1})^{Newt=[b_1]}.\]
	In particular $\Fl(G,\mu^{-1})^{HN=[b_1]}\neq \emptyset$ in this case.
\end{proposition}
\begin{proof}
Let $C|E$ be any algebraically closed perfectoid field.
We have to show that for any point $x\in \Fl(G,\mu^{-1})(C,\Ol_C)$ such that $\nu(\E_{1,x})^\ast=[\nu_{b_1}]$, then $\nu(\E_1,\E_{1,x},f)^\ast=[\nu_{b_1}]$. Since $G$ is quasi-split, $[\nu_{b_1}]=\mu^\diamond$. Then this follows from \cite{CI} Proposition 3.9.
\end{proof}

Let $[b_0]\in B(G,\mu)$ be the unique \emph{basic} element. By the last two subsections, we have the \emph{open} subspaces $\Fl(G,\mu^{-1})^{Newt=[b_0]}$ and $\Fl(G,\mu^{-1})^{HN=[b_0]}$ of $\Fl(G,\mu^{-1})$. Proposition \ref{P: HN and Newt} implies that we have the inclusion \[\Fl(G,\mu^{-1})^{Newt=[b_0]}\subset\Fl(G,\mu^{-1})^{HN=[b_0]}.\]
In section \ref{section full HN}, we will classify the case when the following equality holds \[\Fl(G,\mu^{-1})^{Newt=[b_0]}=\Fl(G,\mu^{-1})^{HN=[b_0]}.\]

\section{Dualities for Newton and HN stratifications}\label{section twin towers}
Let
$[b]\in B(G)_{basic}$ be a \emph{basic} element. Fix a representative $b\in G(\breve{\Q}_p)$ of $[b]$. We have the associated reductive group $J_b$ over $\Q_p$, which is an inner form of $G$.  Fix an isomorphism $J_{b,\breve{\Q}_p}\cong G_{\breve{\Q}_p}$. Let $\Bun_G$ be the groupoid of $G$-bundles on the Fargues-Fontaine curve, cf. \cite{F4} section 2, which is a \emph{small $v$-stack} (over $\F_p$) in the sense of \cite{S}.

\subsection{The twin towers principle}\label{subsection twin towers}
In \cite{CFS} 5.1, we have introduced the so called
``twin towers principle", which is the following isomorphism
$$
\Bun_{J_b} \cong \Bun_G,
$$
that is to say there is an equivalence of groupoids between $G$-bundles and $J_b$-bundles on the curve.
In fact, $J_b\times X$ is the twisted pure inner form of $G\times X$ obtained by twisting by the $G$-torsor $\E_b$,
$$
J_b \times X = \underline{\mathrm{Aut}} (\E_b)
$$
as a group over the curve.
If $\E$ is a $G$-bundle on $X$ one associates to it the $J_b$-bundle
$$
\underline{\text{Isom}} (\E_b,\E).
$$
At the level of points,  the preceding isomorphism of small $v$-stacks gives us the bijection
$$
B(J_b)\st{\sim}{\lra} B(G)
$$
that sends $[1]$ to $[b]$ and $[b^{-1}]$ to $[1]$. Here $[b^{-1}]\in B(J_b)$ is the class defined by \[b^{-1}\in J_b(\breve{\Q}_p)=G(\breve{\Q}_p).\]
In fact, we have the following commutative diagrams on the compatibilities for Newton maps and Kottwitz maps:
\[\xymatrix{
B(J_b)\ar[d]_{\nu_{J_b}}\ar[r]^\sim&B(G)\ar[d]^{\nu_G}\\
\Nc(H)\ar[r]^{\cdot \nu([b])} & \Nc(H),
}
\qquad \xymatrix{
B(J_b)\ar[d]_{\kappa_{J_b}}\ar[r]^\sim&B(G)\ar[d]^{\kappa_G}\\
\pi_1(H)_\Gamma\ar[r]^{ +\kappa([b])} & \pi_1(H)_\Gamma.
}\]
Let us make a comment on the notations. Here we have identified $\Nc(G)=\Nc(J_b)=\Nc(H)$ and $\pi_1(G)_\Gamma=\pi_1(J_b)_\Gamma=\pi_1(H)_\Gamma$, where $H$ is a fixed quasi-split inner form of $G$ (and thus of $J_b$).
Recall that $\pi_1(G)_\Gamma$ is an abelian group, for which we will write the group law additively and the identity as 0; on the other hand, $\Nc(G)\subset X_\ast(G)_\Q/G$, the later has a commutative ordered monoid structure, and we will write its semi-group low multiplicatively.


The isomorphism $\Bun_{J_b}\cong\Bun_G$ respects modifications of a given type $\mu$, that is to say it identifies the corresponding Hecke stacks of modifications (see subsection \ref{subsection Hecke stack}). Let $\{\mu\}$ be a conjugacy class of cocharacter $\mu: \G_{m,\ov{\Q}_p}\ra G_{\ov{\Q}_p}$. In the rest of this section, we
assume that $[b]\in B(G,\mu)$. The isomorphism $J_{b, \ov{\Q}_p}\cong G_{\ov{\Q}_p}$ induces a conjugacy class of cocharacter $\{\mu\}$ of $J_b$. Then 
$$[b^{-1}]\in B(J_b,\mu^{-1})$$ 
is the basic element (in $B(J_b)$), and $[b^{-1}]\mapsto [1]$ via the above bijection $B(J_b)\xrightarrow{\sim} B(G)$. 
One thus has $$J_{b^{-1}}\cong G.$$ 
Recall that in \cite{CFS} 4.1 we have introduced the following generalized Kottwitz sets 
\[B(G,0,\nu_{b}\mu^{-1}):=\{[b']\in B(G)\;|\,\kappa([b'])=0,\quad \nu([b'])\leq \nu([b])w_0(-\mu^\diamond)\}\] and  \[B(J_b,0,\nu_{b^{-1}}\mu):=\{[b'']\in B(J_b)\;|\,\kappa([b''])=0,\quad \nu([b''])\leq \nu([b^{-1}])\mu^\diamond\},\] 
which are finite subsets of $B(G)$ and $B(J_b)$ respectively. They contain the trivial classes $[1]\in  B(G)$ and $[1]\in B(J_b)$ respectively.
One checks directly the following lemma:
\begin{lemma}\label{L: bijections Kottwitz sets}
	The bijection $B(J_b)\st{\sim}{\lra} B(G)$ induces the following bijections:
	\[B(J_b,\mu^{-1})\st{\sim}{\lra} B(G,0,\nu_{b}\mu^{-1}), \quad B(J_b,0,\nu_{b^{-1}}\mu)\st{\sim}{\lra}B(G,\mu).\]
\end{lemma}

Let $E=E(G,\{\mu\})$ and $\breve{E}=\wh{E^{ur}}$ be the completion of the maximal unramified extension of $E$. 
Consider the following $p$-adic flag varieties (as adic spaces) over $\breve{E}$: \[\Fl(G,\mu),\quad \Fl(G,\mu^{-1}), \quad \Fl(J_b,\mu), \quad \tr{and}\quad \Fl(J_b,\mu^{-1}).\] 
We have identifications:
\[\Fl(G,\mu)=\Fl(J_b,\mu),\quad \Fl(G,\mu^{-1})=\Fl(J_b,\mu^{-1}).\]
To summarize, we have the following data:
\begin{itemize}
	\item the triples $(G,\{\mu^{-1}\}, [1])$ and $(J_b,\{\mu\}, [1])$ (which we call the \emph{Hodge-Tate side} for $G$ and $J_b$ respectively), 
	\item the local Shtuka data $(G,\{\mu\}, [b])$ and $(J_b, \{\mu^{-1}\}, [b^{-1}])$ (which we call the \emph{de Rham side} for $G$ and $J_b$ respectively).
\end{itemize}

In the rest of this section, we will assume that $\mu$ is \emph{minuscule}.

\subsection{Newton strata on the de Rham side}\label{subsection newt de rham}
In the last section, we studied the geometry of $\Fl(G,\mu^{-1})$ by modifications of the trivial $G$-bundle $\E_1^G$.
Now we study the flag variety $\Fl(G,\mu^{-1})=\Fl(J_b,\mu^{-1})$ by modifications of the $J_b$-bundle $\E_{b^{-1}}^{J_b}$.
From the local Shimura datum $(J_b,\{\mu^{-1}\},[b^{-1}])$, in \cite{CFS} subsection 5.3,  we have constructed  a stratification of $\Fl(J_b,\mu^{-1})$ by locally closed subsets
\[ \Fl(J_b,\mu^{-1})=\coprod_{[b']\in B(J_b,0,\nu_{b^{-1}}\mu)}\Fl(J_b,\mu^{-1},b^{-1})^{Newt=[b']},\]
which we call the Newton\footnote{In \cite{CFS} this is called the Harder-Narasimhan stratification. Here we change the terminology and modify the notation, since later we will introduce another stratification with the same index set, which we will call the Harder-Narasimhan stratification following \cite{DOR}, as an analogy of that introduced in last section.} stratification.
Let $C|\breve{E}$ be an algebraically closed perfectoid field.
For any point $x\in \Fl(J_b,\mu^{-1})(C,\Ol_C)$, we get a modification $\E_{b^{-1},x}^{J_b}$ of the $J_b$-bundle $\E_{b^{-1}}^{J_b}$ on the Fargues-Fontaine $X=X_{C^\flat}$. Then by definition \[x\in \Fl(J_b,\mu^{-1},b^{-1})^{Newt=[b']}(C,\Ol_C) \quad \Longleftrightarrow\quad b(\E_{b^{-1},x}^{J_b})=[b'].\]
We have the associated $p$-adic period domain
 \[\Fl(J_b,\mu^{-1},b^{-1})^a:=\Fl(J_b,\mu^{-1},b^{-1})^{Newt=[1]},\]which is the maximal open stratum.
 
 Similarly, starting from the local Shimura datum $(G,\{\mu\},[b])$ we can study the geometry of $\Fl(G,\mu)$ by modifications of the $G$-bundle $\E_b^G$. More precisely, we have the Newton stratification
 \[ \Fl(G,\mu)=\coprod_{[b']\in B(G,0,\nu_{b}\mu^{-1})}\Fl(G,\mu,b)^{Newt=[b']},\]
 and the associated $p$-adic period domain
 \[\Fl(G,\mu,b)^a:=\Fl(G,\mu,b)^{Newt=[1]}.\]
 
 Recall that inside $\Fl(G,\mu^{-1})$ and $\Fl(J_b, \mu)$, we have respectively the open Newton strata $\Fl(G,\mu^{-1})^{Newt=[b]}$ and $\Fl(J_b,\mu)^{Newt=[b^{-1}]}$ introduced in subsection \ref{subsection newt}. 
\begin{lemma}\label{L: dual open Newton}
Under the identification $\Fl(G,\mu^{-1})=\Fl(J_b,\mu^{-1})$, we have \[\Fl(J_b,\mu^{-1},b^{-1})^a=\Fl(G,\mu^{-1})^{Newt=[b]}.\]
Similarly, under the identification $\Fl(G,\mu)=\Fl(J_b,\mu)$, we have \[\Fl(G,\mu,b)^a =\Fl(J_b,\mu)^{Newt=[b^{-1}]}.\]
\end{lemma}
\begin{proof}
We only check the identity $\Fl(J_b,\mu^{-1},b^{-1})^a=\Fl(G,\mu^{-1})^{Newt=[b]}$.
Let $C$ be any algebraically closed complete extension of $\breve{E}$ and let  $x\in \Fl(G,\mu^{-1})(C,\Ol_C)=\Fl(J_b,\mu^{-1})(C,\Ol_C)$. Then  we have
\[\begin{split} x\in \Fl(J_b,\mu^{-1},b^{-1})^a(C,\Ol_C)&\quad\Leftrightarrow\quad \E_{b^{-1},x}^{J_b}=\E_1^{J_b}\\ &\quad\Leftrightarrow\quad \E_{1,x}^G=\E_b^G\\ &\quad\Leftrightarrow\quad x\in \Fl(G,\mu^{-1})^{Newt=[b]}(C,\Ol_C). \end{split}\]
\end{proof}

\subsection{Dualities for local Shimura varieties}\label{subsection twin loc shimura}
Consider the local Shimura variety with infinite level
\[\M (G,\mu,b)_\infty,\]which is the moduli space classifying
 \begin{itemize}
 	\item either
 modifications of type $\mu$  between $\E^G_b$ and $\E^G_1$ over $\Fl(G,\mu,b)^a$,
 \item or modifications of type $\mu^{-1}$ between $\E^G_1$ and $\E^G_b$ over $\Fl(G,\mu^{-1})^{Newt=[b]}=\Fl(J_b,\mu^{-1},b^{-1})^a$.
\end{itemize}
 Similarly, we have the local Shimura variety with infinite level \[\M (J_b,\mu^{-1},b^{-1})_\infty,\] which is the moduli space classifying
\begin{itemize}
	\item either
modifications of type $\mu^{-1}$ between $\E^{J_b}_{b^{-1}} $ and $\E^{J_b}_{1}$ over $\Fl(J_b,\mu^{-1},b^{-1})^a$,
\item or
modifications of type $\mu$ between $\E^{J_b}_{1}$  and $\E^{J_b}_{b^{-1}} $ over $\Fl(J_b,\mu)^{Newt=[b^{-1}]}=\Fl(G,\mu,b)^a$.
\end{itemize}
The twin tower principle induces a $J_b(\Q_p)\times G(\Q_p)$-isomorphism of local Shimura varieties with infinite level (\cite{FalTwin, FTwin}, \cite{SW1} section 7, \cite{SW} Corollary 23.2.3.)
$$
\M (G,\mu,b)_\infty \st{\sim}{\lra} \M (J_b,\mu^{-1},b^{-1})_\infty
$$
as diamonds on $\Spa (\breve{E})^\diamond$. This fits into a twin towers diagram 
using the de Rham and Hodge-Tate period morphisms that allow us to collapse each tower on its base 
$$
\begin{tikzcd}[row sep=large,column sep=large]
\M (G,\mu,b)_\infty \ar[d,"\pi_{dR}", two heads]\ar[r,"\sim"]  \ar[rd,"\pi_{HT}" description]
\ar[d, dash, dotted , bend right=40, start anchor={[xshift=-12mm]}, end anchor={[xshift=-12mm]}, start anchor={[yshift=3mm]}, end anchor={[yshift=-3mm]},"G(\Q_p)"']
& \M (J_b,\mu^{-1},b^{-1})_\infty \ar[d,"\pi_{dR}", two heads] 
\ar[ld,"\pi_{HT}" description]    \ar[d, dash, dotted, bend left=40, start anchor={[xshift=12mm]}, end anchor={[xshift=12mm]}, start anchor={[yshift=3mm]}, end anchor={[yshift=-3mm]},"J_b(\Q_p)"]
\\
\Fl(G,\mu,b)^{a,\Diamond} & \Fl(J_b,\mu^{-1},b^{-1})^{a,\Diamond}.   
\end{tikzcd}
$$

\subsection{Harder-Narasimhan strata on the de Rham side}\label{subsection HN de Rham}
Now we continue to look at the $p$-adic flag variety $\Fl(J_b,\mu^{-1})$. In \cite{DOR} chapter IX.6, Dat-Orlik-Rapoport introduced a stratification of $\Fl(J_b,\mu^{-1})$ by locally closed subsets (indexed by Harder-Narasimhan vectors)
\[\Fl(J_b,\mu^{-1})=\coprod_{v\in \H(\mathbb{J},\mu^{-1})}\Fl(J_b,\mu^{-1},b^{-1})^{HN=v}, \]which they called the Harder-Narasimhan stratification. Here $\mathbb{J}$ is the augmented group attached to $J_b$ and $b^{-1}$ as in \cite{DOR} Example 9.1.22. Let $C|\breve{E}$ be an algebraically closed perfectoid field.
For any $x\in \Fl(J_b,\mu^{-1})(C,\Ol_C)$, we have the modification triple $(\E_{b^{-1},x}^{J_b}, \E_{b^{-1}}^{J_b},f)$ of $J_b$-bundles on $X=X_{C^\flat}$.
We write \[\nu(\E_{b^{-1},x}^{J_b}, \E_{b^{-1}}^{J_b},f)=\nu(N_{b^{-1}},\Fc_x)\] for the filtered $F$-isocrystal with $J_b$-structure $(N_{b^{-1}},\Fc_x)$ attached to $(\E_{b^{-1},x}^{J_b}, \E_{b^{-1}}^{J_b},f)$ constructed in subsection \ref{subsection G-str}.
\begin{proposition}\label{P: HN and Newt de Rham}
For any $x\in \Fl(J_b,\mu^{-1})(C,\Ol_C)$, 
\begin{enumerate}
\item we have the following inequality in $\Nc(J_b)$
\[\nu(\E_{b^{-1},x}^{J_b}, \E_{b^{-1}}^{J_b},f)\leq \nu(\E_{b^{-1},x}^{J_b}).\]
\item The Newton map for $J_b$ induces an injection \[ \H(\mathbb{J},\mu^{-1}) \hookrightarrow B(J_b,0,\nu_{b^{-1}}\mu).\]
\end{enumerate}
\end{proposition}
\begin{proof}
Under the bijection $B(J_b)\st{\sim}{\ra}B(G), \quad [b^{-1}]\mapsto [1]$ and the identification $\Fl(J_b,\mu^{-1})=\Fl(G,\mu^{-1})$, we have \[\nu(\E_{b^{-1},x}^{J_b}, \E_{b^{-1}}^{J_b},f)=\nu_b\nu(\E_{1,x},\E_1,f_1)=\nu_b\nu(\E_1, \E_{1,x},f^{-1}_1)\] (for the second ``$=$'', see subsection \ref{subsection adm modif}) and $\nu(\E_{b^{-1},x}^{J_b})=\nu_b\nu(\E_{1,x})$. Since $[b]$ is basic, we have \[\nu_b\nu(\E_1, \E_{1,x},f_1^{-1})\leq \nu_b\nu(\E_{1,x}) \Leftrightarrow \nu(\E_1, \E_{1,x},f_1^{-1})\leq \nu(\E_{1,x}).\] Therefore (1) is equivalent to Proposition \ref{P: HN and Newt} (1).
The proof of (2) is similar, which is equivalent to Proposition \ref{P: HN and Newt} (2) (using Lemma \ref{L: bijections Kottwitz sets}).
\end{proof}

We get the composition \[ |\Fl(J_b,\mu^{-1})|\ra \H(\mathbb{J},\mu^{-1})\hookrightarrow B(J_b,0,\nu_{b^{-1}}\mu) \]and we write $b(\E_{b^{-1},x}^{J_b}, \E_{b^{-1}}^{J_b},f)\in B(J_b,0,\nu_{b^{-1}}\mu)$. 
Therefore, starting from the local Shimura datum $(J_b,\{\mu^{-1}\},[b^{-1}])$, for the flag variety $\Fl(J_b,\mu^{-1})$,  we have the Harder-Narasimhan stratification:
\[ \Fl(J_b,\mu^{-1})=\coprod_{[b']\in B(J_b,0,\nu_{b^{-1}}\mu) }\Fl(J_b,\mu^{-1},b^{-1})^{HN=[b']}.\]

Similarly, starting from the local Shimura datum $(G,\{\mu\},[b])$, for the flag variety $\Fl(G,\mu)$, we have the Harder-Narasimhan stratification:
\[ \Fl(G,\mu)=\coprod_{[b']\in B(G,0,\nu_{b}\mu^{-1})}\Fl(G,\mu,b)^{HN=[b']}.\]
The open Harder-Narasimhan stratum $\Fl(G,\mu,b)^{HN=[1]}$ corresponds to the trivial element $[1]\in B(G,0,\nu_{b}\mu^{-1})$, which is also denoted by (cf. \cite{RZ} chapter 1)
\[\Fl(G,\mu,b)^{wa}:=\Fl(G,\mu,b)^{HN=[1]}.\]
Moreover, by Proposition \ref{P: HN and Newt de Rham} (1) (applied to $(G,\{\mu\},[b])$), we have
\[\Fl(G,\mu,b)^a\subset \Fl(G,\mu,b)^{wa}.\]
Alternatively, the above inclusion also follows from the theorem of Colmez-Fontaine (cf. \cite{FF} chapter 10). Our argument above shows that it is equivalent to the inclusion $\Fl(G,\mu^{-1})^{Newt=[b]}\subset \Fl(G,\mu^{-1})^{HN=[b]}$, see subsection \ref{subsection Newton vs HN}.

\subsection{Dualities for Newton and Harder-Narasimhan stratifications}
Consider the $p$-adic flag variety $\Fl(J_b,\mu^{-1})$. Starting from the datum $(J_b, \{\mu^{-1}\}, [b^{-1}])$ (\emph{de Rham side for the group $J_b$}), by subsection \ref{subsection HN de Rham}
 we have the Harder-Narasimhan stratification: 
\[ \Fl(J_b,\mu^{-1})=\coprod_{[b']\in B(J_b,0,\nu_{b^{-1}}\mu) }\Fl(J_b,\mu^{-1},b^{-1})^{HN=[b']}.\]
By subsection \ref{subsection newt de rham}, we have also the Newton stratification:
\[ \Fl(J_b,\mu^{-1})=\coprod_{[b']\in B(J_b,0,\nu_{b^{-1}}\mu) }\Fl(J_b,\mu^{-1},b^{-1})^{Newt=[b']}.\]
Recall  the Harder-Narasimhan and Newton stratifications for $\Fl(G,\mu^{-1})$ introduced in section \ref{section Newt HN flag} starting from the datum $(G,\{\mu^{-1}\},[1])$ (\emph{Hodge-Tate side for the group $G$}):
\[\Fl(G,\mu^{-1})=\coprod_{[b']\in B(G,\mu)} \Fl(G,\mu^{-1})^{HN=[b']}, \quad \Fl(G,\mu^{-1})=\coprod_{[b']\in B(G,\mu)} \Fl(G,\mu^{-1})^{Newt=[b']}.\]
We have the following generalization of Lemma \ref{L: dual open Newton},
which says that under the twin towers principle, the corresponding Harder-Narasimhan and Newton stratifications introduced in section \ref{section Newt HN flag} and here are identical. 
\begin{theorem}\label{T: duality HN Newt}
Under the identification \[\Fl(G,\mu^{-1})=\Fl(J_b,\mu^{-1}),\] 
for any $[b']\in B(G,\mu)$ corresponding to $[b'']\in B(J_b,0,\nu_{b^{-1}}\mu)$ under the bijection (cf. Lemma \ref{L: bijections Kottwitz sets}) \[B(G,\mu)\st{\sim}{\lra} B(J_b,0,\nu_{b^{-1}}\mu),\] we have
\begin{enumerate}
	\item 
$ \Fl(G,\mu^{-1})^{HN=[b']}=\Fl(J_b,\mu^{-1},b^{-1})^{HN=[b'']}.$
\item $ \Fl(G,\mu^{-1})^{Newt=[b']}=\Fl(J_b,\mu^{-1},b^{-1})^{Newt=[b'']}.$
\end{enumerate}
\end{theorem}
\begin{proof}
The proof for (2) is identical with the proof for Lemma \ref{L: dual open Newton}, which is in fact also \cite{CFS} Proposition 5.3.

The proof for (1) is in fact similar, which follows from the functoriality of the Harder-Narasimhan filtrations and the morphisms $HN$:
let $C$ be any algebraically closed complete extension of $\breve{E}$, we have the following commutative diagram 
\[\xymatrix{
\Fl(J_b,\mu^{-1})(C,\Ol_C)\ar[r]^\sim\ar[d]_{HN_{J_b}}& \Fl(G,\mu^{-1})(C,\Ol_C)\ar[d]^{HN_{G}}\\
\Nc(H)\ar[r]^{\cdot \nu([b])}&\Nc(H)
}\]
where $H$ is a fixed quasi-split inner form of $G$,
see \cite{DOR} p. 252 (3.3), Proposition 9.5.3 (iii) and Remarks 9.6.18 (ii).

\end{proof}

Similarly, starting from $(G,\{\mu\},[b])$ (\emph{de Rham side for the group $G$}),
for the flag variety $\Fl(G,\mu)$, we have the Harder-Narasimhan stratification (see subsection \ref{subsection HN de Rham})
\[ \Fl(G,\mu)=\coprod_{[b']\in B(G,0,\nu_{b}\mu^{-1})}\Fl(G,\mu,b)^{HN=[b']}\]and the Newton stratification (see subsection \ref{subsection newt de rham})
\[ \Fl(G,\mu)=\coprod_{[b']\in B(G,0,\nu_{b}\mu^{-1})}\Fl(G,\mu,b)^{Newt=[b']}\]
introduced in this section.
Recall also the Harder-Narasimhan and Newton stratifications for $\Fl(J_b,\mu)$ in section \ref{section Newt HN flag} starting from the datum $(J_b,\{\mu\}, [1])$ (\emph{Hodge-Tate side for the group $J_b$}):
\[\Fl(J_b,\mu)=\coprod_{[b']\in B(J_b,\mu^{-1})}\Fl(J_b,\mu)^{HN=[b']},\quad  \Fl(J_b,\mu)=\coprod_{[b']\in B(J_b,\mu^{-1})}\Fl(J_b,\mu)^{Newt=[b']}.\]
The following corollary is clear now.
\begin{corollary}\label{C: duality HN Newt}
Under the identification \[\Fl(J_b,\mu)=\Fl(G,\mu),\] 
for any $[b']\in B(G,\mu)$ corresponding to $[b'']\in B(J_b,0,\nu_{b^{-1}}\mu)$ under the bijection (cf. Lemma \ref{L: bijections Kottwitz sets})  \[B(J_b,\mu^{-1})\st{\sim}{\lra} B(G,0,\nu_{b}\mu^{-1}),\] we have
\begin{enumerate}
	\item
$\Fl(J_b,\mu)^{HN=[b']}=\Fl(G,\mu,b)^{HN=[b'']}.$
\item $\Fl(J_b,\mu)^{Newt=[b']}=\Fl(G,\mu,b)^{Newt=[b'']}.$
\end{enumerate}
\end{corollary}

\section{Fully Hodge-Newton decomposable case}\label{section full HN}
We keep the notations of the last section. Let $(G,\{\mu\},[b])$ be a local Shimura datum such that $[b]\in B(G,\mu)$ is \emph{basic}. In particular $\mu$ is \emph{minuscule}. We get the dual local Shimura datum $(J_b, \{\mu^{-1}\}, [b^{-1}])$.

Recall that (cf. \cite{GHN} Definition 2.1 and \cite{CFS} 4.3) we have the notion of \emph{fully Hodge-Newton decomposability} for the Kottwitz set $B(G,\mu)$ (or the pair $(G,\{\mu\})$).  Roughly speaking, this means that for any non basic $[b']\in B(G,\mu)$, the pair $([b'],\{\mu\})$ is Hodge-Newton decomposable.

Now we can summarize the various equivalent conditions for fully Hodge-Newton decomposability studied in \cite{CFS} and here.
\begin{theorem}\label{T: fully HN}
The following are equivalent:
\begin{enumerate}
	\item $B(G,\mu)$ is fully Hodge-Newton decomposable.
	\item $B(J_b,\mu^{-1})$ is fully Hodge-Newton decomposable.
	\item  $\Fl(G,\mu,b)^a=\Fl(G,\mu,b)^{wa}$.
	\item  $\Fl(J_b,\mu)^{Newt=[b^{-1}]}=\Fl(J_b,\mu)^{HN=[b^{-1}]}$.
	\item $\Fl(J_b,\mu^{-1},b^{-1})^a=\Fl(J_b,\mu^{-1},b^{-1})^{wa}$.
	\item  $\Fl(G,\mu^{-1})^{Newt=[b]}=\Fl(G,\mu^{-1})^{HN=[b]}$.
\end{enumerate}
\end{theorem}
\begin{proof}
The equivalences $(1)\Leftrightarrow (2)$ follow from \cite{CFS} Corollary 4.16. 
The equivalence $(1)\Leftrightarrow (3)$ was proved in \cite{CFS} Theorem 6.1, thus we get also the equivalence $(2)\Leftrightarrow (5)$. 

The equivalences $(3)\Leftrightarrow(4)$ and $(5)\Leftrightarrow(6)$ follow from Theorem \ref{T: duality HN Newt} and Corollary \ref{C: duality HN Newt} respectively.  Therefore all the above statements are equivalent.

\end{proof}

\begin{remark}\label{R: fully HN}
In the above theorem, the equivalences (1)-(2) are taken from \cite{CFS}, which are purely group theoretical statements. To show the equivalences with the remaining (3)-(6), we have taken \cite{CFS} Theorem 6.1 as one of the key ingredients. On the other hand, one can show\footnote{This is exactly what the author did at the beginning when preparing this article.} the equivalence $(1)\Leftrightarrow (6)$ directly, by using similar (and in fact easier) arguments as in the proof of \cite{CFS} Theorem 6.1. Then using Theorem \ref{T: duality HN Newt} and Corollary \ref{C: duality HN Newt}, we get another proof of \cite{CFS} Theorem 6.1, although essentially the two proofs are the same. We leave the details to the interested reader.
\end{remark}

\begin{remark}
There are some further (conjectural) equivalences for the fully Hodge-Newton decomposable condition (1). For example, we refer the reader to
\begin{enumerate}
	\item \cite{CFS} Conjecture 7.2, in terms of fundamental domains of $p$-adic period domains and local Shimura varieties,
	\item  \cite{GHN} Theorem 2.3, in terms of the geometry of affine Deligne-Lusztig varieties.
\end{enumerate}
\end{remark}

\section{Non minuscule cocharacters}\label{section non minuscule}
In this section, we indicate how to generalize our previous constructions and results to a general (not necessarily minuscule) cocharacter $\mu$.
Roughly, we need to replace flag varieties and local Shimura varieties by the corresponding $B_{dR}^+$-affine Schubert cells and moduli of local $G$-Shtukas respectively. We have Newton and Harder-Narasimhan stratifications on the diamonds $\Gr_\mu$ and $\Gr_{\mu^{-1}}$, generalizing the previous constructions in sections \ref{section Newt HN flag} and \ref{section twin towers}. In particular, the duality results of section \ref{section twin towers} can be generalized to the current setting. We will analyze the geometry of $\Gr_\mu$ using affine Schubert cells of the Levi subgroups, which is in some sense a theory of (generalized) semi-infinite orbits for $B_{dR}^+$-affine Grassmannians. This is
the  key new step to prove the generalization of \cite{CFS} Theorem 6.1. 

\subsection{$B_{dR}^+$-affine Grassmannians and $B_{dR}^+$-affine Schubert varieties}\label{subsection affine schubert}
Let $G$ be a connected reductive group over $\Q_p$. Recall the $B_{dR}^+$-affine Grassmannian $\Gr_G$ is the \emph{small $v$-sheaf} (cf. \cite{SW} 17.2 and \cite{S}) over $\Spd\,\Q_p:=(\Spa \,\Q_p)^\Diamond$ such that for any affinoid perfectoid space $S=\Spa(R,R^+)$ over $\Q_p$, \[\Gr_G(S)=\{(\E,\beta)\}/\simeq\] where
\begin{itemize}
	\item 
	$\E$ is a $G$-torsor over $\Spec B_{dR}^+(R)$, 
	\item $\beta: \E\ra\E^0$ is a trivialization over  $\Spec B_{dR}(R)$; here $\E^0$ is the trivial $G$-torsor, 
\end{itemize}
cf. \cite{F4} 3.1 and \cite{CS} Definition 3.4.1. Equivalently, \[\Gr_G=LG/L^+G,\] where $LG$ and $L^+G$ are the loop groups such that \[LG(\Spa(R,R^+))=G(B_{dR}(R)), \quad \tr{and}\quad L^+G(\Spa(R,R^+))=G(B_{dR}^+(R)).\]
See also \cite{SW} Definition 20.2.1 and Proposition 20.2.2 (where it is called the Beilinson-Drinfeld Grassmannian  over $\Spd\,\Q_p$). By \cite{SW} Lemma 19.1.4, $\Gr_G$ is partially proper. Let $C|\Q_p$ be an algebraically closed perfectoid field and $\Spd \,C:=(\Spa\,C)^\Diamond$. The base change $\Gr_{G,\Spd\, C}$ of $\Gr_G$ to $\Spd\, C$ is given by Definition 19.1.1 of \cite{SW}. 

Let $T\subset B\subset G_{\ov{\Q}_p}$ be a maximal torus inside a Borel subgroup of $G_{\ov{\Q}_p}$. We have the set of dominant cocharacters $X_\ast(T)^+$ of $T$ with respective to $B$, which is a set of representatives for $X_\ast(T)/W$ where $W$ is the absolute Weyl group of $G$. Recall that
we have the Cartan decomposition\footnote{Here we follow \cite{CS} and \cite{CFS} to normalize the sign.}
\[G(B_{dR}(C))=\coprod_{\mu\in X_\ast(T)^+} G(B_{dR}^+(C))\mu(\xi)^{-1}G(B_{dR}^+(C)),\]where $\xi\in B_{dR}^+(C)$ is a fixed uniformizer.
Any $\mu\in X_\ast(T)^+$ defines a closed subfunctor \[\Gr_{\leq \mu}\] of $\Gr_{G, \Spd\, E}$, with an open subfunctor $\Gr_\mu\subset \Gr_{\leq \mu}$,  where $E=E(G,\{\mu\})$ is the field of definition of $\{\mu\}$. By definition (cf. \cite{SW} Definition 19.2.2), $\Gr_{\leq \mu}$ (resp. $\Gr_{\mu}$ ) parametrizes those $(\E,\beta)$ such that over any geometric points $x$, the relative position $Inv(\beta_x)$ is bounded (resp. exactly given) by $\mu$. 
One of the main results of \cite{SW} is the following theorem.
\begin{theorem}[\cite{SW} Theorem 19.2.4, Corollary 19.3.4 and Proposition 20.2.3]
	$\Gr_{\leq \mu}$ is a spatial diamond, and it is proper over $\Spd\, E$. $\Gr_\mu$  is then a partially proper locally spatial diamond. 
\end{theorem}
By definition we have a stratification of diamonds
\[\Gr_{\leq \mu}=\coprod_{\mu'\leq \mu}\Gr_{\mu'}. \]
In particular if $\mu$ is \emph{minuscule}, we have $\Gr_\mu= \Gr_{\leq \mu}$. 

The inclusion $L^+G\subset LG$ induces
a natural action of $L^+G$ on $\Gr_G$. For any perfectoid affinoid $\Q_p$-algebra $(R, R^+)$, let $\xi\in B_{dR}^+(R)$ denote a generator of $\ker\,\theta$, where $\theta: B_{dR}^+(R)\ra  R$ is the canonical surjection (cf. \cite{SW} page 138).
For any $\mu\in X_\ast(T)$, we write $\xi^\mu=\mu(\xi)$ and $t^\mu=\mu(\xi)^{-1}$ for the corresponding elements in $LG$. By abuse of notation we also denote $\xi^\mu$ and $t^\mu$ the associated points in $\Gr_G$.
The diamond $\Gr_\mu$ can be described as usual the orbit $L^+Gt^\mu$,
and we have \[\Gr_\mu\simeq \frac{L^+G}{L^+G \cap t^\mu L^+Gt^{-\mu}}.\]
Note that we have a natural action of $G(\Q_p)$ on $\Gr_\mu$  via the inclusion $G(\Q_p)\subset G(B_{dR}^+(\C_p))$.
Recall that for a diamond $\D$, we have its underlying topological space $|\D|$ (\cite{S} Proposition 11.13 and Definition 11.14). If $\D$ is locally spatial, then the topological space $|\D|$ is locally spectral, cf. \cite{S} Propositions 11.18 and 11.19.  We call a sub diamond $\D'$ is dense in $\D$, if $|\D'|\subset |\D|$ is dense. The dimension of a locally spatial diamond is defined to be the maximal length of a chain of specializations in its underlying locally spectral space. We refer the reader to \cite{Hub} section 1.8, \cite{S} section 21 and \cite{BF} section 3.2 for more discussions on dimensions of locally spatial diamonds.
\begin{proposition}\label{P:affine Schubert}
	\begin{enumerate}
		\item The open sub diamond $\Gr_\mu\subset \Gr_{\leq \mu}$ is dense in $\Gr_{\leq \mu}$.
		\item The dimension of  $\Gr_\mu$ (and thus $\Gr_{\leq \mu}$) is $\lan2\rho, \mu\ran$, where $\rho$ is as usual the half sum of positive (absolute) roots of $G$.
	\end{enumerate}
\end{proposition}
\begin{proof}
	For both statements we may assume that the base field is $\wh{\ov{E}}$.
	
	(1) We imitate the proof of \cite{Zhu} Proposition 2.1.5 (2) in the equal characteristic setting.
	If $\lambda\leq \mu$, then there exists a positive coroot $\alpha$ such that $\mu-\alpha$ is dominant and $\lambda\leq \mu-\alpha\leq \mu$. Thus it suffices to show that $t^{\mu-\alpha}$ is contained in the closure of $\Gr_\mu$. To prove this, we will construct a curve $C\simeq \mathbb{P}^{1,\Diamond}$ in $\Gr_{\leq\mu}$ such that $t^{\mu-\alpha}\in C$ and $C\setminus\{t^{\mu-\alpha}\}\subset \Gr_\mu$. 
	
	For any integer $m$, let $t^{\lambda_m}:=\begin{pmatrix}
	t^m&0\\
	0&1
	\end{pmatrix}$, regarded as an element in $\mathrm{PGL}_2(B_{dR})$. Let $K_m=\mathrm{Ad}_{t^{\lambda_m}}(L^+\SL_2)\subset L\SL_2$. 
	Then \[\sigma_m:=\begin{pmatrix}
	0&-t^m\\t^{-m}&0
	\end{pmatrix}\in K_m.
	\]
	Consider the map $L^+\SL_2\ra\SL_2$ induced by the natural map $\theta: B_{dR}^+(R)\ra R$ for any perfectoid algebra $R$ over $\Q_p$. Let $L^{>0}\SL_2$ be its kernel and set $K_m^{(1)}=\mathrm{Ad}_{t^{\lambda_m}}(L^{>0}\SL_2)$. Then $K_m/K_m^{(1)}\simeq \SL_2$. Let $i_\alpha: \SL_2\ra G$ be the canonical homomorphism associated to $\alpha$. We get the induced map $Li_\alpha: L\SL_2\ra LG$. Let $m=\lan \mu,\alpha\ran -1$ and consider \[C_{\mu,\alpha}:=Li_\alpha(K_m)t^\mu.\] Since $Li_\alpha(K_m^{(1)})\subset L^+G\cap t^\mu L^+Gt^{-\mu}$, $C_{\mu,\alpha}$ is a homogenous space under $K_m/K_m^{(1)}=\SL_2$. One gets then \[C_{\mu,\alpha}\simeq \mathbb{P}^{1,\Diamond},\quad  \tr{and}\quad (L^+G\cap Li_\alpha(K_m))t^\mu\simeq \mathbb{A}^{1,\Diamond}\subset \mathbb{P}^{1,\Diamond}.\] In addition, \[C_{\mu,\alpha}\setminus(L^+G\cap Li_\alpha(K_m))t^\mu =i_\alpha(\sigma_m)t^\mu=t^{\mu-\alpha}L^+G.\]
	Thus $C_{\mu,\alpha}$ is the desired curve.

	(2) Since $\lan2\rho,\mu\ran=\lan  2\rho, -w_0\mu\ran$ and $\dim \Gr_\mu=\dim \Gr_{\mu^{-1}}$ (note that the map $LG\ra LG,\quad g\mapsto g^{-1}$ induces an isomorphism $\Gr_{\mu}\simeq \Gr_{\mu^{-1}}$), we consider $\Gr_{\mu^{-1}}=L^+G\xi^\mu$. Let $\Phi^+$ be the set of positive (absolute) roots of $G$ for the choice of the above Borel subgroup $B\subset G_{\ov{\Q}_p}$.
	Consider the parabolic subgroups $P_\mu$ and $P_{\mu^{-1}}$ defined by the roots $\alpha$ such that $\lan \alpha, \mu\ran \geq 0$ and $\lan \alpha, \mu\ran \leq 0$ respectively. Then $P_{\mu^{-1}}$ is the opposite parabolic of $P_\mu$. Let $U=U_{P_\mu}$ be the unipotent radical of $P_\mu$. Then $U\times P_{\mu^{-1}} \subset G$ defines an open subspace. Consider the associated open functor $L^+(U\times P_{\mu^{-1}} )=L^+U\times L^+P_{\mu^{-1}}  \subset L^+G$. Then since $L^+P_{\mu^{-1}}\subset L^+G\cap \xi^\mu L^+G\xi^{-\mu}$ acts trivially on $\xi^\mu$, we have open functor $L^+U\xi^\mu\subset L^+G\xi^\mu=\Gr_{\mu^{-1}}$. By definition, $U=\prod_{\alpha\in \Phi^+, \lan \alpha,\mu\ran >0}U_\alpha$ with $U_\alpha$ the subgroup of $G$ corresponding to the root $\alpha$. Then \[\begin{split}L^+U\xi^\mu&=(\prod_{\alpha\in \Phi^+, \lan \alpha,\mu\ran >0}L^+U_\alpha)\xi^\mu\\ &=\prod_{\alpha\in \Phi^+, \lan \alpha,\mu\ran >0}(L^+U_\alpha\xi^\mu)\\ &=\prod_{\alpha\in \Phi^+, \lan \alpha,\mu\ran >0} \mathbb{B}_{dR}^+/\xi^{\lan \alpha,\mu\ran}\xi^\mu,\end{split}\]
	where $\mathbb{B}_{dR}^+$ is the functor which sends a perfectoid affinoid $\Q_p$-algebra $(R, R^+)$ to $B_{dR}^+(R)$, and
	the last ``$=$'' comes from the fact that $L^+U_\alpha\simeq \mathbb{B}_{dR}^+$ which acts on $\xi^\mu$ through $\mathbb{B}_{dR}^+/\xi^{\lan \alpha,\mu\ran}$. Moreover the action 
	of $\mathbb{B}_{dR}^+/\xi^{\lan \alpha,\mu\ran}$ on $\xi^\mu$ is free, thus \[L^+U\xi^\mu\simeq \prod_{\alpha\in \Phi^+, \lan \alpha,\mu\ran >0} \mathbb{B}_{dR}^+/\xi^{\lan \alpha,\mu\ran}.\] By \cite{SW} subsection 15.2, for each $\alpha$ as above, the Banach-Colmez space  $\mathbb{B}_{dR}^+/\xi^{\lan \alpha,\mu\ran}$ is a diamond, which is a successive extension of $\mathbb{A}^{1,\Diamond}$. By induction and \cite{BF} Lemma 3.2.5 we have $\dim \mathbb{B}_{dR}^+/\xi^{\lan \alpha,\mu\ran}=\lan \alpha,\mu\ran$. 
	Therefore, \[\dim \Gr_{\mu}=\sum_{\alpha\in \Phi^+, \lan \alpha,\mu\ran >0}\lan \alpha,\mu\ran=\lan 2\rho,\mu\ran.\]
\end{proof}

For $\mu\in X_\ast(T)^+$, let $P_\mu$ be the associated parabolic subgroup of $G_{\ov{\Q}_p}$ as in the above proof, see also subsection \ref{subsection local shtukas}.
Consider the flag variety $\Fl(G,\mu)=G_C/P_\mu$, which is defined over $E$. By \cite{CS} Proposition 3.4.3,Theorem 3.4.5 and \cite{SW} Proposition 19.4.2, there is a natural Bialynicki-Birula map\footnote{Note that according to our convention, here $\pi_{\mu}$ agrees with that in \cite{CS} Proposition 3.4.3, and it is the $\pi_{\mu^{-1}}$ of that in \cite{SW} Proposition 19.4.2.} for diamonds over $E$
\[\pi_\mu: \Gr_\mu\ra \Fl(G,\mu)^\Diamond,\]which is an \emph{isomorphism} if $\mu$ is \emph{minuscule}. Let us recall the definition of $\pi_\mu$. Group theoretically, over $C$ it is the projection
\[\pi_\mu: \Gr_\mu\simeq \frac{L^+G}{L^+G \cap t^\mu L^+G t^{-\mu}}\lra (G_C/P_\mu)^\Diamond \]induced by the projection \[\theta: L^+G(R)=G(B_{dR}^+(R))\ra G(R)\]  for any $C$-perfectoid algebra $R$.
Alternatively, we can give the moduli interpretation as follows. By Tannakian formalism, it is enough to define it for $\GL_n$. In this case, $\mu$ is given by a tuple of integers $(m_1,\dots,m_n)$ with $m_1\geq\dots\geq m_n$. Then $\Gr_\mu$ parametrizes lattices $\Xi\subset B_{dR}(R)^n$ of relative position $(m_1,\dots,m_n)$. For any such lattice, we can define a descending filtration $\Fil_{\Xi}^\bullet$ on the residue $R^n=B_{dR}^+(R)^n/\xi B_{dR}^+(R)^n $ with \[\Fil^i_\Xi=\frac{\xi^i\Xi\cap B_{dR}^+(R)^n}{\xi^i\Xi\cap \xi B_{dR}^+(R)^n}. \]The stabilizer of this filtration defines a parabolic which is conjugate to $P_\mu$. This gives the desired $\pi_\mu: \Gr_\mu\ra \Fl(G,\mu)^\Diamond$. 
From the construction we see that in general, $\pi_\mu$ is surjective, and in fact it is a fibration in diamonds associated to (iterations of) affine spaces.

For $C$-points, recall $\Fl(G,\mu)(C)=\{\Fc\in \Fil_C(\omega^G)\,|\, \Fc \, \tr{has type}\,\mu\}$, where  $\Fil_C(\omega^G)$ is the set of $\Q$-filtrations over $C$ of the standard fiber functor $\omega^G$. The map $\pi_\mu: \Gr_\mu(C,\Ol_C)\ra \Fl(G,\mu)(C)$ sends a $G$-torsor to a ``$G$-filtration''. We can define similarly \[\pi: \Gr_G(C,\Ol_C)\ra \Fil_C(\omega^G),\] such that the following diagram commutes
\[\xymatrix{\Gr_G(C,\Ol_C)\ar[r]^\pi\ar[d]&\Fil_C(\omega^G)\ar[d]\\
	X_\ast(T)^+\ar[r]&X_\ast(G)_\Q/G,
}\]
where the left vertical arrow is given by the Cartan decomposition, the right vertical arrow is given by taking a splitting modulo conjugacy, and the bottom arrow is given by the identifications $X_\ast(T)^+=X_\ast(T)/W=X_\ast(G)/G$ and the inclusion $X_\ast(G)/G\hookrightarrow X_\ast(G)_\Q/G$.
\\

Now let $H$ be an arbitrary linear algebraic group over $\Q_p$. Then we define the $B_{dR}^+$-affine Grassmannian $\Gr_{H}=LH/L^+H$ similarly as above. 
\begin{proposition}\label{P:general aff Grass}
	$\Gr_{H}$ is representable by an ind-diamond, which is ind-proper if $H$ is reductive.
\end{proposition}
\begin{proof}
	As in the proof of \cite{PR} Theorem 1.4, we can take a faithful representation $H\hookrightarrow \GL_n$ such that $\GL_n/H$ is quasi-affine. Then the arguments in the proof of \cite{SW} Lemma 19.1.5 show that the induced map $\Gr_H \ra \Gr_{\GL_n}$ is a locally closed embedding. Since $\Gr_{\GL_n}$  is representable by an ind-diamond by \cite{SW} 19.3, we conclude that $\Gr_{H}$ is also representable by an ind-diamond. In case $H$ is reductive, Theorem 19.2.4 of \cite{SW} implies that it is ind-proper.
\end{proof}

\subsection{Hecke stacks and $B_{dR}^+$-affine Schubert cells}\label{subsection Hecke stack}
Fix a dominant cocharacter $\mu\in X_\ast(T)^+$ and let $E=E(G,\{\mu\})$. We have the Hecke stack \[Hecke^\mu\] over $\ov{\F}_p$ (here we slightly modify the definition in \cite{F4} 3.4): for any $\Spa(R,R^+)\in \Perf_{\ov{\F}_p}$, $Hecke^\mu(\Spa(R,R^+))$ is the groupoid of quadruples $(\E_1,\E_2,D,f)$, where
\begin{itemize}
	\item $\E_1$ and $\E_2$ are $G$-bundles on the relative Fargues-Fontaine curve $X_R$,
	\item $D$ is an effective Cartier divisor of degree 1 on $X_R$,
	\item $f: \E_1|_{X_R\setminus D}\st{\sim}{\lra}\E_2|_{X_R\setminus D}$ is a modification of $G$-bundles, such that the type of $f_x$ is $\mu$ for any geometric point $x=\Spa(C(x), C(x)^+)\ra \Spa(R,R^+)$.
\end{itemize}
This Hecke stack
fits into the following diagram
\[\xymatrix{ & Hecke^\mu\ar[ld]_{\overleftarrow{h}}\ar[rd]^{\overrightarrow{h}}&\\
	\Bun_{G,\ov{\F}_p}& & \Bun_{G,\ov{\F}_p}\times \mathrm{Div}^1,
} \]
where $\mathrm{Div}^1=\Spd\,\breve{\Q}_p/\varphi^\Z$ is the diamond parametrizing degree one divisors on the Fargues-Fontaine curve and \[\overleftarrow{h}(\E_1,\E_2, f, D)=\E_2, \quad \overrightarrow{h}(\E_1,\E_2, f, D)=(\E_1, D). \]The above diagram is the stack version of the diagram in subsection \ref{subsection modif vect}.

Let $[b]\in B(G,\mu)$ be the basic element. Fix a representative $b\in G(\breve{\Q}_p)$ of $[b]$ and we have the reductive group $J_b$. 
Let \[x_1: \Spa(\ov{\F}_p)\ra [\Spa(\ov{\F}_p)/\ul{G(\Q_p)}]\ra \Bun_{G,\ov{\F}_p}\] and \[x_b: \Spa(\ov{\F}_p)\ra [\Spa(\ov{\F}_p)/\ul{J_b(\Q_p)}]\ra \Bun_{G,\ov{\F}_p}\] be the points associated to the classes $[1]$ and $[b]$. Consider the diamonds $\Gr_\mu$ and $\Gr_{\mu^{-1}}$ over $\breve{E}$.
Then we have the following enlarged diagram\footnote{ We can add $\Sht(G,\mu,b)_\infty$ on the top together with the period maps $\pi_{dR}$ and $\pi_{HT}$ to get a further cartesian square and thus a even larger diagram, cf. \cite{F4} 8.2.} where $\Gr_\mu$ and $\Gr_{\mu^{-1}}$ appear:
\[\xymatrix{ 
	& \Gr_{\mu}\ar[ld]\ar[rd]^{i_b} &  & \Gr_{\mu^{-1}}\ar[rd]\ar[ld]_{i_1} &\\
	\Spa(\ov{\F}_p)\ar[rd]^{x_b}	& & Hecke^\mu\ar[ld]_{\overleftarrow{h}}\ar[rd]^{\overrightarrow{h}}& & \Spa(\ov{\F}_p)\ar[ld]_{(x_1 id)}	\\
	&	\Bun_{G,\ov{\F}_p}&  & \Bun_{G,\ov{\F}_p}\times \mathrm{Div}^1, &
} \]
where both the squares are cartesian. In particular, we get
\[\Gr_{\mu^{-1}}\lra \Bun_{G,\ov{\F}_p}\] which is the composition $
\overleftarrow{h}\circ i_1$, and 
\[ \Gr_{\mu}\lra \Bun_{G,\ov{\F}_p}\] which is the composition
$
\pr\circ\overrightarrow{h}\circ i_b$, where $\pr: \Bun_{G,\ov{\F}_p}\times \mathrm{Div}^1\ra \Bun_{G,\ov{\F}_p}$ is the natural projection.

We have also the version of Hecke stack $Hecke^{\leq\mu}$, which can be defined similarly and it is related to $\Gr_{\leq\mu}$ and $\Gr_{\leq\mu^{-1}}$ as above.

\subsection{Generalized semi-infinite orbits}\label{subsection semi-infinite}

Let $P\subset G$ be a parabolic subgroup over $\Q_p$, $M$ a Levi subgroup contained in $P$, which is identified with the reductive quotient of $P$.  Take a maximal torus inside a Borel $T\subset B\subset G_{\ov{\Q}_p}$ and assume $B\subset P_{\ov{\Q}_p}$ and thus $T\subset M_{\ov{\Q}_p}$.
We have the set of dominant cocharacters $X_\ast(T)^+$. Let $B\cap M_{\ov{\Q}_p}$ be the induced Borel of $M_{\ov{\Q}_p}$. Then we get the set of $M$-dominant cocharacters $X_\ast(T)_M^+$. We have the inclusion $X_\ast(T)^+\subset X_\ast(T)_M^+$.

To simplify notations, the base field in this subsection will be $C$, an algebraically closed perfectoid field of characteristic 0 (in fact an extension $F|\Q_p$ which splits $G$ will be enough). In the following we will write $B_{dR}=B_{dR}(C)$.
Consider $B_{dR}^+$-affine Grassmannians $\Gr_M, \Gr_G$ and $\Gr_P$ (cf. Proposition \ref{P:general aff Grass}) over $C$.
The inclusion $P\subset G$ and the projection $P\ra M$ induce the following diagram of $B_{dR}^+$-affine Grassmannians:
\[\xymatrix{ &\Gr_P\ar[ld]_{\pr}\ar[rd]^{i}&\\
	\Gr_M& & \Gr_G.
}\]
Then
the Iwasawa decomposition \[G(B_{dR})=P(B_{dR})G(B_{dR}^+)\] induces a bijection \[i: \Gr_P(C,\Ol_C)=P(B_{dR})/ P(B_{dR}^+)\st{\sim}{\lra} \Gr_G(C,\Ol_C)=G(B_{dR})/G(B_{dR}^+).\]

Let $U_P\subset P$ be the unipotent radical of $P$. Since $G/M$ (resp. $G/U_P$) is affine (resp. quasi-affine), the natural inclusion $M\subset G$ (resp. $U_P\subset G$) induces a closed embedding $\Gr_M\hookrightarrow \Gr_G$ by \cite{SW} Lemma 19.1.5 (resp. a locally closed embedding $\Gr_{U_P}\hookrightarrow \Gr_G$ by the proof of Proposition \ref{P:general aff Grass}). 
For any $\lambda\in X_\ast(T)^+_M$, we have the locally spatial diamond $\Gr_{M,\lambda}\subset \Gr_M$.
Consider the locally closed sub ind-diamond
\[S_\lambda:=i\big(\pr^{-1}(\Gr_{M,\lambda})\big)\subset \Gr_G.\] 
This is identified with the orbit $LU_P\Gr_{M,\lambda}$ for the natural action $LU_P$ on $\Gr_G$ induced by $LU_P\subset LG$. The natural product defines a map $LU_P\times LM\ra LG$ which induces a map $\Gr_{U_P}\times \Gr_M\ra \Gr_G$. Then
we have \[S_\lambda=\Gr_{U_P}\Gr_{M,\lambda}\subset \Gr_G,\] 
where $\Gr_{U_P}\Gr_{M,\lambda}$ denotes the image of $\Gr_{U_P}\times\Gr_{M,\lambda}$ under $\Gr_{U_P}\times \Gr_M\ra \Gr_G$. 
The Iwasawa decomposition above implies that
\[\Gr_G=\coprod_{\lambda\in X_\ast(T)^+_M}S_\lambda.\]
In the following we consider the partial order $\leq_P$\footnote{Note that this is different from the partial order $\leq_M$ used in some literatures, e.g. \cite{GHKR} 5.1, where one uses simple coroots of  $M$.} on $X_\ast(T)$ (and the restriction to $X_\ast(T)^+_M$) with respective to the coroots appearing in $\Lie\, U_P$. When the setting is clear, we simply write $\lambda_1\leq \lambda_2$ for $\lambda_1,\lambda_2\in X_\ast(T)^+_M$ and $\lambda_1\leq_P\lambda_2$.
For any $\lambda\in X_\ast(T)_M^+$,
like in the classical setting, $S_\lambda$ is of infinite dimensional. Nevertheless, we have
\begin{proposition}\label{P:closure semi-infinite}
	The closure $\ov{S_\lambda}$ of $S_\lambda$ is given by \[S_{\leq\lambda}=\coprod_{\lambda'\leq \lambda}S_{\lambda'}.\] 
\end{proposition}
\begin{proof}
	We adapt the argument of \cite{Zhu} Proposition 5.3.6. We show firstly that $S_{\leq \lambda}$ is closed. 
	First, assume that $G_{der}$ is simply connected. For any highest weight representation $V_\chi$ of $G$, let $\ell_\chi$ be the corresponding highest weight line. Then we have the following description
	\[S_{\leq \lambda}=\bigcap_{V_\chi}\{(\E,\beta)\in \Gr_G |\,\beta^{-1}(\ell_\chi)\subset t^{-\lan \chi,\lambda\ran}(\E_{V_{\chi}})\},\]
	where the intersection runs through all highest weight representations $V_\chi$ of $G$, and $\E_{V_{\chi}}=\E\times^GV_\chi$ is the induced vector bundle.
	It suffices to prove the locus \[\{(\E,\beta)\in \Gr_G |\,\beta^{-1}(\ell_\chi)\subset t^{-\lan \chi,\lambda\ran}(\E_{V_{\chi}})\}\subset \Gr_G\] is closed. This follows from the proof of \cite{SW} Lemma 19.1.4. For general $G$, one can pass to a $z$-extension to reduce to the case when $G_{der}$ is simply connected.
	
	Now we show $\ov{S_{\lambda}}=S_{\leq\lambda}$. For $\lambda'\leq \lambda$, there exists a positive coroot $\alpha$ appearing in $\Lie\, U_P$ such that $\lambda-\alpha$ is $M$-dominant and $\lambda'\leq \lambda-\alpha\leq \lambda$. Then the arguments in the proof of Proposition \ref{P:affine Schubert} (1) apply.
\end{proof}

Let $\mu\in X_\ast(T)^+$ be fixed and consider $\Gr_{G,\mu}$. For any $\lambda\in X_\ast(T)_M^+$, note that \[S_\lambda\cap \Gr_{G,\mu}\neq \emptyset\quad \Longleftrightarrow\quad LU_Pt^\lambda\cap \Gr_{G,\mu}\neq \emptyset.\]Indeed, to prove the direction $``\Rightarrow''$, it suffices to work with an algebraically closed field $C$ and then use the normality of $U_P$.
Set
\[S_M(\mu):=\{\lambda\in X_\ast(T)^+_M|\,S_\lambda\cap \Gr_{G,\mu}\neq \emptyset\}.\]
The stratification $\Gr_G=\coprod_{\lambda\in X_\ast(T)^+_M}S_\lambda$ induces a stratification of locally spatial diamonds
\[ \Gr_{G,\mu}=\coprod_{\lambda\in S_M(\mu)}S_\lambda\cap \Gr_{G,\mu}.\]
For each $\lambda\in S_M(\mu)$, for simplicity we denote $\Gr_{G,\mu,\lambda}=S_{\lambda}\cap \Gr_{G,\mu}$, so that \[\Gr_{G,\mu}=\coprod_{\lambda\in S_M(\mu)}\Gr_{G,\mu,\lambda}.\]

To describe the index set $S_M(\mu)$, first note by \cite{GHKR} Lemma 5.4.1
\[S_M(\mu)\subset \Sigma(\mu)_{M-dom},\]
where $\Sigma(\mu)_{M-dom}\subset X_\ast(T)_M^+$ is the set of $M$-dominant elements in 
$\{\mu'\in X_\ast(T)|\,\mu'_{dom}\leq \mu\}$. 
Indeed, to describe $S_M(\mu)$ we may choose any algebraically closed perfectoid field $C|\Q_p$ and consider the $C$-points of $\Gr_{G,\mu}(C,\Ol_C)$. Then $\lambda\in S_M(\mu)$ if and only if $\lambda\in X_\ast(T)^+_M$, and\[ U_P(B_{dR}(C))t^\lambda\bigcap G(B_{dR}^+(C))t^\mu G(B_{dR}^+(C)) \neq \emptyset\quad \Big(\tr{both as subsets of}\, G(B_{dR}(C))\Big).\]
 Fixing an isomorphism $B_{dR}(C)\simeq C((t))$, we translate these to subsets of $G\Big(C((t))\Big)$. As in the proof of \cite{GHKR} Lemma 5.4.1 (which is purely group theoretical and applies to general base fields), $\lambda\in \Sigma(\mu)_{M-dom}$ if and only if $\lambda\in X_\ast(T)^+_M$ and $U_B\Big(C((t))\Big)t^\lambda \bigcap G(C[[t]])t^\mu G(C[[t]])\neq \emptyset$, where $U_B$ is the unipotent radical of $B$. 

Recall that attached to $\mu$ we have the parabolic subgroup $P_\mu\subset G_{\ov{\Q}_p}$. Let $W$ (resp. $W_P, W_{P_\mu}$) be the absolute Weyl group of $G$ (resp. $P, P_\mu$). 
We have the following inclusion:
\[W\mu\cap X_\ast(T)^+_M\subset S_M(\mu).\]
The set $W\mu\cap X_\ast(T)^+_M$ can be described as \[W\mu\cap X_\ast(T)^+_M={}^PW^{P_\mu}\mu,\] where ${}^PW^{P_\mu}\subset W$ is the set of minimal length representatives in the corresponding coset in $W_P\setminus W/W_{P_\mu}$. 
Then the element \[\lambda_0=\mu\in S_M(\mu)\] is the unique maximal element with respective to the partial order $\leq_P$.
When $\mu$ is minuscule, we have \[W\mu\cap X_\ast(T)^+_M={}^PW^{P_\mu}\mu= S_M(\mu).\] In this case, under the isomorphism $\Gr_{G,\mu}\st{\sim}{\ra}\Fl(G,\mu)^\Diamond$, for $\lambda=w\mu$ with $w\in {}^PW^{P_\mu}$,  we have
\[S_\lambda\cap \Gr_{G,\mu}\simeq (U_Pww_0P_\mu/P_\mu)^\Diamond,\]
where $w_0\in W$ is the element of maximal length.
\begin{remark}
	Using the geometric Satake equivalence for $B_{dR}^+$-affine Grassmannians (cf. \cite{FS}) we have the following representation theoretic description of $S_M(\mu)$:
	
	Let $\wh{G}$ be the dual reductive group of $G$ (over some characteristic zero algebraically closed field) and $\wh{M}\subset \wh{G}$ be Levi subgroup defined by the dual root datum of $M$. Similarly let $\wh{T}\subset \wh{B}\subset \wh{G}$ be the maximal torus dual to $T$ inside the Borel subgroup of $\wh{G}$ dual to $B$. Then we may view $\mu\in X^\ast(\wh{T})^+=X_\ast(T)^+$. Consider the irreducible representation $V_\mu$ of highest weight $\mu$ of $\wh{G}$. The geometric Satake equivalence in the current setting implies that $S_M(\mu)$ is the set of $\wh{M}$-dominant weights of $\wh{T}$ such that the associated highest weight  representations of $\wh{M}$ appear in the restricted representation $V_{\mu}|_{\wh{M}}$:
	\[S_M(\mu)=\{\lambda\in X^\ast(T)^+_{\wh{M}}|\, 0\neq V_\lambda\subset V_{\mu}|_{\wh{M}}\},\]
	where for any $\lambda\in X^\ast(T)^+_{\wh{M}}$, $V_\lambda$ is the irreducible representation of $\wh{M}$ of highest weight $\lambda$.
	
	We identify $W=W(\wh{G})$ and $X^\ast(\wh{T})^+_{\wh{M}}=X_\ast(T)^+_M$. 
	The set $W\mu\cap X^\ast(\wh{T})^+_{\wh{M}}={}^PW^{P_\mu}\mu$ appears naturally when considering the decomposition of $V_{\mu}|_{\wh{M}}$ into irreducible representations of $\wh{M}$: we view $\mu\in X^\ast(\wh{T})^+_{\wh{M}}$, then the associated irreducible representation $V^{\wh{M}}_\mu$ of $\wh{M}$ appears in $V_{\mu}|_{\wh{M}}$. Consider the adjoint action of $W$ on $V_\mu=V_{\mu}|_{\wh{M}}$. For any $w\in {}^PW^{P_\mu}$, we have \[wV^{\wh{M}}_\mu=V_{w\mu}\subset V_{\mu}|_{\wh{M}}.\]
	Any $\lambda\in S_M(\mu)$ is of the form
	\[\lambda=\mu-\sum_{\alpha\in \Delta\setminus \Delta_{\wh{M}}}n_\alpha\alpha,\quad n_\alpha\in \mathbb{N},\forall\, \alpha,\]
	where $\Delta=\Delta_{\wh{G}}$ (resp. $\Delta_{\wh{M}}$) is the set of simple roots of $\wh{G}$ (resp. $\wh{M}$). Therefore, \[W\mu\cap X^\ast(\wh{T})^+_{\wh{M}}=W\mu\cap X_\ast(T)^+_M\subset S_M(\mu)\]
	and $\mu\in S_M(\mu)$ is the unique maximal element.
\end{remark}

Recall that the locally spatial diamond $\Gr_\mu=\Gr_{G,\mu}$ is defined over $\Spd\,E$ with $E=E(G,\{\mu\})$. As usual, let $\breve{E}=\wh{E^{ur}}$ be the completion of the maximal unramified extension of $E$. We will study $\Gr_{G,\mu}$ over $\Spd\,\breve{E}$.
First of all, we explain that
the set $S_M(\mu)$ and the above diagram of $B_{dR}^+$-affine Grassmannians naturally arise when considering reductions of modifications of $G$-bundles to $P$-bundles (resp. $M$-bundles), cf. Lemma \ref{L:modif type}.

For $C|\breve{E}$ any algebraically closed perfectoid field, let $X=X_{C^\flat}$ be the Fargues-Fontaine curve over $\Q_p$ attached to $C^\flat$.
Let $b\in G(\breve{\Q}_p)$ be an element with associated class $[b]\in B(G)$ and the $G$-bundle $\E_b$ on $X$ (cf. \cite{F3}).
For a Levi subgroup $M$ of $G$, recall that (cf. \cite{CFS} Definition 2.5) we have the notion of reductions of $b$ to $M$. Such a reduction is given by an element $b_M\in M(\breve{\Q}_p)$ together with an element $g\in G(\breve{\Q}_p)$ such that $b=gb_M\sigma(g)^{-1}$. Then the $M$-bundle $\E_{b_M}$ is a reduction of $\E_b$.  If $M\subset P$ for some parabolic subgroup $P$ of $G$,
let $b_P\in P(\breve{\Q}_p)$ be the image of $b_M$. This defines a reduction of $b$ to $P$, and thus a reduction of the $G$-bundle $\E_b$ to a $P$-bundle $\E_{b_P}$. By construction, $\E_{b_P}=\E_{b_M}\times^MP$.

For any $x\in \Gr_G(C,\Ol_C)$, we can define a modification $\E_{b,x}$ of $\E$, thus a map \[\Gr_G(C,\Ol_C)\ra H^1_{\textrm{\'et}}(X, G).\] It is functorial in the following sense: we have similar maps \[\Gr_P(C,\Ol_C)\ra H^1_{\textrm{\'et}}(X,P),\quad y\mapsto \E_{b_P,y},\]
\[\Gr_M(C,\Ol_C)\ra H^1_{\textrm{\'et}}(X,M),\quad z\mapsto \E_{b_M,z},\]
by considering modifications of the $P$-bundle $\E_{b_P}$ and the $M$-bundle $\E_{b_M}$ respectively. Then the following diagram commutes:
\[\xymatrix{
	\Gr_G(C,\Ol_C)\ar[r]& H^1_{\textrm{\'et}}(X, G)\\
	\Gr_P(C,\Ol_C)\ar[u]\ar[d]\ar[r]&H^1_{\textrm{\'et}}(X, P)\ar[u]\ar[d]\\
	\Gr_M(C,\Ol_C)\ar[r]& H^1_{\textrm{\'et}}(X,M),
}\]
where the arrows on the right hand side are $\E\mapsto \E\times^PG, \quad \E\mapsto\E\times^PM$, the push forwards of $P$-bundles along $P\subset G$ and $P\ra M$ respectively. Recall that by Iwasawa decomposition, the map $\Gr_P(C,\Ol_C)=P(B_{dR})/ P(B_{dR}^+)\st{\sim}{\lra} \Gr_G(C,\Ol_C)=G(B_{dR})/G(B_{dR}^+)$ is a bijection. For $x\in \Gr_G(C,\Ol_C)$, let $y\in \Gr_P(C,\Ol_C)$ be its inverse image under this bijection. Then \[\E_{b_P,y}\times^PG=\E_{b,x},\] i.e. $\E_{b_P,y}$ is a reduction to $P$ of $\E_{b,x}$. By \cite{CFS} Lemma 2.5, $\E_{b_P,y}$ is the reduction to $P$ of $\E_{b,x}$ induced by the reduction $\E_{b_P}$ of $\E_b$. We will also write \[\E_{b_P,y}=(\E_{b,x})_P\] for this reduction. Recall that we have the decomposition \[\Gr_{G,\mu}(C,\Ol_C)=\coprod_{\lambda\in S_M(\mu)}\Gr_{G,\mu,\lambda}(C,\Ol_C).\] For $\lambda\in S_M(\mu)$, let $\pr_\lambda: \Gr_{G,\mu,\lambda}(C,\Ol_C)\ra \Gr_{M,\lambda}(C,\Ol_C)$ be the projection.
The following generalization of \cite{CFS} Lemma 2.6 is clear now.
\begin{lemma}\label{L:modif type}
	For any $x\in \Gr_{G,\mu}(C,\Ol_C)$, let $\lambda\in S_M(\mu)$ be such that
	$x\in \Gr_{G,\mu,\lambda}(C,\Ol_C)$. Then there is an isomorphism of $M$-bundles
	\[(\E_{b,x})_P\times^PM\simeq \E_{b_M, \pr_{\lambda}(x)},\]
	where $(\E_{b,x})_P$ is the reduction of $\E_{b,x}$ induced by the reduction $\E_{b_P}$ of $\E_b$ as above.
\end{lemma}

\subsection{Newton and Harder-Narasimhan stratifications on $\Gr_{\mu^{-1}}$}\label{subsection Newt HN Hodge-Tate}
We keep the notations as in the last subsection. Consider the affine Schubert cells $\Gr_\mu$ and $\Gr_{\mu^{-1}}$ over $\breve{E}$. 
\\

We first study the geometry of $\Gr_{\mu^{-1}}$ using modifications of the trivial $G$-bundle $\E_1$. Consider the morphism $\Gr_{\mu^{-1}}\ra \Bun_{G,\ov{\F}_p}$ constructed in \ref{subsection Hecke stack}. The induced map on the sets of $C$-valued points can be described in more concrete terms.
Let $C|\breve{E}$ be an algebraically closed perfectoid field. For any $x\in \Gr_{\mu^{-1}}(C,\Ol_C)$, we have the modification \[\E_{1,x}\] of $\E_1$. The isomorphism class of $\E_{1,x}$ defines a point $b(\E_{1,x})\in B(G)$. 
We write $Newt: \Gr_{\mu^{-1}}(C,\Ol_C)\ra B(G)$ for the map.
\begin{proposition}\label{P: image Tate Newton}
	The image of the induced map
	$Newt: \Gr_{\mu^{-1}}(C,\Ol_C)\ra B(G)$ is $B(G,\mu)$.
\end{proposition}
\begin{proof}
	The fact that the image of the above map is included in $B(G,\mu)$ follows from \cite{CS} Proposition 3.5.3. 
	
	To show the surjectivity, if $\mu$ is minuscule, then it follows from \cite{R} Proposition A.9. For the general case,
	consider the affine Schubert cell $\Gr_{\mu}$. Let $[b]\in B(G,\mu)$ be any element with a representative $b\in G(\breve{\Q}_p)$. 
	Let $\Gr_{\mu}^a\subset \Gr_{\mu}$ be the associated admissible locus (here $\Gr_{\mu}^a=\Gr_\mu\cap \Gr_{\leq\mu}^a$ and $\Gr_{\leq\mu}^a$ is the admissible locus introduced in the proof of \cite{SW} Proposition 23.3.3). On the other hand, let $\Fl(G,\mu,b)^{wa}\subset\Fl(G,\mu)$ be the associated weakly admissible locus (cf. \cite{RZ, DOR}).
	Then  the Bialynicki-Birula map induces a morphism of diamonds
	\[\pi_{\mu}:  \Gr_{\mu}^a\ra \Fl(G,\mu,b)^{wa,\Diamond}.\]
	By the theorem of Colmez-Fontaine (cf. \cite{FF} chapter 10), we have \[\Gr_{\mu}^a(K,\Ol_K)=\Fl(G,\mu,b)^{wa}(K,\Ol_K)\] for any finite extension $K|\breve{E}$. Thus $\Fl(G,\mu,b)^{wa}\neq \emptyset$ if and only if $\Gr_{\mu}^a\neq\emptyset$. Since $[b]\in B(G,\mu)$, by \cite{RV} Proposition 3.1, $\Fl(G,\mu,b)^{wa}\neq \emptyset$ and thus $\Gr_{\mu}^a\neq\emptyset$.
	Take a point $x\in \Gr_{\mu^{-1}}(C,\Ol_C)$. By definition, \[x\in \Gr_{\mu^{-1}}^{Newt=[b]}(C,\Ol_C)\quad \Leftrightarrow\quad \E_{1,x}\simeq \E_b\quad \Leftrightarrow\quad \E_{1}=\E_{b,x^\ast}\] for some $x^\ast\in \Gr_{\mu}(C,\Ol_C)$. This is equivalent to $x^\ast\in \Gr_{\mu}^a(C,\Ol_C)$. Thus we get for any $[b]\in B(G,\mu)$,  $\Gr_{\mu^{-1}}^{Newt=[b]}(C,\Ol_C) \neq \emptyset$.
\end{proof}
Letting $C$ vary,
we thus get a map $Newt: |\Gr_{\mu^{-1}}|\lra B(G,\mu)$. By \cite{KL} (in the case $G=\GL_n$) and \cite{SW} Corollary 22.5.1, this map is upper semi-continuous.
The Newton stratification of $\Gr_{\mu^{-1}}$ is the following stratification in diamonds over $\breve{E}$ (which is in fact defined over $E$):
\[\Gr_{\mu^{-1}}=\coprod_{[b']\in B(G,\mu)}\Gr_{\mu^{-1}}^{Newt=[b']}. \]
The \emph{open} Newton stratum \[\Gr_{\mu^{-1}}^{Newt=[b]}\] is associated to the \emph{basic} element $[b]\in B(G,\mu)$. Recall that
We have the natural action of $G(\Q_p)$ on $\Gr_{\mu^{-1}}$. Since $\mathrm{Aut}(\E_1)=G(\Q_p)$, for any $[b']\in B(G,\mu)$ the stratum $\Gr_{\mu^{-1}}^{Newt=[b']}$ is stable under the $G(\Q_p)$-action.
\begin{proposition}
	We have the following dimension formula: for $[b']\in B(G,\mu)$,
	\[ \dim \,\Gr_{\mu^{-1}}^{Newt=[b']}=\langle \mu-\nu([b']), 2\rho\rangle.\]
\end{proposition}
\begin{proof}
This is essentially the same as the proof of Proposition \ref{T: Newton strata} (2), using the diagram in subsection \ref{subsection local shtukas} and the dimension formula $\dim\, \Gr_\mu=\lan \mu, 2\rho\ran$ of Proposition \ref{P:affine Schubert} (2).
\end{proof}

For any point $x\in \Gr_{\mu^{-1}}(C,\Ol_C)$,
consider the admissible modification $(\E_1,\E_{1,x},f)$ and the associated HN vector $\nu(\E_1,\E_{1,x},f)\in \Nc(G)$. Then Proposition \ref{P: HN and Newt} still holds in this setting by easily modifying the proof therein (using the semi-infinite orbit decomposition of $\Gr_{\mu^{-1}}$ instead of the Bruhat decomposition). In other words, we have
\[\nu(\E_1,\E_{1,x},f)\leq \nu(\E_{1,x})\] and
 \[\nu(\E_1,\E_{1,x},f)^\ast=w_0(-\nu(\E_1,\E_{1,x},f))\in  \Nc(G,\mu).\] Letting $C$ vary,
 we thus get a map $HN: |\Gr_{\mu^{-1}}|\lra \Nc(G,\mu)$.
 
\begin{theorem}\label{T: HN semi cont}
	For any $v\in \mathcal{N}(G,\mu)$, the subset  \[\Gr_{\mu^{-1}}^{HN= v}:=\{x\in |\Gr_{\mu^{-1}}|\,| HN(x)= v\}\] is locally closed, stable under the $G(\Q_p)$-action, and it defines a sub diamond of $\Gr_{\mu^{-1}}$. 
	Moreover, for the basic element $[b]\in B(G,\mu)$ with $v_0=\nu([b])$, the subset $\Gr_{\mu^{-1}}^{HN= v_0}$ is open, which is the semi-stable locus, and we have an inclusion \[\Gr_{\mu^{-1}}^{Newt=v_0}\subset \Gr_{\mu^{-1}}^{HN= v_0}.\]
\end{theorem}
\begin{proof}
	We generalize the arguments in the proof of Theorem \ref{T: property HN strata}.
Let $\Theta(G,\mu)$ be the set of pairs $(P, \lambda)$, where $P$ is a standard parabolic (here including $G$) of $G$ with associated standard Levi $M$, $\lambda\in X_\ast(T)_M^+$, such that the following conditions hold:
\begin{enumerate}
	\item $\lambda\in S_M(\mu^{-1})$, where $S_M(\mu^{-1})$ is the subset of $X_\ast(T)_M^+$ introduced in the last subsection;
	\item Let $v(\lambda)\in X_\ast(A_P)_\Q$ be the image of $\lambda$ under the natural projections $X_\ast(T)_\Q\ra X_\ast(A)_\Q\ra  X_\ast(A_P)_\Q$. Then
	\[\lan v(\lambda),\alpha\ran >0,\quad \forall\,\alpha\in \Delta_{0,P}.\]
\end{enumerate}
In particular, $\Theta(G,\mu)$  is a finite set. The inclusion $X_\ast(A_P)_\Q\subset X_\ast(A)_\Q$ induces a natural map
\[\Theta(G,\mu)\ra X_\ast(A)_{\Q}^+,\quad (P,\lambda)\mapsto v(\lambda).\]This map factors through $\Nc(G,\mu)$. The element $\theta_0=(G,\mu)$ maps to the basic element $v_0$ of $\Nc(G,\mu)$. For any $x\in \Gr_{\mu^{-1}}(C,\Ol_C)$, let $P$ be the standard parabolic determined by the HN vector $HN(x)^\ast=\nu(\E_1,\E_{1,x},f)$. Consider the Iwasawa decomposition relative to $P$ and $M$, and let $\lambda\in S_M(\mu^{-1})$ be the cocharacter such that $x\in \Gr_{\mu^{-1}}(C,\Ol_C)\cap S_\lambda(C,\Ol_C)$. By \cite{CS} Lemma 3.5.5, \[c_1^M(\E_{1_M,\pr_{\lambda}(x)})=\lambda^\sharp\in \pi_1(M)_\Gamma.\] Therefore \[\mu(\E_{1_M},\E_{1_M, \pr_{\lambda}(x)}, f_M)=\lambda^\sharp\otimes 1=v(\lambda)\in \pi_1(M)_{\Gamma,\Q}.\]
By Proposition \ref{P: semi-stable adm modif}, we have $v(\lambda)=\nu(\E_1,\E_{1,x},f)$ and $(P,\lambda)\in \Theta(G,\mu)$.  Thus we get
 a well defined map \[\Theta: |\Gr_{\mu^{-1}}|\ra \Theta(G,\mu). \]
It suffices to show that for any $\theta=(P,\lambda)\in \Theta(G,\mu)$, 
the subset  \[\Gr_{\mu^{-1}}^{\theta}:=\{x\in |\Gr_{\mu^{-1}}|\,| \Theta(x)=\theta\}\]  defines a sub diamond of $\Gr_{\mu^{-1}}$. In fact, we have \[\Gr_{\mu^{-1}}^{HN= v}=\coprod_{\theta=(P,\lambda),v(\lambda)^\ast=v}\Gr_{\mu^{-1}}^{\theta}.\]

Note that we have the equality $\Gr_{\mu^{-1}}^{HN= v_0}=\Gr_{\mu^{-1}}^{\theta_0}$ which is the semi-stable locus and we denote by $\Gr_{G,\mu^{-1}}^{ss}$. We first show that this is an open sub diamond of $\Gr_{\mu^{-1}}$, i.e. $|\Gr_{G,\mu^{-1}}|\setminus |\Gr_{G,\mu^{-1}}^{ss}|$ is closed. This follows from Corollary \ref{C: semi-stable adm modif} and Proposition \ref{P:closure semi-infinite}. Indeed, for any $x\in |\Gr_{G,\mu^{-1}}|\setminus |\Gr_{G,\mu^{-1}}^{ss}|$, by Corollary \ref{C: semi-stable adm modif}, there exits a maximal standard parabolic $P$ with standard Levi $M$, such that $\lan\mu(\E_{1,x,M}), \alpha\ran >0$, where $\Delta_{0,P}=\{\alpha\}$.  Then an easy argument shows that $P\supset P_{HN(x)}$.
Consider the construction of subsection \ref{subsection semi-infinite} with respect to $P$ and $M$. Let $\lambda\in S_{M}(\mu^{-1})$ be such that $x\in \Gr_{G,\mu^{-1}}\cap S_\lambda$, where $S_\lambda$ is the generalized semi-infinite orbit attached to $\lambda$. By Corollary \ref{C: semi-stable adm modif}, $\Gr_{G,\mu^{-1}}\cap S_\lambda \subset |\Gr_{G,\mu^{-1}}|\setminus |\Gr_{G,\mu^{-1}}^{ss}|$. By Proposition 
\ref{P:closure semi-infinite}, the closure of $\Gr_{G,\mu^{-1}}\cap S_\lambda$  inside $\Gr_{G,\mu^{-1}}$
 \[\ov{\Gr_{G,\mu^{-1}}\cap S_\lambda}\subset\coprod_{\lambda'\leq \lambda}\Gr_{G,\mu^{-1}}\cap S_{\lambda'}.\] Then by the last paragraph in the proof of Proposition \ref{P:closure semi-infinite}, for any $\lambda'\leq \lambda$, we have $\lan v(\lambda'), \alpha\ran >0$ for the parabolic $P$ with $\Delta_{0, P}=\{\alpha\}$, as the inequality holds for $\lambda$. By  Corollary \ref{C: semi-stable adm modif}, this implies \[\ov{\Gr_{G,\mu^{-1}}\cap S_\lambda}\subset |\Gr_{G,\mu^{-1}}|\setminus |\Gr_{G,\mu^{-1}}^{ss}|.\] Thus $|\Gr_{G,\mu^{-1}}|\setminus |\Gr_{G,\mu^{-1}}^{ss}|$ is closed.

For a general $\theta=(P,\lambda)\in \Theta(G,\mu)$, consider the stratification $\Gr_{\mu^{-1}}=\coprod_{\lambda'\in S_M(\mu^{-1})}\Gr_{\mu^{-1}}\cap S_{\lambda'}$ with respect to $P$ and the associated Levi $M$. Set \[\Gr_{P,\lambda}^{\theta}:=\Gr_{\mu^{-1}}^{\theta}\cap  S_\lambda.\] Then we have two inclusions \[ \Gr_{P,\lambda}^{\theta}\subset \Gr_{\mu^{-1}}\cap S_\lambda, \quad \Gr_{P,\lambda}^{\theta}\subset \Gr_{\mu^{-1}}^{\theta}.\] 
Consider the first inclusion $\Gr_{P,\lambda}^{\theta}\subset \Gr_{\mu^{-1}}\cap S_\lambda$.
By Proposition \ref{P: HN reduction}, $\Gr_{P, \lambda}^{\theta}$  is 
the preimage (fiber product) of the semi-stable locus $\Gr_{M,\lambda}^{ss}\subset \Gr_{M,\lambda}$ under the projection \[\pr_\lambda: \Gr_{\mu^{-1}}\cap  S_\lambda\ra \Gr_{M,\lambda}.\]
As we just proved that $\Gr_{M,\lambda}^{ss}$ is open in $\Gr_{M,\lambda}$, $\Gr_{P,\lambda}^{\theta}$ is a locally closed and locally spatial sub diamond of $\Gr_{\mu^{-1}}$.
Now we look at $\Gr_{P,\lambda}^{\theta}\subset \Gr_{\mu^{-1}}^{\theta}$.
Note that the natural action of $G(\Q_p)$ on $\Gr_{\mu^{-1}}$ restricts to an action of $G(\Q_p)$ (resp. $P(\Q_p)$) on $\Gr_{\mu^{-1}}^{\theta}$ (resp. $\Gr_{P,\lambda}^{\theta}$). Then
by construction, as subsets of $|\Gr_{\mu^{-1}}|$ we have \[\Gr_{\mu^{-1}}^{\theta}=\bigcup_{g\in G(\Q_p)/P(\Q_p)}g\Gr_{P,\lambda}^{\theta}.\] We claim that $\Gr_{\mu^{-1}}^{\theta}$ is  locally closed in $\Gr_{\mu^{-1}}$. Indeed, as \[\bigcup_{g\in G(\Q_p)/P(\Q_p)}g\Gr_{P,\lambda}^{\theta}\subset \bigcup_{g\in G(\Q_p)/P(\Q_p)}g(\Gr_{\mu^{-1}}\cap  S_\lambda)\] is open, it suffices to show the later is locally closed in $\Gr_{\mu^{-1}}$. Then we are further reduced to show that $\bigcup_{g\in G(\Q_p)/P(\Q_p)}(\Gr_{\mu^{-1}}\cap  S_{\leq g\lambda g^{-1}})$ is closed in $\Gr_{\mu^{-1}}$, where $S_{\leq g\lambda g^{-1}}$ is the closed semi-infinite orbit associated to the parabolic $gPg^{-1}$ and $g\lambda g^{-1}$. Now, the closedness of the last union follows from the fact that $G(\Q_p)/P(\Q_p)$ is compact and $S_{\leq g\lambda g^{-1}}$ is closed by Proposition \ref{P:closure semi-infinite}.  Moreover, as all subsets are generalizing, $\Gr_{\mu^{-1}}^{\theta}$ defines a locally spatial sub $v$-sheaf of $\Gr_{\mu^{-1}}$.
Then, we get an isomorphism of locally spatial  $v$-sheaves
 \[\Gr_{P,\lambda}^{\theta}\times^{\ul{P(\Q_p)}}\ul{G(\Q_p)}\st{\sim}{\lra} \Gr_{\mu^{-1}}^{\theta}.\] 
 Indeed, the left hand defines a priori a locally spatial sub $v$-sheaf $\Gr_{P,\lambda}^{\theta}\times^{\ul{P(\Q_p)}}\ul{G(\Q_p)}\subset \Gr_{\mu^{-1}}^{\theta}$. The inclusion map is quasicompact, since the locally closed subspace of $|\Gr_{\mu^{-1}}^{\theta}|$ is the whole $|\Gr_{P,\lambda}^{\theta}\times^{\ul{P(\Q_p)}}\ul{G(\Q_p)}|=\bigcup_{g\in G(\Q_p)/P(\Q_p)}g|\Gr_{P,\lambda}^{\theta}|=|\Gr_{\mu^{-1}}^{\theta}|$. By \cite{S} Lemma 12.11 and Proposition 12.15, this inclusion is in fact an isomorphism of $v$-sheaves.

Finally, the inclusion $\Gr_{\mu^{-1}}^{Newt=v_0}\subset \Gr_{\mu^{-1}}^{HN= v_0}$ comes from the inequality
$\nu(\E_1, \E_{1,x}, f)\leq \nu(\E_{1,x})$ as in Proposition \ref{P: HN and Newt} (1) (which holds for general $\mu$).

\end{proof}
\begin{remark}
	\begin{enumerate}
		\item If $\mu$ is minuscule, by the proofs of Theorem \ref{T: property HN strata} and Theorem \ref{T: HN semi cont}, the HN type stratifications of $\Gr_{\mu^{-1}}$ and $\Fl(G,\mu^{-1})$ coincide via the Bialynicki-Birula isomorphism $\Gr_{\mu^{-1}}\st{\sim}{\ra} \Fl(G,\mu^{-1})^\Diamond$.
		\item One can actually show that the complement of the semi-stable locus $\Gr_{\mu^{-1}}\setminus \Gr_{\mu^{-1}}^{ss}$ is a profinite union of closed subspaces of the form $\ov{\Gr_{\mu^{-1}}\cap S_\lambda}$, where $S_\lambda$ is a generalized semi-infinite orbit with respect to some proper parabolic $P$ of $G$. We leave this to the interested reader.
	\end{enumerate}
\end{remark}
In the following, we will identify $\mathcal{N}(G,\mu)$ with $B(G,\mu)$ by the Newton map.
We have the following stratification of diamonds over $\breve{E}$ (which is in fact defined over $E$):
\[\Gr_{\mu^{-1}}=\coprod_{[b']\in B(G,\mu)}\Gr_{\mu^{-1}}^{HN=[b']}.\]
For any $[b']\in B(G,\mu)$, the stratum $\Gr_{\mu^{-1}}^{HN=[b']}$ is stable under the action of $G(\Q_p)$ on $\Gr_{\mu^{-1}}$.
The open Harder-Narasimhan stratum $\Gr_{\mu^{-1}}^{HN=[b]}$ is associated to the basic element $[b]\in B(G,\mu)$.
By the proof of Theorem \ref{T: HN semi cont},
we get
\begin{corollary}\label{C: property HN strata general}
For any non basic $[b']\in B(G,\mu)$, the stratum $\Gr_{\mu^{-1}}^{HN=[b']}$ is a parabolic induction.
\end{corollary}
We have also the following generalization of Proposition \ref{P: Newt HN ordinary}, as the argument in the proof there applies:
\begin{corollary}
Let $[b_1]\in B(G,\mu)$ be the unique maximal element. Assume that $G$ is quasi-split, then \[\Gr_{\mu^{-1}}^{HN=[b_1]}=\Gr_{\mu^{-1}}^{Newt=[b_1]}.\]
\end{corollary}

\subsection{Newton and Harder-Narasimhan stratifications on $\Gr_{\mu}$}\label{subsection Newt HN de Rham}
Let $[b]\in B(G,\mu)$ be the \emph{basic} element.
Now we study the geometry of $\Gr_{\mu}$ using modifications of the $G$-bundle $\E_b$.  Consider the map $\Gr_\mu\ra \Bun_{G,\ov{\F}_p}$ constructed in subsection \ref{subsection Hecke stack}. Let $C|\breve{E}$ be an algebraically closed perfectoid field.
The induced map on the sets of $C$-valued points can be described  in more concrete terms.
For any $x\in \Gr_{\mu}(C,\Ol_C)$, we have the modification \[\E_{b,x}\] of $\E_b$. The isomorphism class of $\E_{b,x}$ defines a point $b(\E_{b,x})\in B(G)$. 
We write $Newt: \Gr_{\mu}(C,\Ol_C)\ra B(G)$ for the map. Recall the subset $B(G,0,\nu_{b}\mu^{-1})\subset B(G)$ introduced in subsection \ref{subsection twin towers}.
\begin{proposition}
	The image of the induced map
	$Newt: \Gr_{\mu}(C,\Ol_C)\ra B(G)$ is $B(G,0,\nu_{b}\mu^{-1})$.
\end{proposition}
\begin{proof}
	For $\mu$ minuscule, this has been studied in \cite{CFS} section 5 (see also \cite{R} A.10).
	The arguments in \cite{CFS} section 5 work in the general case. See also the proof of Proposition \ref{P: image Tate Newton}.
\end{proof}
Letting $C$ vary, we thus get a map 
$Newt: |\Gr_{\mu}|\lra B(G,0,\nu_{b}\mu^{-1})$,
which is upper semi-continuous by \cite{KL, SW}.
Thus we have the Newton stratification of diamonds over $\breve{E}$:
\[\Gr_{\mu}=\coprod_{[b']\in B(G,0,\nu_{b}\mu^{-1})}\Gr_{\mu}^{Newt=[b']}. \]
Note that $J_b(\Q_p)$ acts on $\Gr_{\mu}$ via the inclusions $J_b(\Q_p)\subset G(\breve{\Q}_p)\subset G(B_{dR}^+(\C_p))$.
For any $[b']\in B(G,0,\nu_{b}\mu^{-1})$, the stratum $\Gr_{\mu}^{Newt=[b']}$ is stable under the action of $J_{b}(\Q_p)=\mathrm{Aut}(\E_b)$ on $\Gr_{\mu}$. The open Newton stratum $\Gr_{\mu}^{Newt=[1]}$ corresponds to the trivial element $[1]\in B(G,0,\nu_{b}\mu^{-1})$, which we will also denote by $\Gr_\mu^a$ (the admissible locus inside $\Gr_\mu$ with respective to $(G,\{\mu\},[b])$, which we already introduced in the proof of Proposition \ref{P: image Tate Newton}).
\begin{remark}\label{R:general adm and wa}
	In this subsection, to define the Newton stratification, in fact we don't need the assumption that $[b]$ is basic. However, we don't know a description of the index set for a non basic $[b]$. 
\end{remark}

Now we want to define a Harder-Narasimhan stratification on $\Gr_{\mu}$. Consider the triple $(J_b, \{\mu\}, [1])$. Then we can consider the Newton and Harder-Narasimhan stratifications on $\Gr_{J_b,\mu}$ as in the last subsection. Since $[b]$ is basic, the isomorphism $J_{b,\breve{\Q}_p}\st{\sim}{\ra}G_{\breve{\Q}_p}$ induces identifications \[\Gr_{G,\mu}\cong\Gr_{J_b,\mu}  \] as diamonds over $\Spd\,\breve{E}$. The  Harder-Narasimhan stratification on $\Gr_{J_b,\mu}$  induces a $J_b(\Q_p)$-equivariant stratification on $\Gr_{\mu}$ over $\breve{E}$:
\[\Gr_{\mu}=\coprod_{[b']\in B(G,0,\nu_b\mu^{-1})}\Gr_{\mu}^{HN=[b']},\]
which we call the Harder-Narasimhan stratification. The open Harder-Narasimhan stratum $\Gr_{\mu}^{HN=[1]}$ corresponds to the trivial element $[1]\in B(G,0,\nu_{b}\mu^{-1})$, which we will also denote by $\Gr_\mu^{wa}$ (the weakly admissible locus inside $\Gr_\mu$ with respective to $(G,\{\mu\},[b])$). 

Consider also the dual local Shtuka datum $(J_b,\{\mu^{-1}\},[b^{-1}])$.
The results of subsection \ref{subsection twin loc shimura} still hold (cf. \cite{SW} subsection 23.3). 
Now the following generalization\footnote{The part on Harder-Narasimhan stratifications is by our construction.} of
 Theorem \ref{T: duality HN Newt} and Corollary \ref{C: duality HN Newt} is clear:
 \begin{theorem}\label{T: duality general}
 \begin{enumerate}
 	\item Under the identification $\Gr_{G,\mu^{-1}}=\Gr_{J_b,\mu^{-1}}$, 
 	for any $[b']\in B(G,\mu)$ corresponding to $[b'']\in B(J_b,0,\nu_{b^{-1}}\mu)$ under the bijection (cf. Lemma \ref{L: bijections Kottwitz sets}) \[B(G,\mu)\st{\sim}{\lra} B(J_b,0,\nu_{b^{-1}}\mu),\] we have
 	\begin{enumerate}
 		\item 
 		$ \Gr_{G,\mu^{-1}}^{HN=[b']}=\Gr_{J_b,\mu^{-1}}^{HN=[b'']}.$
 		\item $ \Gr_{G,\mu^{-1}}^{Newt=[b']}=\Gr_{J_b,\mu^{-1}}^{Newt=[b'']}.$
 	\end{enumerate}
 	\item Under the identification $\Gr_{J_b,\mu}=\Gr_{G,\mu}$,
 	for any $[b']\in B(G,\mu)$ corresponding to $[b'']\in B(J_b,0,\nu_{b^{-1}}\mu)$ under the bijection (cf. Lemma \ref{L: bijections Kottwitz sets})  \[B(J_b,\mu^{-1})\st{\sim}{\lra} B(G,0,\nu_{b}\mu^{-1}),\] we have
 	\begin{enumerate}
 		\item
 		$\Gr_{J_b,\mu}^{HN=[b']}=\Gr_{G,\mu}^{HN=[b'']}.$
 		\item $\Gr_{J_b,\mu}^{Newt=[b']}=\Gr_{G,\mu}^{Newt=[b'']}.$
 	\end{enumerate}
 \end{enumerate}
 \end{theorem}
By Theorem \ref{T: HN semi cont} and the construction, we have $\Gr_\mu^a\subset \Gr_\mu^{wa}$. 
\subsection{Extensions to $\Gr_{\leq\mu}$ and $\Gr_{\leq{\mu^{-1} }}$}
We can extend the above constructions to $\Gr_{\leq\mu}$ and $\Gr_{\leq{\mu^{-1} }}$. First, we note the following lemma.
\begin{lemma}
	For $\mu_1,\mu_2\in X_\ast(T)^+$ with $w_0(-\mu_1)\leq w_0(-\mu_2)$, we have a natural injection $B(G,\mu_1)\hookrightarrow B(G,\mu_2)$.
\end{lemma}
\begin{proof}
The assumption $w_0(-\mu_1)\leq w_0(-\mu_2)$ implies that $\mu_1\leq \mu_2$ and thus $\mu_1^\diamond\leq \mu_2^\diamond$. Recall that by \cite{Kot2} 4.13, \[(\kappa,\nu): B(G)\rightarrow \pi_1(G)_\Gamma\times \Nc(G)\] is injective. For $[b]\in B(G,\mu_1)$,  consider the pair $(\mu_2^\sharp, \nu([b])\in \pi_1(G)_\Gamma\times \Nc(G)$. It comes from a unique element $[b']\in B(G)$ under the injection $(\kappa,\nu): B(G)\hookrightarrow \pi_1(G)_\Gamma\times \Nc(G)$, since $\kappa$ is surjective and $\mu_2^\sharp\equiv \nu([b])$ in $\pi_1(G)_{\Gamma,\Q}$. Then since $\nu([b'])=\nu([b])\leq\mu_1^\diamond\leq \mu_2^\diamond$,  by definition $[b']\in B(G,\mu_2)$. In this way we get an injection $B(G,\mu_1)\hookrightarrow B(G,\mu_2)$.
\end{proof}
By the above lemma, we can define Newton and Harder-Narasimhan stratifications on $\Gr_{\leq\mu^{-1}}$ by modifications of the trivial $G$-bundle $\E_1$, with both of the index sets as $B(G,\mu)$. These strata will be the union over all $(\mu')^{-1}\leq \mu^{-1}$ of the corresponding strata (could be empty) inside $\Gr_{(\mu')^{-1}}$. Similar, for $[b]\in B(G,\mu)$ basic, we can define Newton and Harder-Narasimhan stratifications on $\Gr_{\leq\mu}$ by modifications of the $G$-bundle $\E_b$, with both of the index sets as $B(G,0,\nu_b\mu^{-1})$. The strata will be the union over all $\mu'\leq \mu$ of the corresponding strata (which could be empty) inside $\Gr_{\mu'}$. We will use the version of moduli of local $G$-Shtukas $\Sht(G,\leq\mu,b)_\infty$ in this setting. We can consider the dual local Shtuka datum $(J_b,\{\mu^{-1}\},[b^{-1}])$. Then the constructions and results in subsections \ref{subsection Newt HN Hodge-Tate} and \ref{subsection Newt HN de Rham}, in particular Theorem \ref{T: duality general}, can be generalized to the current setting.
We leave the details to the interested reader.

\subsection{Fargues-Rapoport conjecture for general $\mu$}\label{subsection FF general}
In the following we explain how to generalize the arguments in the proof of \cite{CFS} Theorem 6.1 to the non minuscule case.

Let $\mu\in X_\ast(T)^+$ and consider the diamond $\Gr_{G,\mu}$ over $\breve{E}$. Applying Proposition \ref{P: semi-stable adm modif},
we deduce the following generalization of \cite{CFS} Proposition 2.7 in non minuscule case (but for $[b]$ basic). 
\begin{proposition}\label{P:weakly adm general}
	Assume that $G$ is quasi-split and $[b]\in B(G,\mu)$ basic. Then $x\in \Gr_{G,\mu}(C,\Ol_C)$ is weakly admissible if and only if for any standard parabolic $P$ with associated standard Levi $M$, any reduction $b_M$ of $b$ to $M$, and any $\chi\in X^\ast(P/Z_G)^+$, we have
	\[\deg\chi_\ast(\E_{b,x})_P\leq 0,\]
	where $(\E_{b,x})_P$ is the reduction to $P$ of $\E_{b,x}$ induced by the reduction $\E_{b_P}$ of $\E_b$ as above.
\end{proposition}
\begin{proof}
As $[b]$ is basic, the group $J_b$ is an inner form of $G$. Since $G$ is quasi-split,  for any standard parabolic $P$ of $G$ with associated standard Levi $M$ such that $b$ admits a reduction $b_M$ to $M$, there exists a unique parabolic $P'$ of $J_b$ with associated Levi $M'$. Moreover, we have a bijection $X^\ast(P'/Z_{J_b})^+\st{\sim}{\ra}X^\ast(P/Z_G)^+, \chi'\mapsto \chi$.
Under the bijection
$\Gr_{J_b,\mu}(C,\Ol_C)\st{\sim}{\ra} \Gr_{G,\mu}(C,\Ol_C), x'\mapsto x$, we have
$\deg\chi'_\ast(\E_{1,x'})_{P'}=\deg\chi_\ast(\E_{b,x})_P$. Now as $x\in \Gr_{G,\mu}(C,\Ol_C)$ is weakly admissible if and only if $x'\in\Gr_{J_b,\mu}(C)$ is semi-stable, the proposition follows from  
Proposition \ref{P: semi-stable adm modif}.
\end{proof}

Let $b\in G(\breve{\Q}_p)$ be such that
$[b]\in B(G,\mu)$ is basic.
We have the weakly admissible locus \[\Gr_\mu^{wa}\subset \Gr_\mu,\] which is defined as the open Harder-Narasimhan stratum  $\Gr_\mu^{HN=[1]}$,
and the admissible locus \[\Gr_{\mu}^a\subset \Gr_\mu, \]which is defined as the open Newton stratum $\Gr_\mu^{Newt=[1]}$.
Recall that we have the inclusion of locally spatial diamonds over $\Spd\,\breve{E}$:
\[\Gr_\mu^a\subset \Gr_\mu^{wa}. \]
\begin{theorem}\label{T: fully HN general}
	Assume that $[b]\in B(G,\mu)$ is basic. Then we have the following equivalent statements:
	
	$B(G,\mu)$ is fully Hodge-Newton decomposable
	$\Longleftrightarrow \Gr_\mu^a=\Gr_\mu^{wa}$.
	
\end{theorem}
\begin{proof}
	Let $G_{ad}$ be the adjoint group attached to $G$. Then we get a natural surjective morphism $\phi: \Gr_{G,\mu}\ra \Gr_{G_{ad},\mu_{ad}}$. Let $[b_{ad}]\in B(G_{ad},\mu_{ad})$ be the corresponding element under the bijection $B(G,\mu)\st{\sim}{\ra} B(G_{ad},\mu_{ad})$. We consider the admissible locus and weakly admissible locus of $\Gr_{G_{ad},\mu_{ad}}$ with respective to $b_{ad}$.
	One checks easily that \[\Gr_{G,\mu}^a=\phi^{-1}(\Gr_{G_{ad},\mu_{ad}}^a)\quad \tr{and}\quad \Gr_{G,\mu}^{wa}=\phi^{-1}(\Gr_{G_{ad},\mu_{ad}}^{wa}).\] Thus we are reduced to the case $G$ is adjoint.
	We first assume that $G$ is quasi-split.
	\\
	
	\textbf{The direction ``$\Rightarrow$'':}
	The arguments are identical to the direction (1) $\Rightarrow$ (2) in \cite{CFS} Theorem 6.1, using 
	\begin{itemize}
		\item 
		the Newton stratification $\Gr_\mu=\coprod_{[b']\in B(G,0,\nu_b\mu^{-1})}\Gr_\mu^{Newt=[b']}$, 
		\item \cite{CFS} Corollary 4.16 and Lemma 4.11, 
		\item the above
		Lemma \ref{L:modif type}, 
		\item  \cite{CFS} Lemmas 6.3 and 6.4,
		\item the above Proposition \ref{P:weakly adm general}.
	\end{itemize}
	Since this is the easier direction, we leave the details to the reader.
	\\
	
	\textbf{The direction ``$\Leftarrow$'':}
	We follow the arguments in the direction $(2)\Rightarrow (1)$  of \cite{CFS} Theorem 6.1, except in the last step Proposition \ref{P:closure semi-infinite} will be used. For the reader's convenience, and to clarify the ideas, we provide the details as follows. 
	\\
	
	We prove that if $B(G,\mu)$ is not fully Hodge-Newton decomposable, then $\Gr_\mu^{wa}\supsetneq \Gr_\mu^a$, i.e. there exists a point $x\in \Gr_\mu^{wa}(C,\Ol_C)\setminus \Gr_\mu^a(C,\Ol_C)$, for any algebraically closed perfectoid field $C|\breve{E}$.
	\\
	
	By \cite{CFS} Corollary 4.16, $B(J_b, \mu^{-1})$ is not fully Hodge-Newton decomposable, and thus by \cite{CFS} Proposition 4.14 (and its proof), there exists $\alpha\in \Delta_0$ such that \[\lan -w_0\mu^\diamond, \tilde{\omega}_\alpha\ran >1,\] where $\tilde{\omega}_\alpha=\sum_{\gamma\in \Phi,\gamma|_A=\alpha}\omega_\gamma$.
	Let $\beta\in \Delta$ such that $\beta|_A=-w_0\alpha$ with corresponding coroot $\beta^\vee\in \Delta^\vee$. 
	Then $\lan \beta^\vee, \tilde{\omega}_{-w_0\alpha}\ran =\lan(\beta^\vee)^\diamond, \tilde{\omega}_{-w_0\alpha}\ran = 1$ and thus \[\lan \mu-\beta^\vee,\tilde{\omega}_{-w_0\alpha}\ran=\lan -w_0\mu^\diamond,  \tilde{\omega}_\alpha\ran -\lan \beta^\vee, \tilde{\omega}_{-w_0\alpha}\ran  >0.\]
	Let $M$ be the standard Levi subgroup such that $\Delta_{0,M}=\Delta_0\setminus\{-w_0\alpha\}$. Write $P$ the associated standard parabolic subgroup. Then the element $(-\beta^\vee)^\sharp\in\pi_1(G)_\Gamma$ admits to a lift to $\pi_1(M)_\Gamma$, which we still denote by
	\[(-\beta^\vee)^\sharp\in \pi_1(M)_\Gamma=\Big(X_\ast(T)/\lan \Phi_M^\vee\ran\Big)_\Gamma.\]
	Let $[b_M']\in B(M)_{basic}$ be the basic element in $B(M)$ such that it is mapped to $(-\beta^\vee)^\sharp$ under the bijection $\kappa_M: B(M)_{basic}\st{\sim}{\ra}\pi_1(M)_\Gamma$. Then $[\nu_{b'_M}]$ is $G$-antidominant.  Let $[b']\in B(G)$ be the image of $[b_M']$ under the natural map $B(M)\ra B(G)$.  Then  by construction \[[\nu_{b'}]=w_0[\nu_{b'_M}],\quad M_{b'}=M,\quad [b']\in B(G,0,\nu_b\mu^{-1})\] and the image of $[b']$ in $B( J_b,\mu^{-1})$ is not Hodge-Newton decomposable. As in \cite{CFS}, we may assume that $[b']\in B(G,0,\nu_b\mu^{-1})\setminus\{[1]\}$ is minimal.
	\\
	
	Consider the Newton stratum $Z:=\Gr_\mu^{Newt=[b']}$ attached to $[b']$. Then naturally $Z\bigcap \Gr_\mu^{wa}\subset \Gr_\mu^{wa}\setminus \Gr_\mu^a$. We claim that \[Z(C,\Ol_C)\bigcap \Gr_\mu^{wa}(C,\Ol_C)\neq \emptyset.\]This will conclude the proof of the direction ``$\Leftarrow$''. 
	\\
	
	Suppose the claim was not true, i.e. for any $x\in Z(C,\Ol_C)$, $x$ is not weakly admissible. By the definition of $Z$, we have \[\E_{b,x}\simeq \E_{b'}.\] By Proposition \ref{P:weakly adm general}, there exists a standard maximal parabolic $Q$ with the corresponding Levi $M_Q$, a reduction $b_{M_\Q}(x)$ of $b$ to $M_Q$, a character $\chi\in X_\ast(Q/Z_G)^+$ such that $\deg\chi_\ast(\E_{b,x})_Q>0$. Consider the map \[v: X^\ast(Q/Z_G)\ra \Z, \quad \chi'\mapsto \deg\chi_\ast'(\E_{b,x})_Q.\] It defines an element $v\in \Nc(G)$ and we have \[v\leq \nu(\E_{b,x})=-w_0[\nu_{b'}]\] by \cite{CFS} Theorem 1.8 (1). Then $-w_0v\leq [\nu_{b'}]$. By the description of $B(G,0,\nu_b\mu^{-1})$ in \cite{CFS} Corollary 4.4, we have $-w_0v\in B(G,0,\nu_b\mu^{-1})$. By the minimality of $[b']$, we get \[-w_0v=  [\nu_{b'}].\] By \cite{CFS} Theorem 1.8 (2) and the maximality of $Q$, we get $Q=P$ and $(\E_{b,x})_Q$ is the canonical reduction of $\E_{b,x}$. 
	
	Consider the decomposition $\Gr_{G,\mu}(C,\Ol_C)=\coprod_{\lambda\in S_M(\mu)}\Gr_{G,\mu,\lambda}(C,\Ol_C)$.
	Let $\lambda\in S_M(\mu)$ such that $x\in \Gr_{G,\mu,\lambda}(C,\Ol_C)$. By Lemma \ref{L:modif type} we have
	\[\E_{b_M'}=(\E_{b'})_P\times^PM=(\E_{b,x})_P\times^PM=\E_{b_M, \pr_\lambda(x)},\]
	where $b_M=b_{M_Q}(x)$ is the above reduction of $b$ to $M_Q=M$. 
	Therefore, by taking $-c_1^M(\cdot)$, we get
	\[\kappa_M(b_M')=\kappa_M(b_M)-\lambda^\sharp \in \pi_1(M)_\Gamma,\]
	which implies \[[\nu_{b'_M}]=[\nu_{b_M}]-\lambda^\sharp\otimes 1\in \pi_1(M)_{\Gamma,\Q}\] by our previous convention. As $\kappa_M(b_M')=(-\beta^\vee)^\sharp$ by construction, we get
	\begin{equation}\label{E:1}
	\lambda^\sharp\otimes 1=[\nu_{b_M}]+(\beta^\vee)^\sharp\otimes1\in \pi_1(M)_{\Gamma,\Q}.
	\end{equation}
	
	Next we pass to the dual side.
	Consider the inner form $J_b$ of $G$. Let $[b'']\in B(J_b)$ be the element which is mapped to $[b']\in B(G)$ under the bijection $B(J_b)\st{\sim}{\ra} B(G)$. Since $G$ is quasi-split and $b'$ admits reductions to $P$ and $M$ (by construction), the groups $P$ and $M$ transfer to parabolic and Levi subgroups respectively of $J_b$, which we still denote by $P$ and $M$ by abuse of notation. Moreover, there exist corresponding reductions $b_M''$ and $b_P''$ of $b''$ to $M$ and $P$ respectively. The isomorphism $J_{b,\breve{\Q}_p}\simeq G_{\breve{\Q}_p}$ induces an identification $\Gr_{J_b,\mu}=\Gr_{G,\mu}$, and by Theorem \ref{T: duality general} we have \[\Gr_{J_b,\mu}^{Newt=[b'']}=\Gr_{G,\mu}^{Newt=[b']}=Z.\]
	By Lemma 4.1, the bijection $B(J_b)\st{\sim}{\ra}B(G)$ restricts to a bijection $B(J_b,\mu^{-1})\st{\sim}{\ra} B(G,0,\nu_b\mu^{-1})$. As $[b']\in B(G,0,\nu_b\mu^{-1})$, we get $[b'']\in B(J_b,\mu^{-1})$.
	Consider the dual local Shtuka datum $(J_b,\{\mu^{-1}\},[b''])$. We have the following diagram
	\[\xymatrix{&\Sht(J_b,\mu^{-1}, b'')_\infty\ar@{->>}[ld]_{\pi_{dR}}\ar@{->>}[rd]^{\pi_{HT}}&\\
		\Gr_{J_b,\mu^{-1}}^a& &\Gr_{J_b, \mu}^{Newt=[b'']}.
	}\]
	
	Recall that we have our point $x\in Z(C,\Ol_C)=\Gr_{J_b, \mu}^{Newt=[b'']}(C,\Ol_C)$. Consider the subset
	\[\pi_{dR}(\pi_{HT}^{-1}(x))\subset \Gr_{J_b,\mu^{-1}}(C,\Ol_C).\]
	For the parabolic $P$ and Levi $M$ of $J_b$, we consider the digram of the corresponding $B_{dR}^+$-affine Grassmannians. Under the identifications $\Gr_{G,\mu}=\Gr_{J_b,\mu}$ and $S_M^G(\mu)=S_M^{J_b}(\mu)$, the decompositions $\Gr_{G,\mu}=\coprod_{\lambda\in S_M^G(\mu)}\Gr_{G,\mu,\lambda}$ and $\Gr_{J_b,\mu}=\coprod_{\lambda\in S_M^{J_b}(\mu)}\Gr_{J_b,\mu,\lambda}$ coincide. Therefore, we can view $x\in \Gr_{J_b,\mu,\lambda}(C,\Ol_C)$.
	
	We consider the side $\Gr_{J_b,\mu^{-1}}$.
	Let $z\in \pi_{dR}(\pi_{HT}^{-1}(x))\subset \Gr_{J_b,\mu^{-1}}(C,\Ol_C)$ be a point. Consider the decomposition of $\Gr_{J_b,\mu^{-1}}(C,\Ol_C)$ indexed by $S_M(\mu^{-1}):=S_M^{J_b}(\mu^{-1})$.
	Let $\lambda'\in S_M(\mu^{-1})$ be such that \[z\in \Gr_{J_b,\mu^{-1},\lambda'}(C,\Ol_C).\] By Lemma \ref{L:modif type} again, we have
	\[(\E_{b'',z})\times^PM\simeq \E_{b''_M, \pr_{\lambda'}(z)}.\]
	Let $\lambda_0:=-w_0\mu\in S_M(\mu^{-1})$ be the maximal element. If $\lambda'=\lambda_0$, that is 
	\[\lambda'=-w_0\mu\in X_\ast(T)^+_M\subset X_\ast(T),\]  then 
	we have
	\[\lambda'\otimes 1=(-w_0\mu)\otimes1\in \pi_1(M)_\Q=\Big(X_\ast(T)/\lan\Phi_M^\vee\ran\Big)_\Q.\]
	
	Now we come back to the group $G$ and consider $M$ as a Levi subgroup of $G$. We have our previous notation $\pi_1(M)_{\Gamma,\Q}$, taking into account the Galois action on $G_{\breve{\Q}_p}$ defined by $G$ over $\Q_p$.
	Then
	\begin{equation}\label{E:2}(\lambda')^\sharp\otimes 1=(-w_0\mu)^\sharp\otimes1\in \pi_1(M)_{\Gamma,\Q}.\end{equation}
	Recall that we have the corresponding element $\lambda\in S_M(\mu)$ such that $x\in \Gr_{J_b,\lambda}(C,\Ol_C)$.
	Then
	\[\begin{split}(\lambda')^\sharp\otimes 1&=-\lambda^\sharp\otimes 1\\&=-[\nu_{b_M}]-(\beta^\vee)^\sharp\otimes 1\in \pi_1(M)_{\Gamma,\Q},\end{split}\]
	where the second ``$=$'' comes from equation (\ref{E:1}).
	Combined with equation (\ref{E:2}), we get
	\begin{equation}\label{E:3}-[\nu_{b_M}]=(-w_0\mu)^\sharp\otimes1+(\beta^\vee)^\sharp\otimes 1\in \pi_1(M)_{\Gamma,\Q}.\end{equation}
	Pushing forward equation (\ref{E:3}) to $\pi_1(G)_{\Gamma,\Q}$ 
	and taking $\lan\cdot,\tilde{\omega}_{-w_0\alpha}\ran$, we get
	
	\[\lan [\nu_{b}],\tilde{\omega}_{-w_0\alpha}\ran=\lan \mu-\beta^\vee,\tilde{\omega}_{-w_0\alpha}\ran  >0. \]
	This is a contradiction, since $b$ is basic in $G$ and thus $\lan [\nu_{b}],\tilde{\omega}_{-w_0\alpha}\ran=0$.
	Therefore, for any $z\in \pi_{dR}(\pi_{HT}^{-1}(x))$ such that $z\in \Gr_{J_b,\mu^{-1},\lambda'}(C,\Ol_C)$, we have \[\lambda'\neq \lambda_0.\] 
	
	Now let $x\in Z(C,\Ol_C)$ vary. Since \[\Gr_{J_b,\mu^{-1}}^a(C,\Ol_C)=\coprod_{x\in Z(C,\Ol_C)}\pi_{dR}(\pi_{HT}^{-1}(x)) ,\] by the above discussion, we get
	\[\Gr_{J_b,\mu^{-1}}^a(C,\Ol_C)\bigcap \Gr_{J_b,\mu^{-1},\lambda_0}(C,\Ol_C)=\emptyset.\]
	As $\Gr_{J_b,\mu^{-1}}^a\subset \Gr_{J_b,\mu^{-1}}$ is open and $\Gr_{J_b,\mu^{-1},\lambda_0}\subset \Gr_{J_b,\mu^{-1}}$ is dense by Proposition \ref{P:closure semi-infinite}, we must have 
	\[\Gr_{J_b,\mu^{-1}}^a(C,\Ol_C)\bigcap\Gr_{J_b,\mu^{-1},\lambda_0}(C,\Ol_C)\neq\emptyset.\]
	This contradiction implies that the claim is true: $Z(C,\Ol_C)\bigcap \Gr_\mu^{wa}(C,\Ol_C)\neq \emptyset$. Thus we have proved the direction ``$\Leftarrow$''.
	\\
	
	\textbf{The general case:} now consider the case $G$ non necessarily quasi-split.
	Recall that we can assume that $G$ is adjoint. Let $H$ be a quasi-split inner form of $G$. Then $H$ is adjoint and $G=J_{b^\ast}$ for some $[b^\ast]\in B(H)_{basic}=H^1(\Q_p, H)$. Let $[b^H]\in B(H)$ be the image of $[b]$ under the bijection $B(G)\st{\sim}{\ra}B(H)$. We can consider the admissible locus and weakly admissible locus inside $\Gr_{H,\mu}$ with respective to $b^H$.
	Under the identification
	$\Gr_{G,\mu}=\Gr_{H,\mu}$, we have \[\Gr_{G,\mu}^a=\Gr_{H,\mu}^a \quad \tr{and} \quad \Gr_{G,\mu}^{wa}=\Gr_{H,\mu}^{wa}.\] Thus we are reduced to the quasi-split case as the last paragraph of the proof of \cite{CFS} Theorem 6.1.
	
\end{proof}

Come back to the Hodge-Tate side $\Gr_{\mu^{-1}}$. For any algebraically perfectoid field $C|\breve{E}$ and any $x\in \Gr_{\mu^{-1}}(C,\Ol_C)$, the inequality $\nu(\E_1,\E_{1,x},f)\leq \nu(\E_{1,x})$ (see subsection \ref{subsection Newt HN Hodge-Tate} and Proposition \ref{P: HN and Newt} (1))  implies that we have always the inclusion for open Newton and Harder-Narasimhan strata: \[\Gr_{\mu^{-1}}^{Newt=[b]}\subset\Gr_{\mu^{-1}}^{HN=[b]}.\]
Our previous efforts (cf. Theorems \ref{T: fully HN} and \ref{T: duality general}) imply the following enlarged version of Theorem \ref{T: fully HN general}:
\begin{corollary}\label{C: fully HN general}
	Let $[b]\in B(G,\mu)$ be basic. The following statements are equivalent:
	\begin{enumerate}
		\item $B(G,\mu)$ is fully Hodge-Newton decomposable,
		\item $\Gr_\mu^a=\Gr_\mu^{wa}$,
		\item  $\Gr_{\mu^{-1}}^{Newt=[b]}=\Gr_{\mu^{-1}}^{HN=[b]}$.
	\end{enumerate}
\end{corollary}
Of course, one can make the above corollary into a similar version as Theorem \ref{T: fully HN}, by including the corresponding information for the dual local Shtuka datum $(J_b, \{\mu^{-1}\}, [b^{-1}])$. One can also generalize the results further to $\Gr_{\leq\mu}$ and $\Gr_{\leq \mu^{-1}}$. We leave these tasks to the reader.


\section{Application to moduli of local $G$-Shtukas}\label{section loc Shtukas}
Let $(G,\{\mu\},[b])$ be a local Shtuka datum. Fix a representative $b\in G(\breve{\Q}_p)$ of $[b]$, and let $\Sht(G,\mu, b)_\infty$ be the associated moduli space of local $G$-Shtukas of type $\{\mu\}$ with infinite level.

Consider the Hodge-Tate period map of diamonds over $\breve{E}$
\[\pi_{HT}: \Sht(G,\mu,b)_\infty\lra \Gr_{\mu^{-1}}^{[b]}, \]where
we write $\Gr_{\mu^{-1}}^{[b]}=\Gr_{\mu^{-1}}^{Newt=[b]}$ for the associated Newton stratum inside $\Gr_{\mu^{-1}}$ for simplicity. By subsection \ref{subsection Newt HN Hodge-Tate}, the Harder-Narasimhan stratification on $\Gr_{\mu^{-1}}$ induces a Harder-Narasimhan stratification on $\Gr_{\mu^{-1}}^{[b]}$:
\[\Gr_{\mu^{-1}}^{[b]}=\coprod_{[b']\in B(G,\mu), [b']\leq [b]}\Gr_{\mu^{-1}}^{[b], HN=[b']} \]where each $\Gr_{\mu^{-1}}^{[b], HN=[b']}\subset \Gr_{\mu^{-1}}^{[b]}$ is the pullback of $\Gr_{\mu^{-1}}^{HN=[b']}\subset \Gr_{\mu^{-1}}$ under the inclusion $\Gr_{\mu^{-1}}^{[b]}\subset \Gr_{\mu^{-1}}$, which is empty if $[b']\geq [b]$ and $[b']\neq [b]$ (see subsection \ref{subsection Newt HN Hodge-Tate}).
The above stratification in turn induces a Harder-Narasimhan stratification on $\Sht(G,\mu,b)_\infty$ by diamonds
\[ \Sht(G,\mu,b)_\infty=\coprod_{[b']\in B(G,\mu),[b']\leq [b]}\Sht(G,\mu,b)^{HN=[b']}_\infty,\]where \[\Sht(G,\mu,b)^{HN=[b']}_\infty=\pi_{HT}^{-1}(\Gr_{\mu^{-1}}^{[b], HN=[b']}).\] By Corollary \ref{C: property HN strata general}, we have
\begin{corollary}\label{C: HN local shtuka}
	For any non basic $[b']\in B(G,\mu)$ such that $[b']\leq [b]$, the stratum $\Sht(G,\mu,b)^{HN=[b']}_\infty$ is a parabolic induction.
\end{corollary}
Of course, when $[b]=[b_0]$ is basic, the above Harder-Narasimhan stratification on $\Sht(G,\mu,b)_\infty$ is trivial and thus Corollary \ref{C: HN local shtuka} says nothing in this case.

We may also consider the Hodge-Tate period map of diamonds over $\breve{E}$
\[\pi_{HT}: \Sht(G,\leq\mu,b)_\infty\lra \Gr_{\leq\mu^{-1}}^{[b]}. \]
Then we have similar conclusion for $\Sht(G,\leq\mu,b)_\infty$ as above.

\section{Application to Shimura varieties}\label{section Shimura}
Let $(\mathbf{G},\mathbf{X})$ be \emph{an arbitrary Shimura datum}. Let $p$ be a prime number. Consider the conjugacy class of Hodge cocharacters $\{\mu\}$ attached to $\mathbf{X}$, which we view a conjugacy class of cocharacters over $\ov{\Q}_p$. Set $G=\mathbf{G}_{\Q_p}$.

Let $v|p$ be a place of the reflex field $\mathbf{E}=\mathbf{E}(\mathbf{G},\mathbf{X})$ above $p$  and $E=\mathbf{E}_v$. Let $\mathbf{K}\subset \mathbf{G}(\mathbb{A}_f)$ be a sufficiently small open compact subgroup. Attached to $(\mathbf{G},\mathbf{X},\mathbf{K})$, we have the Shimura variety
$\Sh_\K$ over the local reflex field $E$, which we view as an adic space. Assume that $\K$ is of the form $\K=KK^p$ with $K\subset G(\Q_p)$ and $K^p\subset\mathbf{G}(\A_f^p)$. Consider the $p$-adic flag variety $\Fl(G,\mu^{-1})$ over $E$, on which we have an action of $G(\Q_p)$. Let $G^c$ denote the quotient of $G$ by the maximal $\Q$-
anisotropic $\mathbb{R}$-split subtorus in the center $Z_G$ of $G$. Then we have an induced action $G^c(\Q_p)$ on $\Fl(G,\mu^{-1})$. Let $K^c\subset G^c(\Q_p)$ be the induced open compact subgroup.
The quotient space \[[\ul{K^c}\setminus \Fl(G,\mu^{-1})^\Diamond]\] exists as \emph{a small $v$-stack} in the sense of \cite{S}.
The main results of \cite{LZ}  imply that we have the Hodge-Tate period map
\[\pi_{HT}: \Sh_\K^\Diamond \lra [\ul{K^c}\setminus \Fl(G,\mu^{-1})^\Diamond],\]
which is a morphism of small $v$-stacks over $E$. More precisely, by \cite{LZ} Theorem 1.2 the universal $p$-adic local system over $\Sh_\K$ is de Rham, thus we get a relative Hodge-Tate filtration on it; by standard arguments as in \cite{CS} section 2, the type of this Hodge-Tate filtration is exactly given by $\{\mu^{-1}\}$.  


Noting that the $\ul{K^c}$-action on $\Fl(G,\mu^{-1})^\Diamond$ preserves the Harder-Narasimhan stratification \[\Fl(G,\mu^{-1})^\Diamond=\coprod_{[b]\in B(G,\mu)}\Fl(G,\mu^{-1})^{HN=[b],\Diamond},\]
we get a stratification on $\Sh_\K^\Diamond$ via $\pi_{HT}$:
\[ \Sh_\K^\Diamond=\coprod_{[b]\in B(G,\mu)}\Sh_\K^{HN=[b]},\]
where \[\Sh_\K^{HN=[b]}=\pi_{HT}^{-1}\Big(\big[\ul{K^c}\setminus \Fl(G,\mu^{-1})^{HN=[b],\Diamond}\big]\Big).\]
By Theorem \ref{T: property HN strata}, we have
\begin{corollary}\label{C: HN strata shimura}
	For any non basic $[b]\neq [b_0]$, the stratum $\Sh_\K^{HN=[b]}$ is a parabolic induction. 
\end{corollary}

Similarly, the $\ul{K^c}$-invariant Newton stratification \[\Fl(G,\mu^{-1})^\Diamond=\coprod_{[b]\in B(G,\mu)}\Fl(G,\mu^{-1})^{Newt=[b],\Diamond}\] also induces a stratification\footnote{Let us call it the Newton stratification of $\Sh_\K^\Diamond$. If there exists some nice integral model of $\Sh_\K$, then we have the Newton stratification on the special fiber and we can pullback it to the tube (the good reduction locus) over the special fiber. There exist some subtleties between the two Newton stratifications, cf. \cite{CS}.} on $\Sh_\K^\Diamond$:
\[ \Sh_\K^\Diamond=\coprod_{[b]\in B(G,\mu)}\Sh_\K^{Newt=[b]},\]
where \[\Sh_\K^{Newt=[b]}=\pi_{HT}^{-1}\Big(\big[\ul{K^c}\setminus \Fl(G,\mu^{-1})^{Newt=[b],\Diamond}\big]\Big).\]
Then we have an inclusion of open strata
\[ \Sh_\K^{Newt=[b_0]}\subset \Sh_\K^{HN=[b_0]}.\]
Theorem \ref{T: fully HN} implies
\begin{corollary}\label{C: full HN shimura}
	If the pair $(G,\{\mu\})$ is fully Hodge-Newton decomposable, then $\Sh_\K^{Newt=[b_0]}=\Sh_\K^{HN=[b_0]}$. 
\end{corollary}
Corollaries \ref{C: HN strata shimura} and \ref{C: full HN shimura} together imply that for fully Hodge-Newton decomposable Shimura varieties, the supercuspidal part of their cohomology concentrates on the (part contributed by) basic Newton strata. We will work out the details in a future work.
We refer the reader to \cite{F2} 9.7.2 for more speculations on cohomological applications related to the results above.
\begin{remark}
The readers who prefer diamonds can replace the above by the following considerations.
Let \[\Sh_{K^p}=\varprojlim_{K_p}\Sh_\K^\Diamond\] be the diamond of Shimura variety with infinite level at $p$ and prime-to-$p$ level $K^p$, on which we have a natural action of $G(\Q_p)$.
Then we get the 
Hodge-Tate period map of diamonds\footnote{If $(\mathbf{G},\mathbf{X})$ is of abelian type, then $\Sh_{K^p}$ is representable by a perfectoid space and $\pi_{HT}$ comes from a morphism of adic spaces over $E$, cf. \cite{Sh}. In the general case, see also \cite{Hansen1} for the morphisms of diamonds $\pi_{HT}$.} over $E$
\[\pi_{HT}: \Sh_{K^p} \lra \Fl(G,\mu^{-1})^\Diamond,\]which is $G(\Q_p)$-equivariant.
We can define Harder-Narasimhan strata and Newton strata on the diamond $\Sh_{K^p}$, 
which are inverse limits of the corresponding strata at finite levels. Corollaries \ref{C: HN strata shimura} and \ref{C: full HN shimura} admit the corresponding diamond versions.
\end{remark}

\end{document}